\documentclass[10pt,regno]{amsart}
\usepackage{amsmath,amsthm,amsfonts,amssymb,amscd,overpic,mathrsfs}
\usepackage{euscript,graphicx,color}
\usepackage{epstopdf,rotating}

\usepackage{chngcntr}
\usepackage{apptools}
\AtAppendix{\counterwithin{Appproposition}{section}}

\AtAppendix{\counterwithin{Applemma}{section}}

\AtAppendix{\counterwithin{Appclaim}{section}}

\newtheorem{mtheo}{Theorem}
\newtheorem{mcoro}{Corollary}[mtheo]

\newtheorem{mtheorem}{Theorem}
\newtheorem{mcorollary}[mtheorem]{Corollary}

\newtheorem{theorem}{Theorem}[section]
\newtheorem{lemma}[theorem]{Lemma}
\newtheorem{claim}[theorem]{Claim}

\newtheorem{corollary}[theorem]{Corollary}
\newtheorem{proposition}[theorem]{Proposition}
\newtheorem*{conjecture}{Conjecture}

\theoremstyle{definition}

\newtheorem{remark}[theorem]{Remark}
\newtheorem{caveat}[theorem]{Caveat}
\newtheorem*{remark*}{Remark}

\numberwithin{equation}{section}

\newcommand{\eqdef}{\stackrel{\scriptscriptstyle\rm def}{=}}

\DeclareMathOperator{\clocon}{\overline{\rm conv}}

\DeclareMathOperator{\trace}{trace}

\DeclareMathOperator{\card}{card}

\DeclareMathOperator{\interior}{int}
\DeclareMathOperator{\Leb}{Leb}
\DeclareMathOperator{\sgn}{sgn}
\DeclareMathOperator{\Lip}{Lip}

\def\varkappa{\kappa}

\def\bN{\mathbb{N}}

\def\bZ{\mathbb{Z}}
\def\bP{\mathbb{P}}
\def\bR{\mathbb{R}}

\def\bQ{\mathbb{Q}}
\def\bS{\mathbb{S}}

\def\cM{\EuScript{M}}
\def\cN{\EuScript{N}}
\def\cE{\EuScript{E}}
\def\cL{\EuScript{L}}
\def\cO{\EuScript{O}}
\def\cP{\EuScript{P}}

\def\fE{\mathfrak{E}}
\def\fH{\mathfrak{H}}
\def\fI{\mathfrak{I}}
\def\fm{\Leb}

\DeclareMathSymbol{\varnothing}{\mathord}{AMSb}{"3F}
\renewcommand{\emptyset}{\varnothing}

\begin{document}

\title[Entropy spectrum of Lyapunov exponents]{Entropy spectrum of Lyapunov exponents\\
for
nonhyperbolic step skew-products\\and elliptic cocycles}
\author[L.~J.~D\'iaz]{L. J. D\'\i az}
\address{Departamento de Matem\'atica PUC-Rio, Marqu\^es de S\~ao Vicente 225, G\'avea, Rio de Janeiro 22451-900, Brazil}
\email{lodiaz@mat.puc-rio.br}
\author[K.~Gelfert]{K.~Gelfert}
\address{Instituto de Matem\'atica Universidade Federal do Rio de Janeiro, Av. Athos da Silveira Ramos 149, Cidade Universit\'aria - Ilha do Fund\~ao, Rio de Janeiro 21945-909,  Brazil}\email{gelfert@im.ufrj.br}
\author[M.~Rams]{M. Rams} \address{Institute of Mathematics, Polish Academy of Sciences, ul. \'{S}niadeckich 8,  00-656 Warszawa, Poland}
\email{rams@impan.pl}

\begin{abstract}
	We study the fiber Lyapunov exponents of step skew-product maps over a complete shift of $N$, $N\ge2$, symbols and with $C^1$ diffeomorphisms of the circle as fiber maps. The systems we study are transitive and genuinely nonhyperbolic, exhibiting simultaneously ergodic measures with positive, negative, and zero exponents. Examples of such systems arise from the projective action of $2\times 2$ matrix cocycles and our results apply to an open and dense subset of elliptic $\mathrm{SL}(2,\bR)$ cocycles. We derive a multifractal analysis for the topological entropy of the level sets of  Lyapunov exponent. The results are formulated in terms of Legendre-Fenchel transforms of restricted variational pressures, considering hyperbolic ergodic measures only, as well as in terms of restricted variational principles of entropies of ergodic measures with a given exponent. We show that the entropy of the level sets is a continuous function of the Lyapunov exponent. The level set of the zero exponent has positive, but not maximal, topological entropy. Under the additional assumption of proximality, as for example for skew-products arising from certain matrix cocycles, there exist two unique ergodic  measures of maximal entropy, one with negative and one with positive fiber Lyapunov exponent.
\end{abstract}

\begin{thanks}{This research has been supported  [in part]  by CNE-Faperj, CNPq-grants (Brazil), EU Marie-Curie IRSES ``Brazilian-European partnership in Dynamical
Systems" (FP7-PEOPLE-2012-IRSES 318999 BREUDS), and National Science Centre grant 2014/13/B/ST1/01033 (Poland). The authors acknowledge the hospitality of IMPAN, IM-UFRJ, and PUC-Rio and thank Anton Gorodetski, Yali Liang, and Silvius Klein for their comments. }\end{thanks}

\keywords{entropy,
ergodic measures,
Legendre-Fenchel  transform,
Lyapunov exponents,
pressure,
restricted variational principles,
skew-product,
transitivity}
\subjclass[2000]{%
37D25, 
37D35, 
37D30, 
28D20, 
28D99
}

\maketitle
\tableofcontents

\section{Introduction}

We will study the entropy spectrum of Lyapunov exponents, that is, the topological entropy of level sets of points with a common given Lyapunov exponent. This subject forms part of the multifractal analysis which, in general, studies thermodynamical quantities and objects (such as, for example, equilibrium states, entropies, Lyapunov exponents, Birkhoff averages, and recurrence rates) and their relations with geometrical properties (for example, fractal dimensions). Those properties are often encoded, and we follow this approach, by the topological pressure and its Legendre-Fenchel transform. The novelty of this paper is that we consider transitive systems which are genuinely nonhyperbolic in the sense that their Lyapunov spectra contain zero in its interior (this property continues to hold also for perturbations) and that we provide a description of the full spectrum.

The systems that we investigate are step skew-products with circle fibers.  These systems provide quite easily describable examples in which (robust) nonhyperbolicity can be studied. At the same time,  they serve as models for robustly transitive  and (nonhyperbolic)  partially hyperbolic  diffeomorphisms in a setting motivated by~\cite{BonDiaUre:02,RodRodTahUre:12}, for further details see~\cite[Section 8.3]{DiaGelRam:}. Let us observe that they also appear quite naturally as limit systems (using the terminology in~\cite{GorIly:99}) in some non-local bifurcations and fit into the theory rigorously initiated in~\cite{GorIly:00}. From another point of view,  they can also be  considered as actions of a group of diffeomorphisms on the circle or as random dynamical systems~\cite{Nav:11}. An important class of examples that fit into our setting is the one of step skew-products on the circle which are induced by the projective action of a linear cocycle of $2\times 2$-matrices. Indeed, there are a kind of paradigmatic examples which admit a fairly simple description where our results can be applied (for the complete general setting and precise results  see Section~\ref{sec:results}).

In Sections~\ref{subsec:stepskew} and \ref{subsec:cocyles} we will skip all major technicalities and announce in a quick way our main result and its application to the study of cocycles of matrices in $\mathrm{SL}(2,\mathbb{R})$, while
in Section~\ref{sec:results} we announce our results in their full generality.
We point out that we always work in the lowest possible regularity and consider $C^1$ circle diffeomorphisms as fiber maps.

\subsection{Step skew-products with circle fibers} \label{subsec:stepskew}

Consider a finite family $f_i\colon \bS^1\to \bS^1$, $i=0,\ldots,N-1$ for $N\ge2$, of $C^1$ diffeomorphisms and the associated  step skew-product
\begin{equation}\label{eq:sp}
	F\colon \Sigma_N\times \bS^1\to \Sigma_N\times \bS^1,
	\quad
	F(\xi,x) =(\sigma(\xi), f_{\xi_0}(x)),
\end{equation}
where $\Sigma_N=\{0,\ldots,N-1\}^\bZ$.
We  consider the class $\mathrm{SP}^1_{\rm shyp}(\Sigma_N\times\bS^1)$ of such maps which are topologically transitive and ``nonhyperbolic in a nontrivial way" in the sense that there are some ``expanding region'' and some ``contracting region'' (relative to the fiber direction) and that any of those can be reached from anywhere in the ambient space under forward/backward iterations as follows:
\begin{itemize}
\item Some hyperbolicity: There is  a ``forward blending" interval $J^+\subset\bS^1$ such that for every sufficiently small interval $H$ with $H\cap J^+\ne \emptyset$ there are $\ell\sim |\log\lvert H\rvert | $ and  a finite sequence $(\xi_0\ldots\xi_{\ell})$ such that $f_{\xi_{\ell}}\circ\ldots\circ f_{\xi_0} (H)$ covers $J^+$ in an uniformly expanding way.
Similarly, there is a ``backward expanding blending" interval $J^-$.
\item Transitions in finite time to/from blending intervals: There exists $M\ge1$ such that for every $x\in\bS^1$ there are finite sequences $(\theta_{-r}\ldots\theta_{-1})$ and $(\beta_0\ldots\beta_s)$, $s,r\le M$, such that $f_{\beta_{s}}\circ\ldots\circ f_{\beta_0}(x)\in J^+$ and $f_{\theta_{-r}}^{-1}\circ\ldots\circ f_{\theta_{-1}}^{-1}(x)\in J^+$. Similarly, there are transitions to/from the  ``backward  blending" interval $J^-$.
\end{itemize}

\begin{remark}\label{r.Gorod}
The simplest setting where the two properties above can be verified are skew-product maps defined on $\Sigma_2\times \mathbb{S}^1$ whose fiber maps  and  are a Morse-Smale diffeomorphism $f_0$ with one attracting and one repelling fixed point and an irrational rotation $f_1$. Moreover, small perturbations of these maps also satisfy the hypotheses, see \cite[Proposition 8.8]{DiaGelRam:}.  Indeed, this type of example plays a specially important role in this paper as they satisfy the following property of \emph{proximality}:
\begin{itemize}
\item Proximality: For every $x,y\in\bS^1$ there exists one bi-infinite sequence $\xi\in\Sigma_N$ such that $\lvert f_{\xi_{n}}\circ\ldots\circ f_{\xi_0}(x)-f_{\xi_n}\circ\ldots\circ f_{\xi_0}(y)\rvert\to0$ and $\lvert f_{\xi_{-n}}^{-1}\circ\ldots\circ f_{\xi_{-1}}^{-1}(x)-f_{\xi_{-n}}^{-1}\circ\ldots\circ f_{\xi_{-1}}^{-1}(y)\rvert\to0$ as $n\to0$
\end{itemize}
Further, examples of a quite different nature can be found in \cite[Section 8.1]{DiaGelRam:}, where also the motivation for the term ``blending" is  discussed. Let us observe that the above properties hold open and densely among $C^1$ transitive nonhyperbolic step skew-products (see~\cite[Proposition 8.9]{DiaGelRam:} for details).
\end{remark}

Given $X=(\xi,x)\in \Sigma_N\times\bS^1$, consider the \emph{(fiber) Lyapunov exponent} of $X$
\begin{equation}\label{def:exponentFfirst}
	\chi(X)
	\eqdef
	 \lim_{n\to\pm\infty}\frac{1}{ n}\log\,\lvert (f_{\xi}^n)'(x)\lvert  ,
\end{equation}
(where $ f_\xi^{-n}\eqdef f_{\xi_{-n}}\circ\ldots \circ f_{-1}$ and $f_\xi^n \eqdef f_{\xi_{n-1}}\circ\ldots\circ f_{\xi_0}$) 
where we assume that both limits $n\to\pm \infty$  exist and coincide. 
We will analyze the topological entropy of the following \emph{level sets of Lyapunov exponents}: given $\alpha\in\bR$ let
\[
	\cL(\alpha)
	\eqdef \big\{X\in\Sigma_N\times\bS^1\colon \chi(X)=\alpha\big\}.
\]
 Here we will rely on the general concept of topological entropy $h_{\rm top}$ introduced by Bowen~\cite{Bow:73} (see Appendix~\ref{App:B}).
 Given an $F$-ergodic measure $\mu$, denote by $\chi(\mu)$ the Lyapunov exponent of $\mu$  defined by 
 \[
 	\chi(\mu)
	\eqdef \int\log\,\lvert(f_{\xi_0})'(x)\rvert\,d\mu(\xi,x).
\]	

\begin{mtheo}\label{teo:specialste}
	For every $N\ge2$ and for every $F\in \mathrm{SP}^1_{\rm shyp}(\Sigma_N\times\bS^1)$  we have $\cL(\alpha)\ne\emptyset$ if and only if $\alpha\in[\alpha_{\rm min},\alpha_{\rm max}]$ for some numbers $\alpha_{\rm min}<0<\alpha_{\rm max}$. Moreover, the map $\alpha\mapsto h_{\rm top}(\cL(\alpha))$ is continuous and concave on each interval $[\alpha_{\rm min},0]$ and $[0,\alpha_{\rm max}]$ and for all $\alpha\in[\alpha_{\rm min},0)\cup(0,\alpha_{\rm max}]$
we have
\[
	h_{\rm top}(\cL(\alpha))
	= \sup\{h(\mu)\colon \mu\in\cM_{\rm erg}(\Sigma_N\times\bS^1),\chi(\mu)=\alpha\}. 
\]	
Moreover, assuming proximality, there exist unique ergodic $F$-invariant probability measures $\mu_-$ and $\mu_+$ of maximal entropy $h(\mu_\pm)=\log N$, respectively, and satisfying
\[
	\alpha_{\rm min}<\alpha_-\eqdef \chi(\mu_-)<0<\alpha_ +\eqdef\chi(\mu_+)<\alpha_{\rm max}
\]
and for all $\alpha \in(\alpha_{\rm min},\alpha_{\rm max})\setminus\{ \alpha_-, \alpha_+\}$ we have
\[
	0< h_{\rm top}(\cL(\alpha)) < \log N.
\]
\end{mtheo}

Theorem~\ref{teo:specialste} will be a consequence of the more elaborate version stated in Theorems~\ref{main1},~\ref{main2}, and~\ref{main3}, see Section~\ref{sec:results} and compare Figure~\ref{fig.0}.
\begin{figure}[h] 
 \begin{overpic}[scale=.40]{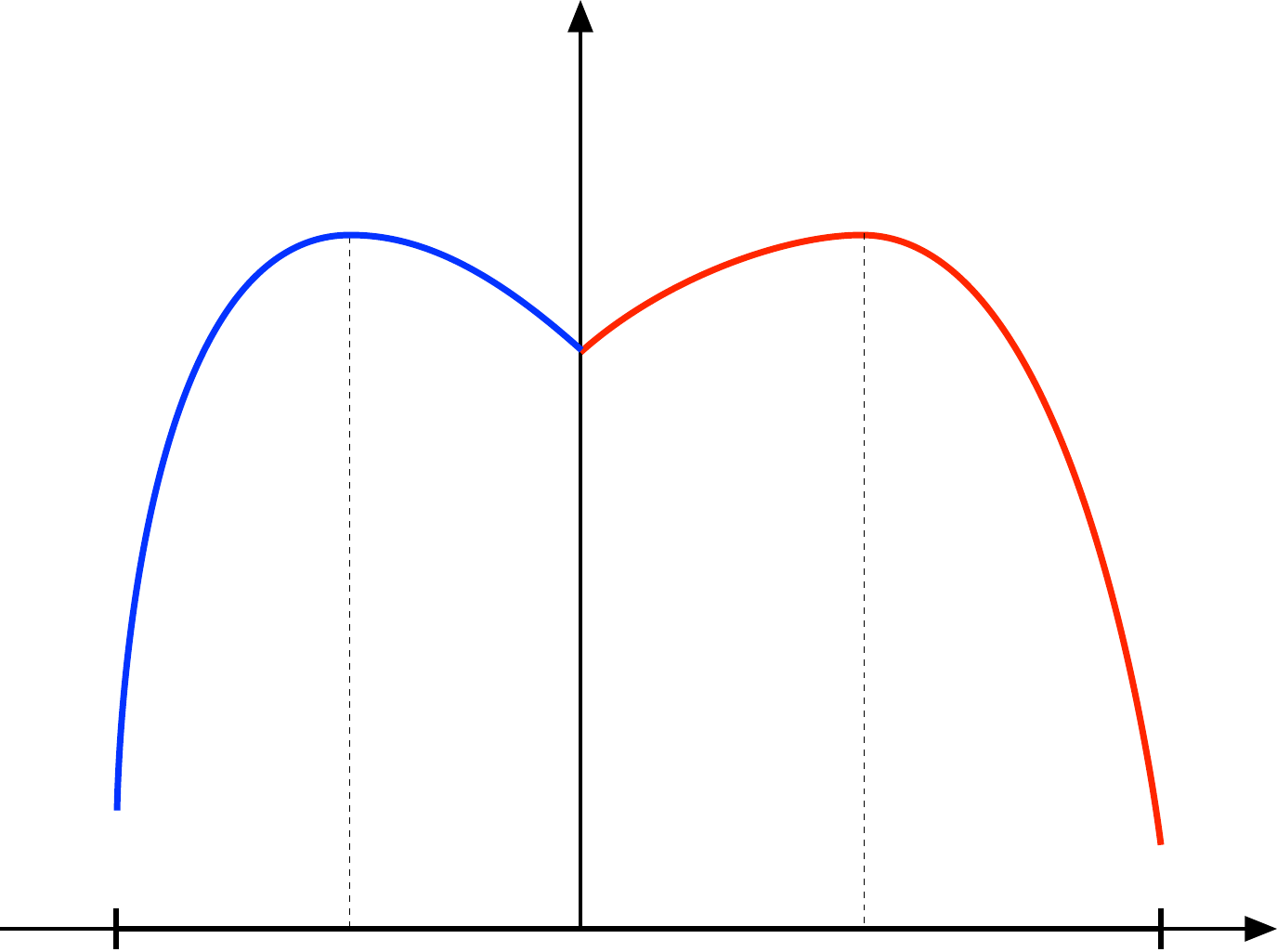}
 \put(48,68){\small entropy$(\cL(\alpha))$}
 	\put(102,0){\small $\alpha$}
	\put(44,-5){\small $0$}
 	\put(5,-5){\small $\alpha_{\rm min}$}
 	\put(82,-5){\small $\alpha_{\rm max}$}
 	\end{overpic}
 \caption{The entropy spectrum, assuming proximality}
 \label{fig.0}
\end{figure}
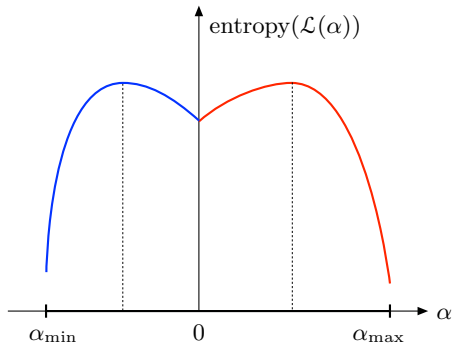

\subsection{Application to $\mathrm{GL}^+(2,\bR)$ and $\mathrm{SL}(2,\bR)$ cocycles} \label{subsec:cocyles}

Consider first the group $\mathrm{GL}^+(2, \mathbb{R})$ of all $2\times 2$ matrices with real coefficients and positive determinant. Given $N\ge2$, a continuous map $A\colon\Sigma_N \to \mathrm{GL}^+(2, \mathbb{R})$ is called a $2\times2$
\emph{matrix cocycle}. If $A$ is piecewise constant and depends only on the zeroth coordinate of the sequences $\xi \in \Sigma_N$, that is $A(\xi) = A_{\xi_0}$ where $\mathbf{A}\eqdef\{A_0,\ldots,A_{N-1}\}\in\mathrm{GL}^+(2, \mathbb{R})^N$, then we refer to it as the \emph{one-step cocycle generated} by $\mathbf{A}$ or simply as the \emph{one-step cocycle} $\mathbf{A}$. 
One-step matrix cocycles are an object of intensive study in several branches of mathematics and also have serious physical applications. See for example \cite[Sections 2 and 3]{Dam:}, \cite[Section 7.2]{DuaKle:16}, and~\cite{Via:14}. 

Note that the projective line $\mathbb{P}^1$ is topologically the circle $\bS^1$
and the action of any $\mathrm{GL}^+(2, \mathbb{R})$ matrix on $\mathbb{P}^1$ is a diffeomorphism. We will continue to take this point of view, given a  matrix $A\in \mathrm{GL}^+(2,\bR)$, define $f_A\colon\bP^1\to\bP^1$ by
\begin{equation}\label{eq:defsteskecoc}
	f_A(v)
	\eqdef \frac{Av}{\lVert Av\rVert}.
\end{equation}
Given a one-step $2\times 2$ matrix cocycle $\mathbf A$, we denote by $F_{\mathbf A}$ the associated step skew-product generated by the family of maps $f_{A_0},\ldots,f_{A_{N-1}}$  as in~\eqref{eq:sp}.

To simplify our study of the  Lyapunov exponents of the cocycle,  we consider the \emph{one-sided} one-step cocycle $A\colon\Sigma_N^+ \to \mathrm{GL}^+(2, \mathbb{R})$, where $\Sigma_N^+=\{0,\ldots,N-1\}^{\bN_0}$.
We denote 
$$
	\mathbf A^n(\xi^+) 
	\eqdef A_{\xi_{n-1}} \circ \ldots \circ A_{\xi_{1}} \circ A_{\xi_0}, 
	\quad \xi^+\in \Sigma_N^+, \, n\ge 0.
 $$
The \emph{Lyapunov exponents of the cocycle $\mathbf A$} at $\xi^+ \in \Sigma_N^+$ are the
limits
$$
	\lambda_1 (\mathbf A,\xi^+)
	\eqdef \lim_{n\to \infty} \frac{1}{n}\log \,\lVert \mathbf A^n(\xi^+)\rVert
	\,\mbox{ and }\,
	\lambda_2 (\mathbf A,\xi^+)
	\eqdef \lim_{n\to \infty} \frac{1}{n}\log\, \lVert (\mathbf A^n(\xi^+))^{-1}\rVert^{-1},
$$
where $\lVert L\rVert$ denotes the norm of the matrix $L$, whenever they exist. Given $\alpha\in\bR$, we consider the level set
\begin{equation}\label{def:levelsetA}
	\cL^+_{\mathbf A}(\alpha)
	\eqdef \big\{\xi^+\in\Sigma_N^+\colon \lambda_1(\mathbf A,\xi^+)=\alpha\big\}.
\end{equation}

We now consider the subgroup $\mathrm{SL}(2, \mathbb{R})\subset \mathrm{GL}^+(2,\bR)$ of $2\times 2$ matrices with real coefficients and determinant one.
The space $\mathrm{SL}(2, \mathbb{R})^N$ can be roughly divided into two subsets:
{\emph{elliptic}} and {\emph{uniformly hyperbolic}} ones (denoted by $\fE_N$ and $\fH_N$, respectively). 
 Both $\fE_N$ and $\fH_N$ are open and their union is dense in $\mathrm{SL}(2, \mathbb{R})^N$, see \cite[Proposition 6]{Yoc:04}.
Hyperbolic cocycles are quite well understood, for their characterization see \cite{AviBocYoc:10}.
 Though much less is known about elliptic cocycles. Here we will introduce a subset of elliptic cocycles having ``some hyperbolicity'', denoted by $\fE_{N,{\rm shyp}}$, which forms an open and dense subset of $\fE_N$. For any $\mathbf A\in\fE_{N,{\rm shyp}}$ we will provide a detailed description of the spectrum of its Lyapunov exponents.

To be more precise,  denote by $\langle\mathbf A\rangle$ the semigroup generated by $\mathbf A$.
Recall that an element $R\in\mathrm{SL}(2,\bR)$ is
\emph{elliptic} if  the absolute value of its trace is strictly less than $2$; in such a case the matrix $R$ is conjugate to a rotation by some angle, called its \emph{rotation number} and denoted by $\varrho(R)$. An element $A\in\mathrm{SL}(2,\bR)$ is \emph{hyperbolic} if the absolute value of its trace is strictly larger than $2$, which is equivalent to the fact that the matrix $A$ has one eigenvalue bigger than one and one smaller than one.
The set $\fE_N$ of \emph{elliptic cocycles} is the set of cocycles $\mathbf A\in\mathrm{SL}(2,\bR)^N$ such that $\langle \mathbf A\rangle$ contains an elliptic element.

If a matrix $R$ is elliptic then $f_R$ is  conjugate to a rotation by angle $\varrho(R)$.  Note that in the case when $\varrho(R)$ is irrational then $f_R$ is of the same type as (differentiably conjugate to) the map $f_1$ in Remark~\ref{r.Gorod}.
If a matrix $A$ is hyperbolic then $f_A$ has one attracting and one repelling fixed point. Note that then $f_A$ is a very specific case of a Morse-Smale diffeomorphism of $\bP^1$ of the same type as $f_0$ in Remark~\ref{r.Gorod}.

We consider a subset $\fE_{N,{\rm shyp}}$  of $\fE_N$
consisting of cocycles $\mathbf{A}\in\mathrm{SL}(2,\bR)^N$ which have ``some hyperbolicity'' in the following sense (the precise definition is provided in Appendix~\ref{App:A}):
\begin{itemize}
\item Some hyperbolicity:
There exists $A\in\langle\mathbf A\rangle$ which is hyperbolic.
\item Transitions in finite time: 
There exists $B\in\langle\mathbf A\rangle$ which is sufficiently ``close'' to an irrational rotation.
\end{itemize}
Note that these properties are just a translation of the properties of systems in $\mathrm{SP}^1_{\rm shyp}(\Sigma_N\times\bS^1)$  for the induced fiber maps $f_A$ and $f_B$ arising from matrix cocycles in the spirit of Remark \ref{r.Gorod}. Indeed, it is easy to check that each such a cocycle automatically also satisfies the property Proximality.
Following \cite{AviBocYoc:10}, we will show that $\fE_{N,\rm shyp}$  is an open and dense subset of  $\fE_N$ (see Proposition~\ref{p.elliptic}).

Given $\nu$ an ergodic measure on $\Sigma_N^+$ (with respect to $\sigma^+\colon\Sigma_N^+\to\Sigma_N^+$), denote 
\[
	\lambda_1(\mathbf A,\nu)
	\eqdef \lim_{n\to\infty}\int\frac1n\log\,\lVert \mathbf A^n(\xi^+)\rVert\,d\nu
\]
and note that this number is the Lyapunov exponent $\lambda_1(\mathbf A,\xi^+)$ for any $\nu$-typical $\xi^+$. Denote by $h(\nu)$ the metric entropy of $\nu$.

\begin{mtheo}\label{teo:introSL2R}
For every $N\ge2$ the set $\fE_{N,\rm shyp}$  is open and dense  in $\fE_N$ and has the following property: 
 For every $\mathbf{A}$ in $\fE_{N,\rm shyp}$ there are numbers $0<\alpha_+<\alpha_{\rm max}$ such that the map $\alpha\mapsto h_{\rm top}(\cL^+_{\mathbf A}(\alpha))$ is continuous and concave on $[0,\alpha_{\rm max}]$, having a unique maximum at $\alpha^+$  and 
 \[
	 h_{\rm top}(\cL^+_{\mathbf A}(\alpha_+)) 
	= \log N,
\]
we have $0<h_{\rm top}(\cL^+_{\mathbf A}(0))<\log N$, and for every $\alpha\in(0,\alpha_{\rm max}]$  we have
$$
	h_{\rm top}(\cL^+_{\mathbf A}(\alpha) )
	=\sup\{h(\nu)\colon\nu\in\cM_{\rm erg}(\Sigma_N^+),\lambda_1(\mathbf A,\nu)=\alpha\}.
$$
\end{mtheo}

A fundamental step to prove the above theorem is to study the relations between the Lyapunov exponents of the cocycle and the ones of the associated step skew-product. With these results at hand, we can invoke the results about skew-products and prove Theorem~\ref{teo:introSL2R}. 
More precisely, by Remark~\ref{rem:example}, for every $\mathbf A\in \fE_{N,\rm shyp}$ the associated step skew-product $F_{\mathbf A}$ satisfies the hypothesis of Theorems~\ref{main2} and~\ref{main3}.  Theorem~\ref{teo:SL2Rskewproduct} then translates the Lyapunov spectrum of the skew-product to the one of the cocycle and hence proves Theorem~\ref{teo:introSL2R}.

Note that the existence of a zero  Lyapunov exponent for $\mathrm{SL}^+(2,\bR)$ cocycles  immediately ``translates" to the condition  of having two equal exponents for $\mathrm{GL}^+(2,\bR)$ cocycles, just by considering the normalization $A\mapsto A/\sqrt{\lvert\det(A)\rvert}$. Thus, for such cocycles the above result can be read as follows. 

\begin{mcoro}\label{cor-teo:introSL2R}
	For every $N\ge2$ there is an open and dense subset $\mathcal S\subset\mathrm{GL}^+(2,\bR)^N$ such that for every $\mathbf A\in\mathcal S$ we have 
\[
	0
	< h_{\rm top}(\{\xi^+\in\Sigma_N^+
		\colon\lambda_1(\mathbf A,\xi^+)=\lambda_2(\mathbf A,\xi^+)\})
	<\log N.
\]	
\end{mcoro}

The left inequality in Corollary~\ref{cor-teo:introSL2R} also follows from~\cite[Theorem 3]{BocRam:16}, where a different approach is used.

The study of level sets of \emph{Lyapunov exponents} within the context of cocycles fits within the analysis of the simplicity of the Lyapunov spectrum. There is an intensive line of research, perhaps initiated by \cite{Kni:91}, where a measure in the base space is fixed and varying the cocycle one aims to establish conditions which guarantee that the integrated top Lyapunov is (or is not) positive, see for instance \cite{Avi:11,AviVia:07}. In a slightly different context, the designated measure in the base is the Riemannian volume, the fiber dynamics is the derivative cocycle of a volume preserving diffeomorphism, where more precise results about the spectrum of Lyapunov exponents are obtained (a dichotomy between all exponents being  equal  to zero versus the existence of a dominated Oseledets splitting), see
for instance \cite{BocVia:02,Boc:02}.

In contrast to these works, here our cocycle is fixed (within an open and dense set of elliptic cocycles), and \emph{a priori} no base measure is designated, and we study the orbitwise Lyapunov exponents. Our measurement of the level sets  will be in terms of topological  entropy. For that we will establish restricted variational principles and
develop a multifractal analysis in a nonhyperbolic setting, which we will now discuss. 

\subsection{Multifractal context}\label{subsec:multi}
For  uniformly hyperbolic dynamics multifractal analysis is understood in great depth and
has found already far reaching applications.
There is a enormous literature on this subject. To highlight a collection of results in the field at different stages of development, we refer, for example, to
\cite{Rue:78} (analyticity of pressure and its consequences),
\cite{Ols:95,PesWei:97} (multifractal analysis for conformal expanding maps and Smale's horsehoes), and
\cite{BarSau:01} (mixed spectra and restricted variational principles).
In many of those references, particular attention is drawn to the so-called geometric potentials because of their close relation to Lyapunov exponents, entropy, and Sinai-Ruelle-Bowen measures.
One key property of uniformly hyperbolic systems, under which the classical context of multifractal analysis was developed so far, is the specification property (studied for example in~\cite{TakVer:03,PfiSul:07,FanLiaPey:08}). Note that it is also essential (compare~\cite{Bow:75b}) to guarantee the uniqueness of equilibrium states which is another key property to study multifractal analysis. The specification property implies many further strong properties, for example that the set of all invariant probability measures is a Poulsen simplex (\cite{Sig:74}) with a hence very rich topological structure.

The multifractal analysis theory extends also to ``one-sided" nonuniformly hyperbolic systems, that is, for example to nonuniformly expanding maps where the presence of a nonpositive Lyapunov exponent is the only obstruction to hyperbolicity, that is, the spectrum of Lyapunov exponents covers a range of hyperbolicity and the zero exponent bounds this range from one side, see for example \cite{GelPrzRam:10} (expansive Markov maps of the interval) and
\cite{PrzRiv:,IomTod:11} (multimodal interval maps).
So far, there is not much understanding of a multifractal analysis for more complicated types of nonhyperbolic systems. It is difficult to describe all the situations that can happen in general; one natural class of systems to focus on could be the systems with a designated line field (associated with the Oseledets decomposition) for which the Lyapunov exponent takes both positive and negative values arbitrarily close to zero (and also zero). Naturally, we assume topological transitivity, hence the system in question cannot split into ``one-sided" nonuniformly hyperbolic parts.

Probably, the simplest setting of such ``two-sided" nonhyperbolic dynamical systems (that is, with zero Lyapunov exponent in the interior of the spectrum)
can be found in step skew-products with a hyperbolic horseshoe map in its base and circle diffeomorphisms in its fibers.
The nonuniform hyperbolicity arises from the coexistence of  contracting and expanding regions (in the fibers) which are blended by the dynamics. These properties are exemplified by the hypotheses ``some hyperbolicity" and ``transitions in finite time" stated in Section~\ref{subsec:stepskew}.  
The considered dynamics is topologically transitive and simultaneously has ``horseshoes" which are contracting  and   ``horseshoes" which are expanding in the fiber direction. 
These horseshoes are intermingled and there coexist dense sets of periodic points with negative and positive fiber Lyapunov exponents. 
As a consequence, the system exhibits ergodic hyperbolic measures with positive entropy. 
An important feature is the occurrence of ergodic \emph{nonhyperbolic measures} (i.e., with zero Lyapunov exponent) with positive entropy, see \cite{BocBonDia:16}.
A natural question is what type of behavior (hyperbolic or nonhyperbolic one) prevails, for example in terms of entropy. Another important question is how the degree of hyperbolicity measured in terms of exponents varies, for example, how the entropy of the corresponding level sets changes.

In~\cite{DiaGelRam:} we provide a conceptual framework for the prototypical dynamics which present the features in the above paragraph (see also Sections~\ref{sec:Axioms} and~\ref{ss.previous}). Moreover, we see that the mentioned topological and ergodic properties hold even for perturbations of these systems%
\footnote{Indeed, as explained in~\cite[Section 8.3]{DiaGelRam:}, if $\mathcal S$ denotes the set of step skew-product maps $F$ as in~\eqref{eq:sp} which are robustly transitive and have periodic points of different indices, then there is a $C^1$-open and dense subset of $\mathcal S$ consisting of maps with satisfy the axioms stated in Section~\ref{sec:Axioms}.}.
The works~\cite{DiaGelRam:,DiaGelRam:b}
 contain results about the topology of the space of invariant measures which laid the basis for the multifractal analysis of the entropy of the level sets of fiber Lyapunov exponents. 

\subsection{Tools of multifractal analysis}\label{subsec:summary}
 
In the classical approach for multifractal analysis one expresses the entropy of a level set simultaneously
\begin{itemize}
\item  in terms of a restricted variational principle and
\item  in terms of the Legendre-Fenchel transform of a topological pressure function.
\end{itemize}
It is important to point out that in our setting the dynamical system as a whole does not satisfy the specification property and none of the previous approaches applies.
Instead we rather follow a thermodynamic approach based on restricted variational principles.
The philosophy is that in order to obtain relevant multifractal information about the respective classes of exponents one should not consider the whole variational-topological pressure, but instead  its restrictions to ergodic measures with corresponding exponents, so-called \emph{restricted variational pressures}\footnote{The use of restricted (sometimes also called \emph{hidden}) pressures was initiated in~\cite{MakSmi:00} (for rational maps of the Riemann sphere) and  subsequently used, for example, in~\cite{GelPrzRam:10} (for non-exceptional rational maps) and \cite{PrzRiv:} (for multimodal interval maps).}, and to derive the information about entropy on level sets from so-called \emph{exhausting families}.
 As the difficulty  in our setting comes from  the coexistence of negative, zero, and positive fiber Lyapunov exponents and as zero exponent measures are notoriously difficult to analyze, a natural solution is to separately consider the restricted pressures defined on the ergodic measures with negative and positive exponents, respectively. 

To make the link between restricted variational pressures and the multifractal information  which they carry for the relevant subsystems,  we follow  a somewhat general principle.
While we do not have  specification  on the whole space, we are able to find certain families of subsets (basic sets) on which we do have this property.
 First, we recall the general restricted variational principle for topological entropy in \cite{Bow:73} which provides a natural lower bound for $h_{\rm top}(\cL(\alpha))$, see Section \ref{subsec:ent}. 
 In Section~\ref{sec:exhau}, we are going to present a general theory of \emph{restricted pressures} which allows us to obtain dynamical properties of the full system knowing the properties of  subsystems. Our key-concept is  the existence of so-called \emph{exhausting families} on which each restricted variational pressure can be approached gradually. 
 In Section~\ref{sec:bridging} we show the existence of exhausting families in our setting, treating negative and positive exponents separately. 
For that we will strongly use the fact that for any pair of uniformly hyperbolic sets with negative (positive) fiber exponents there exists a larger one containing them both. 
We show that the entropy spectrum of fiber Lyapunov exponents is described in terms of the Legendre-Fenchel transforms of the respective restricted variational pressure functions and is simultaneously given in terms of a restricted variational principle. This applies to negative/positive exponents only.
As a consequence, in our setting, we show that for each $\alpha\in [\alpha_{\min}, 0)\cup(0,\alpha_{\max}]$ the level set $\cL(\alpha)$ of points with fiber exponent equal to $\alpha$ is nonempty and its topological entropy changes continuously with $\alpha$ (see Figure~\ref{fig.0}).

We proved in  \cite{DiaGelRam:} that any nonhyperbolic ergodic measure can be  approached by hyperbolic ones (weak$\ast$ and in metric entropy) and that then the difficulties arising from zero exponents can be somewhat circumvented. This provides a tool to deal with zero exponent (nonhyperbolic) ergodic measures, enables us to consider exhausting families ``approaching nonhyperbolicity", and to ``glue" the two parts of the spectrum,  which would be completely unrelated otherwise.

To extend our results to a description of the level set of zero exponent,  we then combine a thermodynamical and an orbitwise approach. On the one hand, we study the restricted variational pressure functions and extract properties from its  shape. This approach gives us convexity for free, which turns out to be a surprisingly useful property.
On the other hand, in our approach we put our hands on the orbits of the level sets (the amount of their entropy provides explicit information about them), using natural recurrence properties of the systems (which is guided by the concept of so-called blending intervals in Section~\ref{subsec:stepskew}), and follow an ``orbit-gluing approach". 

While for exponents $\alpha\in[\alpha_{\rm min},0)\cup(0,\alpha_{\rm max}]$ we can give the full description of the Lyapunov exponent level sets, including the restricted variational principle and the exact formula for their entropy, there are very restricted tools for studying the level set $\cL(0)$. 
We are able to describe its entropy, but the restricted variational principle for $h_{\rm top}(\cL(0))$ cannot be obtained by our methods.
Let us observe that the fact that $\cL (0)$ has positive topological entropy was obtained in a similar context in \cite{BocBonDia:16} by proving the existence of ergodic measures with positive entropy and zero exponents%
\footnote{Indeed, \cite{BocBonDia:16} shows the existence of a compact and invariant set with positive topological entropy consisting of points with zero Lyapunov exponent.}. In this paper, this property is also obtained as a surprising consequence of the shape of the pressure map.
Though positive, we also show that in our setting and assuming proximality the topological entropy of $\cL(0)$ is strictly smaller than the maximal, that is, the topological entropy of the system.

The systems we study always have (at least) two hyperbolic ergodic measures of maximal entropy, one  with negative and one with positive fiber Lyapunov exponent. Indeed, this is an immediate consequence of \cite{Cra:90}, obtained from a different point of view of our system as a random dynamical system, that is, as a product of independent and identically distributed circle diffeomorphisms, also observing the fundamental fact that our hypotheses exclude the case that our system is a rotation extension of a Bernoulli shift.
 This is a particular case of a result in a more general setting~\cite{RodRodTahUre:12}, stated for accessible partially hyperbolic diffeomorphisms having compact center leaves, see also~\cite{TahYan:} where higher regularity is required. Under the additional assumption of proximality, with~\cite{Mal:} we even can conclude uniqueness of the ergodic measure of maximal entropy with negative and  positive exponents, respectively.

\subsection{Structure of the paper}

In Section \ref{sec:results} we precisely state our main results with all details from which we deduce the simplified versions Theorems~\ref{teo:specialste} and~\ref{teo:introSL2R}. In Section \ref{sec:Axioms} we describe the setting in which we derive our results and in Section~\ref{ss.previous} recall some key results about ergodic approximations. In Section \ref{sec:entpres} we give some basic information about the thermodynamical formalism (entropy and pressure function and its convex conjugate). 
In Section \ref{sec:exhau} we introduce (in an abstract setting) the restricted pressures and exhausting families, then in Section \ref{sec:bridging} verify their existence in the setting of our paper. 
Our main result Theorem~\ref{main1} is proved in Sections~\ref{sec:proofmain1} and~\ref{sec:proofoflowerbound}.
Theorem~\ref{main2} and Corollary~\ref{c.maximalentropy} are proved in Section~\ref{sec:proofmain2}.
Theorem~\ref{main3} is shown in Section~\ref{sec:proofmain3}. 
To apply the above results to matrix cocycles, in Section~\ref{sec:cocyles} we develop several general tools to relate Lyapunov exponents of cocycles with the ones of the induced skew-products. The main result there is Theorem~\ref{theoprop:onesidedspectrum} which implies Theorem~\ref{teo:SL2Rskewproduct}. Finally,  we recall in Appendix~\ref{App:A} some more details about the space of elliptic cocycles and in Appendix~\ref{App:B} the definition of topological entropy  of general sets.

\section{Precise statements of the results}\label{sec:results}

Let $\sigma\colon\Sigma_N\to\Sigma_N$,  $N\ge2$,  be the usual shift map on the space $\Sigma_N\eqdef\{0,\ldots,N-1\}^\bZ$ of two-sided sequences. We equip the shift space $\Sigma_N$ with the standard metric $d_1(\xi,\eta)\eqdef 2^{-n(\xi,\eta)}$, where $n(\xi,\eta)\eqdef\sup\{\lvert \ell\rvert\colon \xi_{ i}=\eta_{i}\text{ for }i=-\ell,\ldots,\ell\}$. We equip $\Sigma_N\times\bS^1$   with the metric $d((\xi,x),(\eta,y))\eqdef\sup\{d_1(\xi,\eta),\lvert x-y\rvert\}$, where $\lvert\cdot\rvert$ is the usual metric on $\bS^1$.

We will require the step skew-product
\begin{equation}\label{eq:fssp}
	F\colon \Sigma_N\times \bS^1\to \Sigma_N\times \bS^1,
	\quad
	F(\xi,x) =(\sigma(\xi), f_{\xi_0}(x))
\end{equation}
to satisfy Axioms CEC$\pm$ and Acc$\pm$ (see Section~\ref{sec:Axioms}).
Sometimes, we will also take another point of view and study the underlying \emph{iterated function system} (\emph{IFS}) generated by the family of maps $\{f_i\}_{i=0}^{N-1}$.
 We will denote by $\pi\colon\Sigma_N\times\bS^1\to\Sigma_N$ the natural projection $\pi(\xi,x)\eqdef\xi$.

Let $\cM$ be the space of $F$-invariant probability measures supported in $\Sigma_N\times \bS^1$, equip $\cM$ with the weak$\ast$ topology, and denote by $\cM_{\rm erg}\subset\cM$ the subset of ergodic measures. To characterize nonhyperbolicity, given $\mu\in\cM$
 denote by $\chi(\mu)$ its \emph{(fiber) Lyapunov exponent} which is given by
\[
	\chi(\mu)
	\eqdef \int\log\,\lvert(f_{\xi_0})'(x)\rvert\,d\mu(\xi,x). 	
\]
An ergodic measure $\mu$ is  {\emph{nonhyperbolic}} if $\chi(\mu)=0$. Otherwise the measure is  {\emph{hyperbolic}}. In our setting, any hyperbolic ergodic  measure has either a negative or a positive
exponent.
Accordingly, we divide the set of all \emph{ergodic} measures and  consider the decomposition
\begin{equation}\label{eq:ergdecompo}
	\cM_{\rm erg}=\cM_{\rm erg,<0}\cup\cM_{\rm erg,0}\cup\cM_{\rm erg,>0}
\end{equation}
into measures with negative, zero, and positive fiber Lyapunov exponent, respectively.
In our setting, each component is nonempty.
In general, it is very difficult to determine which type of hyperbolicity ``prevails''.
For that we will study the spectrum of possible exponents and will perform a  multifractal analysis of the topological entropy of level sets of equal (fiber) Lyapunov exponent.

To be more precise, a sequence $\xi=(\ldots\xi_{-1}.\xi_0\xi_1\ldots)\in\Sigma_N$ can be written as $\xi=\xi^-.\xi^+$, where $\xi^+\in\Sigma_N^+\eqdef\{0,\ldots,N-1\}^{\bN_0}$ and $\xi^-\in\Sigma_N^-\eqdef\{0,\ldots,N-1\}^{-\bN}$.
Given  \emph{finite} sequences $(\xi_0\ldots \xi_n)$ and  $(\xi_{-m}\ldots\xi_{-1})$, we let
\[
    f_{[\xi_0\ldots\,\xi_n]}
    \eqdef f_{\xi_n} \circ \cdots \circ f_{\xi_0}
    \quad\text{ and }\quad
    f_{[\xi_{-m}\ldots\,\xi_{-1}.]}
    \eqdef 
    (f_{[\xi_{-m}\ldots\,\xi_{-1}]})^{-1}.
\]
For $n\ge0$ denote also
\[
	f^n_\xi\eqdef f_{[\xi_0\ldots\,\xi_{n-1}]}
	\quad\text{ and }\quad
	f^{-n}_\xi\eqdef f_{[\xi_{-n}\ldots\,\xi_{-1}.]}.
\]	
As usual, we use the following notation for cylinder sets
\[
	[\xi_0\ldots\xi_n]
	\eqdef \{\eta\in\Sigma_N\colon \eta_0=\xi_0,\ldots,\eta_n=\xi_n\}.
\]
Given $X=(\xi,x)\in \Sigma_N\times\bS^1$ consider the \emph{(fiber) Lyapunov exponent} of $X$
\[
	\chi(X)
	\eqdef
	 \lim_{n\to\pm\infty}\frac{1}{ n}\log\,\lvert (f_{\xi}^n)'(x)\lvert  ,
\]
where we assume that both limits $n\to\pm \infty$  exist and coincide. Note that in our context the  exponent is nothing but the Birkhoff average of the continuous function (also called \emph{potential}) $\varphi\colon \Sigma_N\times\bS^1\to\bR$ defined for $X=(\xi,x)$ by
\begin{equation}\label{def:potential}
	\varphi(X)
	\eqdef  \log\,\lvert (f_{\xi_0})'(x)\rvert.
\end{equation}

We will analyze the topological entropy of the following \emph{level sets of Lyapunov exponents}: given $\alpha\in\bR$ let
\begin{equation}\label{def:levelset}
	\cL(\alpha)
	\eqdef \big\{X\in\Sigma_N\times\bS^1\colon \chi(X)=\alpha\big\},
\end{equation}
assuming that the Lyapunov exponent at $X$ is well defined and equal to $\alpha$.
Note that each level set is invariant but, in general, noncompact. Hence we will rely on the general concept of topological entropy $h_{\rm top}$ introduced by Bowen~\cite{Bow:73} (see Appendix~\ref{App:B}).
Denoting by $\cL_{\rm irr}$   the set of points where the fiber Lyapunov exponent is not well-defined (either one of the limits does not exist or both limits exist but they do not coincide), we obtain the following \emph{multifractal decomposition}
\[
   \Sigma_N\times \bS^1
	=\bigcup_{\alpha\in\bR}\cL(\alpha)\cup\cL_{\rm irr}.
\]
Note that $\cL(\alpha)$ will be nonempty in some range of $\alpha$, only. Under our axioms  this range   decomposes into three natural nonempty parts
\[
	\{\alpha\colon \cL(\alpha)\ne\emptyset\}
	= [\alpha_{\rm min},0)
	\cup\{0\}
	\cup(0,\alpha_{\rm max}],
\]
where
\[
	\alpha_{\rm max}
	\eqdef \sup\big\{\alpha\colon \cL (\alpha)\ne\emptyset\big\},
	\quad
	\alpha_{\rm min}
	\eqdef \inf\big\{\alpha\colon \cL (\alpha)\ne\emptyset\big\}.
\]
We have that $\inf$ and $\sup$ are indeed attained, justifying the notation.

To state our main results, we need the following thermodynamical quantities. Denote by $h(\mu)$ the \emph{entropy} of a measure $\mu$ and consider the following pressures and their convex conjugates (see Section~\ref{sec:entpres} for details)
\begin{equation}\label{eq:defpres}
	P_{\ast}(q\varphi)
	\eqdef \sup_{\mu\in\cM_{\rm erg,\ast}}\big(h(\mu)-q\chi(\mu)\big)
	,\quad
	\cE_\ast(\alpha)
	\eqdef \inf_{q\in\bR}\big(P_\ast(q\varphi)-q\alpha\big),
\end{equation}
where $\ast$ should be replaced by $<0$ and $>0$, respectively (recall~\eqref{eq:ergdecompo}). In the terminology of~\cite{PrzRivSmi:04}, this would be called (\emph{positive}/\emph{negative}) \emph{variational hyperbolic pressure}, we call them simply \emph{pressures}. For simplicity we will use the notation
\[
\cP_\ast(q) \eqdef P_\ast(q\varphi),
\]
as $\{q\varphi\}_{q\in\bR}$ is the only family of potentials whose pressure  we are going to consider.
Similarly, we define
\[
	\cP_0(q)
	\eqdef \sup_{\mu\in\cM_{\rm erg,0}}h(\mu).
\]
Clearly,
\[
	\max\{\cP_{<0}(q),\cP_0(q),\cP_{>0}(q)\}=P_{\rm top}(q\varphi)
\]
is the classical \emph{topological pressure} of $q\varphi $ with respect to $F$ (see~\cite[Chapter 7]{Wal:82}).
We will also write $\cE$ for both $\cE_{>0}$ and $\cE_{<0}$, because the domains of those two functions are disjoint.

\begin{mtheorem}\label{main1}
Consider a transitive step skew-product map $F$ as in~\eqref{eq:fssp} whose fiber maps are $C^1$.
	Assume that $F$ satisfies Axioms CEC$\pm$ and Acc$\pm$.

Then there are numbers $\alpha_{\rm min}<0<\alpha_{\rm max}$ such that $\alpha\in[\alpha_{\min},\alpha_{\max}]$ if and only if $\cL(\alpha)\ne\emptyset$. Moreover,
\begin{itemize}
\item[a)] for every $\alpha\in[\alpha_{\rm min},0)$ we have
\[
	h_{\rm top}(\cL(\alpha))
	= \sup \big\{h(\mu)\colon \mu\in\cM_{\rm erg},\chi(\mu)=\alpha\big\}
	= \cE_{<0}(\alpha),
\]
\item[b)] for every $\alpha\in(0,\alpha_{\rm max}]$ we have
\[
	h_{\rm top}(\cL(\alpha))
	= \sup \big\{h(\mu)\colon \mu\in\cM_{\rm erg},\chi(\mu)=\alpha\big\}
	= \cE_{>0}(\alpha),
\]
\item[c)] for every $\alpha\in\{\alpha_{\min},0,\alpha_{\max}\}$ we have
\[
	\lim_{\beta\to \alpha}h_{\rm top}(\cL(\beta))
	=h_{\rm top}(\cL(\alpha)),
\]
\item[d)] $h_{\rm top}(\cL(0))>0$.
\end{itemize}
Moreover, there exist (finitely many) ergodic measures $\mu_+, \mu_-$ of maximal entropy $h(\mu_\pm)=\log N$ and with $\chi(\mu_-) < 0 < \chi(\mu_+)$.
\end{mtheorem}

\begin{figure}
\begin{minipage}[c]{\linewidth}
\centering
\begin{overpic}[scale=.28]{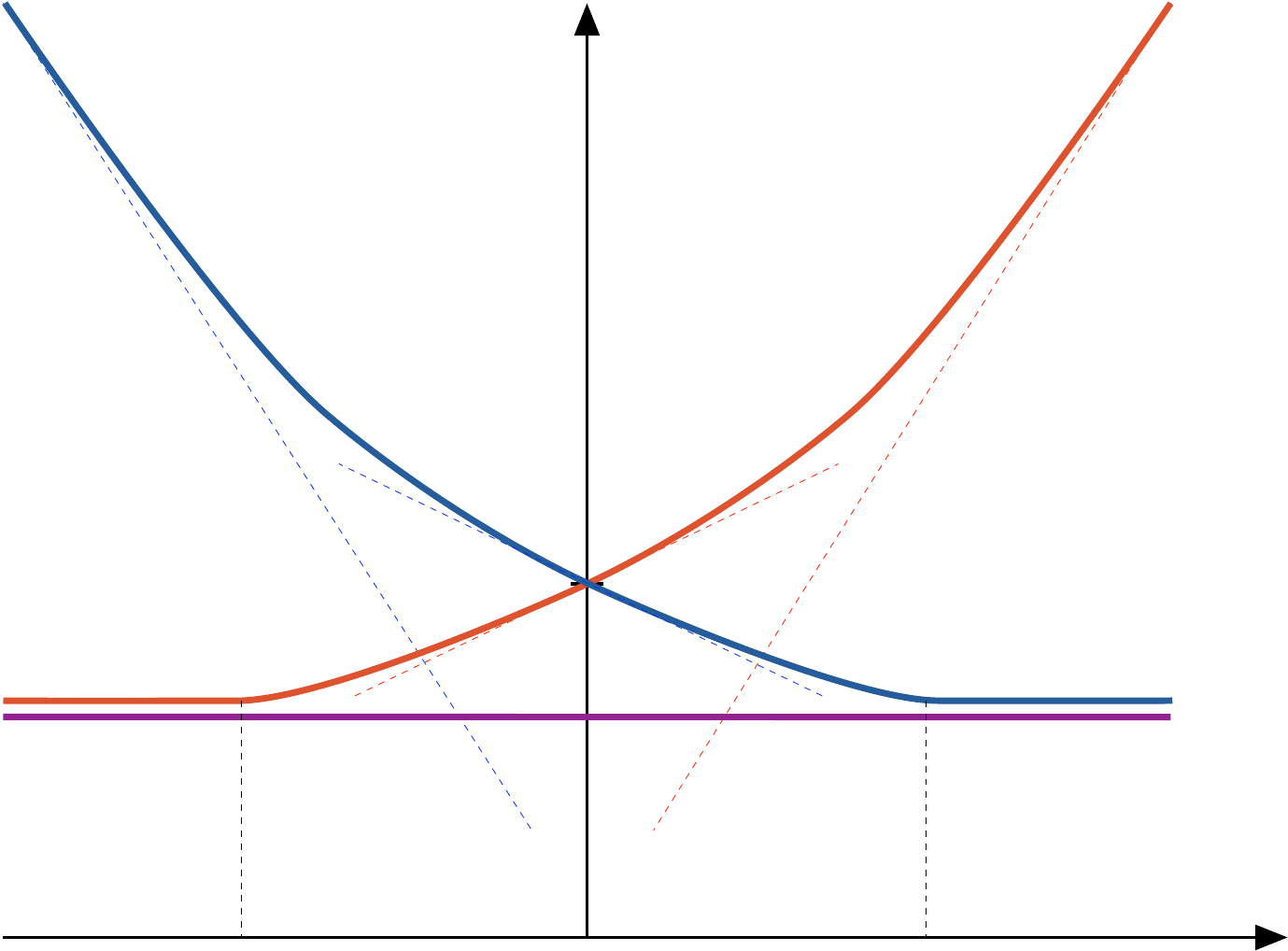}
  	\put(16,-7){\tiny$D_+$}
 	\put(68,-7){\tiny$D_-$}
	\put(57,70){\tiny \textcolor{red}{$\cP_{>0}(q)$}}
 	\put(5,70){\tiny \textcolor{blue}{$\cP_{<0}(q)$}}
 	\put(0,10){\tiny \textcolor{magenta}{$\cP_{0}(q)$}}
 	\put(95,5){\tiny $q$}
 \end{overpic}
 \hspace{0.1cm}
 \begin{overpic}[scale=.28]{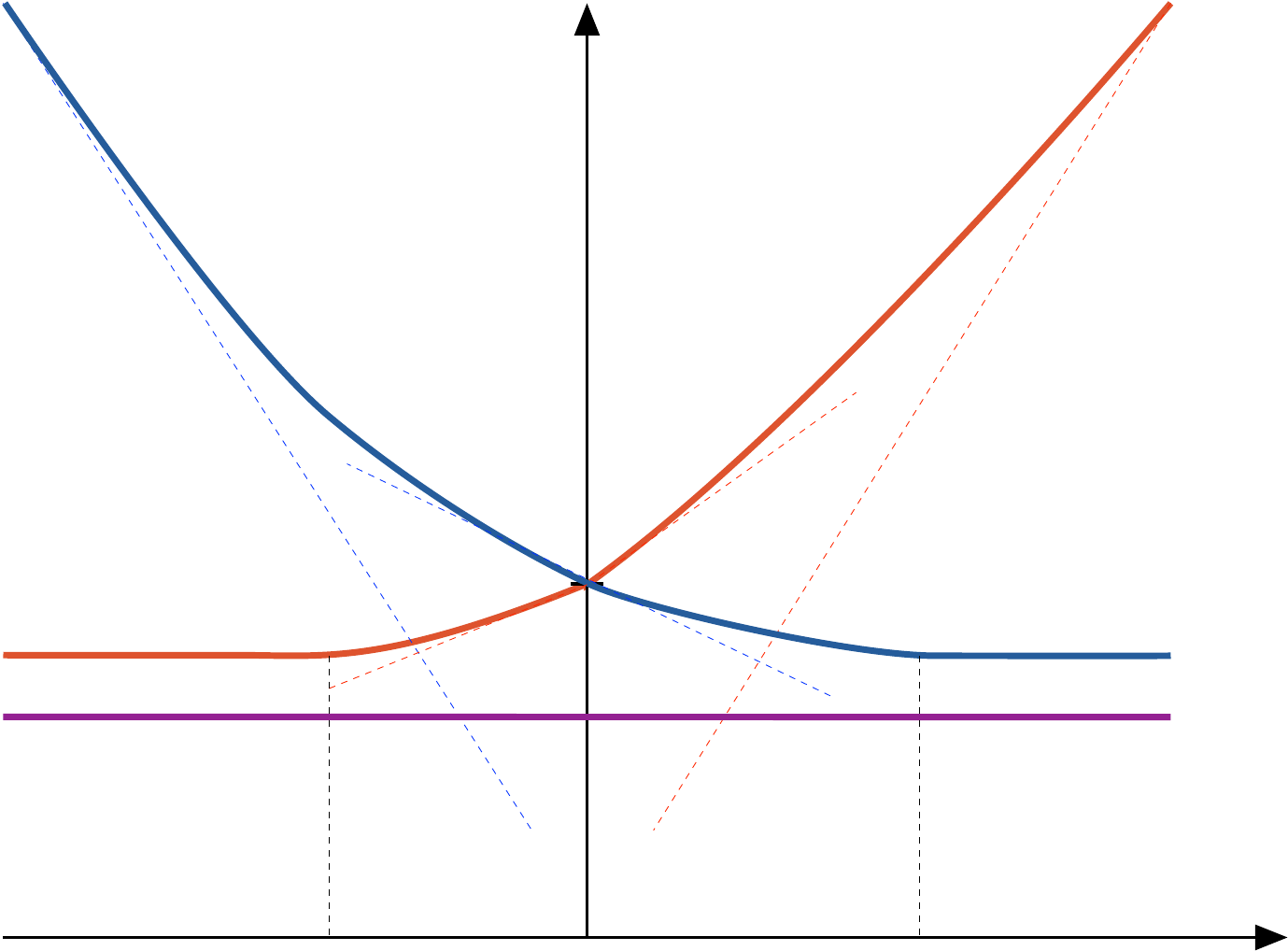}
 	\put(20,-7){\tiny$D_+$}
 	\put(68,-7){\tiny$D_-$}
 	\put(57,70){\tiny \textcolor{red}{$\cP_{>0}(q)$}}
 	\put(5,70){\tiny \textcolor{blue}{$\cP_{<0}(q)$}}
 	\put(0,10){\tiny \textcolor{magenta}{$\cP_{0}(q)$}}
 	\put(95,5){\tiny $q$}
 \end{overpic}
 \hspace{0.1cm}
 \begin{overpic}[scale=.28]{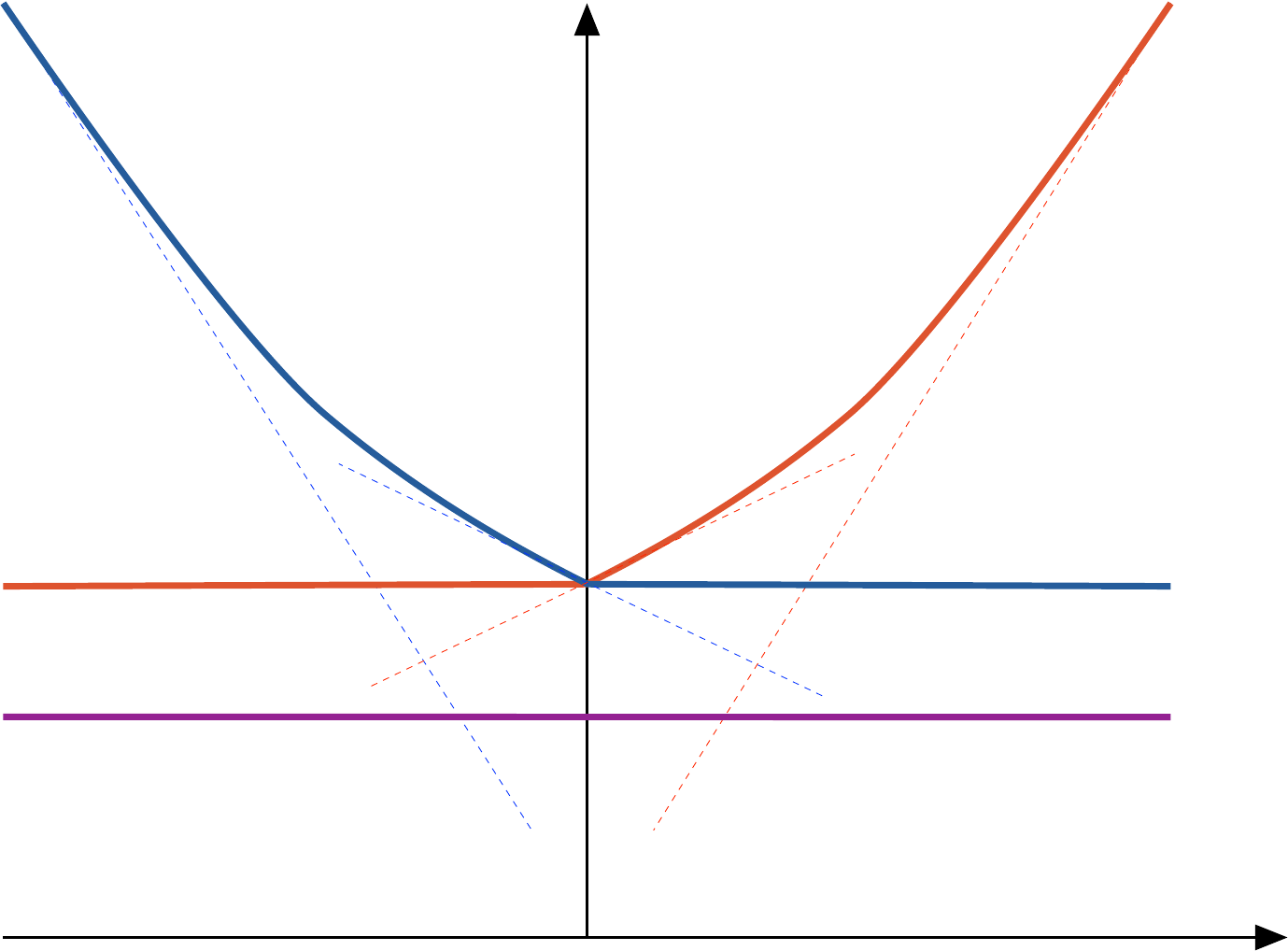}
  	\put(20,-7){\tiny$D_+=D_-=0$}
	\put(57,70){\tiny \textcolor{red}{$\cP_{>0}(q)$}}
 	\put(5,70){\tiny \textcolor{blue}{$\cP_{<0}(q)$}}
 	\put(0,10){\tiny \textcolor{magenta}{$\cP_{0}(q)$}}
 	\put(95,5){\tiny $q$}
\end{overpic}
\caption{Pressures. Left figure: Under the hypothesis of Theorem~\ref{main2}}
\label{Fig:Pressure}
\vspace{0.5cm}
\begin{overpic}[scale=.28]{P_Ea_ex2.pdf}
 	\put(48,68){\tiny $\cE(\alpha)$}
 	\put(95,5){\tiny $\alpha$}
 	\put(23,-5){\tiny $\alpha_-$}
 	\put(62,-5){\tiny $\alpha_+$}
 	\put(2,-5){\tiny $\alpha_{\rm min}$}
 	\put(79,-5){\tiny $\alpha_{\rm max}$}
 \end{overpic}
 \hspace{0.1cm}
 \begin{overpic}[scale=.28]{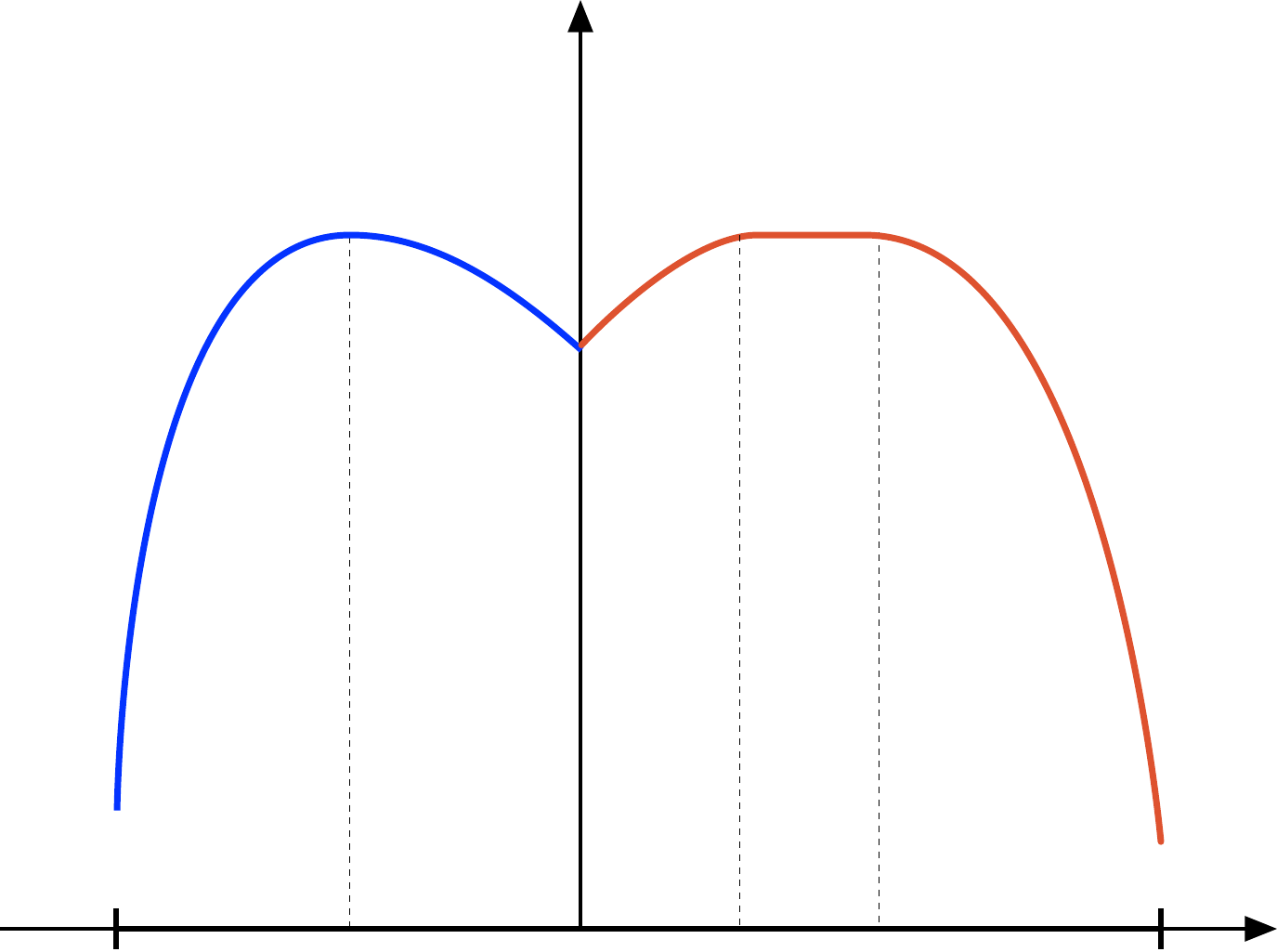}
 	\put(48,68){\tiny $\cE(\alpha)$}
 	\put(95,5){\tiny $\alpha$}
 	\put(23,-5){\tiny $\alpha_-$}
 	\put(62,-5){\tiny $\alpha_+$}
 	\put(2,-5){\tiny $\alpha_{\rm min}$}
 	\put(79,-5){\tiny $\alpha_{\rm max}$}
 \end{overpic}
 \hspace{0.1cm}
 \begin{overpic}[scale=.28]{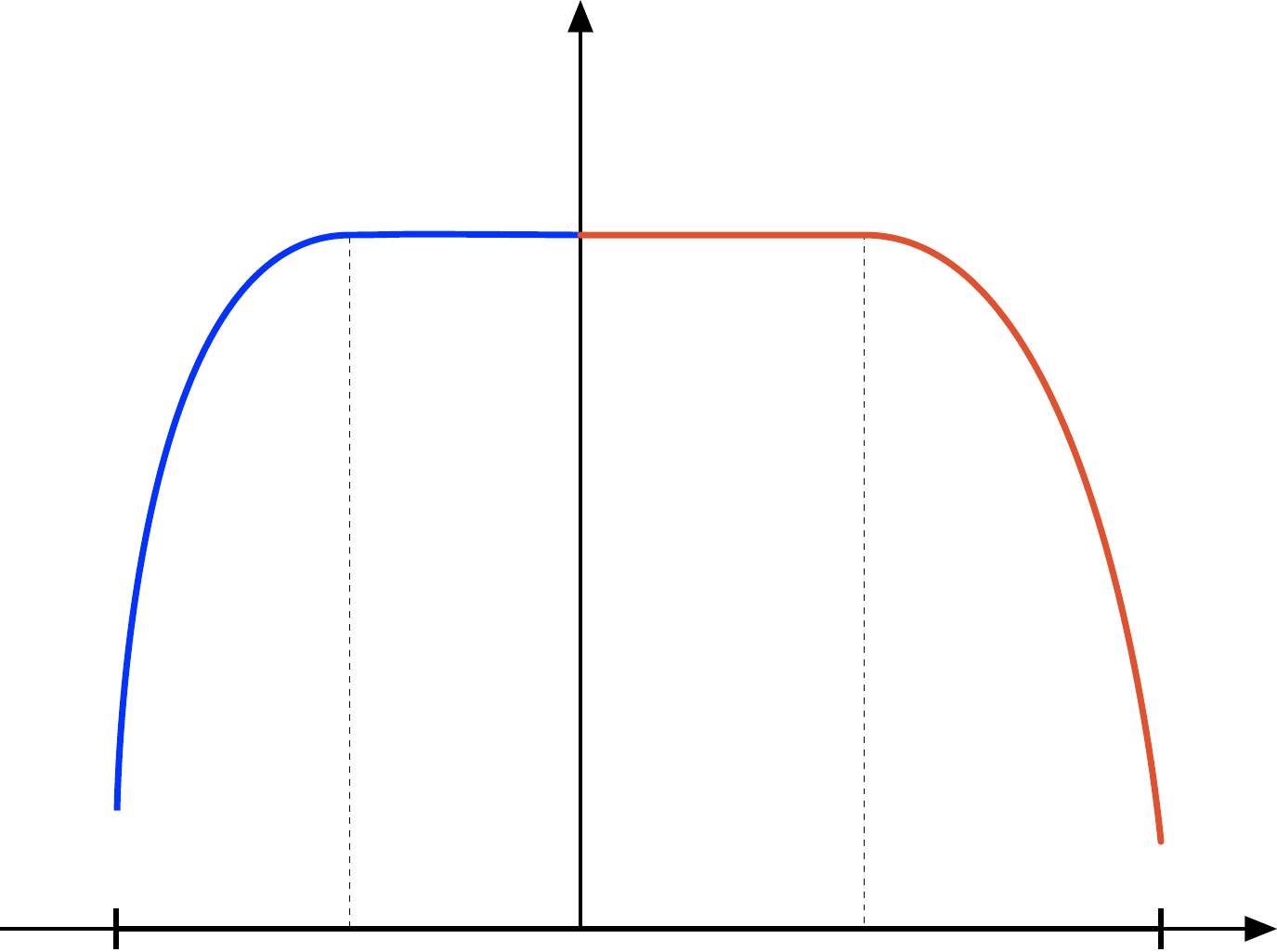}
 	\put(48,68){\tiny $\cE(\alpha)$}
 	\put(95,5){\tiny $\alpha$}
 	\put(23,-5){\tiny $\alpha_-$}
 	\put(62,-5){\tiny $\alpha_+$}
 	\put(2,-5){\tiny $\alpha_{\rm min}$}
 	\put(79,-5){\tiny $\alpha_{\rm max}$}
\end{overpic}
\caption{Convex conjugates. Left figure: Under the hypothesis of Theorem~\ref{main2}}
\label{Fig:Legendre}
\end{minipage}
\end{figure}

To prove uniqueness of the measures $\mu_\pm$ of maximal entropy, we require the additional assumption (see Section \ref{sec:synch} for  discussion). We say that the iterated function system (IFS) generated by the family of fiber maps $\{f_i\}_{i=0}^{N-1}$ of the step skew-product map $F$ is \emph{proximal}%
\footnote{We borrow this terminology from~\cite{Mal:}.} 
if for every $x,y\in \bS^1$ there exists at least one sequence $\xi\in\Sigma_N$ such that $\lvert f_{\xi}^n(x)- f_\xi^n(y)\rvert \to 0$ as $\lvert n\rvert\to\infty$. By some abuse of notation, in this case we also say that the skew-product is proximal.

\begin{remark}\label{rem:proximal}
It is easy to see that the step skew-product is proximal if, for example, there exists one Morse-Smale fiber map with exactly one attracting and one repelling fixed point 
(North pole-South pole map) 
and the step skew-product satisfies  Axioms CEC$\pm$ and Acc$\pm$.
\end{remark}

\begin{mtheorem}\label{main2}
Assume the hypothesis of Theorem~\ref{main1} and also proximality of the step skew-product. Then there exist unique ergodic $F$-invariant probability measures $\mu_-$ and $\mu_+$ of maximal entropy $h(\mu_\pm)=\log N$, respectively, and satisfying
\[
	\alpha_-\eqdef \chi(\mu_-)<0<\alpha_ +\eqdef\chi(\mu_+).
\]
We have
\[
	h_{\rm top}(\cL(\alpha_-)) = h_{\rm top}(\cL(\alpha_+)) = \log N
\]
and for all $\alpha \in(\alpha_{\rm min},\alpha_{\rm max})\setminus\{ \alpha_-, \alpha_+\}$ we have
\[
	0< h_{\rm top}(\cL(\alpha)) < \log N.
\]
\end{mtheorem}

Under the hypothesis of Theorem~\ref{main1}, possible shapes of the graph of the corresponding function $\alpha \to h_{\rm top}(\cL(\alpha))=\cE(\alpha)$ are as in Figure \ref{Fig:Legendre} and under the hypotheses of Theorem~\ref{main2} as in Figure \ref{Fig:Legendre} (left figure). 

Similar phenomenon as in Theorem~\ref{main2} (the entropy achieving its maximum away from zero exponent) in a slightly different setting (for ergodic measures on $C^2$ systems) was observed in \cite{TahYan:}. 
We note that somewhat related questions about the topology of the space of measures are considered in~\cite{GorPes:17,DiaGelRam:b,BocBonGel:,GelKwi:}.

In the following, when referring to  \emph{weak$\ast$ and in entropy convergence} we mean that a sequence of measures converges in the weak$\ast$ topology and their entropies  converge to the entropy of the limit measure.

\begin{mcorollary}\label{c.maximalentropy}
Under the hypothesis of Theorem~\ref{main2}, no measure which is a nontrivial convex combination of the two ergodic measures of maximal entropy is a weak$\ast$ and in entropy limit of ergodic measures.
\end{mcorollary}

The results in~\cite{TahYan:} and our results suggest the following conjecture (which is indeed true for maximal entropy measures, by Corollary~\ref{c.maximalentropy}).

\begin{conjecture}
For every pair of hyperbolic ergodic measures $\mu_1$ and $\mu_2$ with
$\chi(\mu_1)<0<\chi(\mu_2)$  every nontrivial convex combination of $\mu_1$ and $\mu_2$ cannot be
approximated (weak$\ast$ and in entropy) by ergodic measures.
\end{conjecture}

We finally summarize the properties of (restricted) pressure functions, its Legendre-Fenchel transform, and of the entropy spectrum of Lyapunov exponents in the following theorem (compare  Figures~\ref{Fig:Pressure} and~\ref{Fig:Legendre}).

\begin{mtheorem} \label{main3}
	Under the assumptions of Theorem \ref{main1}, we have the following:
\begin{itemize}
\item[a)] $\cP_{<0}$  and $\cP_{>0}$ are nonincreasing and nondecreasing convex functions, respectively,
\item[b)] (Plateaus) There are numbers $D_\pm$ and $h_\pm>0$ such that
\[
	\cP_{<0}(q) = h_-\text{ for all }q\ge D_-
	\quad\text{ and } \quad
	\cP_{>0}(q) = h_+\text{ for all }q\le D_+
.
\]	
\item[c)] $h_- = h_+ = h_{\rm top} (\cL(0))$.
\item[d)] $D_+ \leq 0 \leq D_-$.
\item[e)] $\cP_{>0}(0) = \cP_{<0}(0) = \log N = h_{\rm top} (F)$.
\item [f)] The map $\alpha \mapsto h_{\rm top} (\cL(\alpha))$ achieves its maximum value $\log N$ at some points
\[	
	\alpha_-<0
	\quad\text{ and }\quad
	\alpha_+>0.
\]	
\item[g)] For $\alpha < 0$ the function $\alpha \mapsto \cE(\alpha)$ is a Legendre-Fenchel transform of $q\mapsto \cP_{<0}(q)$. Similarly, for $\alpha > 0$ the function $\alpha \mapsto \cE(\alpha)$ is a Legendre-Fenchel transform of $q\mapsto \cP_{>0}(q)$. In particular, $\alpha \mapsto \cE(\alpha)$ is a concave function on the domains  $\alpha <0$ and $\alpha>0$, respectively.
\item[h)] $h_{\rm top} (\cL(\alpha))$ is a continuous function on $[\alpha_{\min},\alpha_{\max}]$.
\item[i)] We have $0\leq -D_L\cE(0)<\infty$ and $0\leq D_R\cE(0)<\infty$, where $D_L$ and $D_R$ denote the one-sided derivatives from the left and from the right, respectively.
\item[j)] $\cE(0)=\lim_{\alpha\to0}\cE(\alpha)>0$ and hence $h_{\rm top} (\cL(0)) >0$.
\end{itemize}
Moreover, under the assumptions of Theorem \ref{main2} we have additional properties
\begin{itemize}
\item[k)] $\cP_{>0}$ and $\cP_{<0}$ are differentiable at $q=0$
\end{itemize}
and in items d) and i) we have strict inequalities:
\[
	D_+ <0 <D_-
	\quad\text{ and }\quad	
	D_L\cE(0) <0<D_R\cE(0),
\]	
 and the points $\alpha_-, \alpha_+$ in item f) are the unique numbers $\alpha$ for which $h_{\rm top}(\cL(\alpha))=\log N$.
\end{mtheorem}

\begin{remark}
The following questions remain open. The restricted pressures can be differentiable or nondifferentiable at the beginning of the plateaus in Theorem~\ref{main3} item b). The nondifferentiability of, for example, $\cP_{>0}$ at $D_-$ would mean that $\cE(\alpha)$ is linear on some interval $[0, q]$. Further regularity properties (smoothness, analyticity)  of the restricted pressure functions (excluding the ends of plateaus) and of the spectrum are unknown.

The asymptote of $\cP_{>0}$ at $q\to \infty$ is some line $\{P=\alpha_{\max}q + h_{\max}\} $, similarly $\cP_{<0}$ is asymptotic to $\{P=\alpha_{\min}q+h_{\min}\}$, and we do not know whether $h_{\max}$ and $h_{\min}$ are equal to zero (which would mean that $h_{\rm top}( \cL(\alpha_{\max})) = h_{\rm top} (\cL(\alpha_{\min}))=0$; this phenomenon is sometimes referred to as \emph{ergodic optimization}, see for example~\cite{Jen:06}).

Finally, even though we know that there do exist ergodic measures with Lyapunov exponent zero and with positive entropy (this follows from \cite{BocBonDia:16}), we do not know if there exist such measures with entropy arbitrarily close to $h_{\rm top}( \cL(0))$ (which would mean that we have the restricted variational principle also for exponent zero).
\end{remark}

Our final result deals with cocycles. Recall, given a cocycle $\mathbf A\in \mathrm{SL}(2,\bR)^N$, the definitions of the associated skew-product $F_{\mathbf A}\colon \Sigma_N\times\bP^1\to\Sigma_N\times\bP^1$ with fiber maps $f_A$ as in~\eqref{eq:defsteskecoc} and the level sets $\cL^+_{\mathbf A}$ in~\eqref{def:levelsetA} and $\cL(\alpha)$ as defined in~\eqref{def:levelset} for $F_{\mathbf A}$. Recall also the existence of the open and dense subset $\fE_{N,\rm shyp}\subset \fE_N$ in Theorem \ref{teo:introSL2R} and  Appendix~\ref{App:A}.

\begin{mtheorem}\label{teo:SL2Rskewproduct}
	For every $N\ge2$ and every $\mathbf A\in \fE_{N,\rm shyp}$  we have the following: There are numbers $0<\alpha_+<\alpha_{\rm max}$  such that for every $\alpha\in[0,\alpha_{\rm max}]$ we have $\cL^+_{\mathbf A}(\alpha)\ne\emptyset$. Moreover,
\begin{enumerate}
\item[a)] for every $\alpha\in[0,\alpha_{\rm max}]$ we have 
\[
	h_{\rm top}(\cL^+_{\mathbf A}(\frac\alpha2))
	= h_{\rm top}(\cL(\alpha))
	= h_{\rm top}(\cL(-\alpha))
	.
\]
In particular, the function $\alpha\mapsto h_{\rm top}(\cL(\alpha))$ is even.
\item[b)] For all $\alpha\in[0,\alpha_{\rm max})\setminus\{\alpha_+\}$ we have
\[
	0
	< 	 h_{\rm top}(\cL^+_{\mathbf A}(\alpha))
	< \log N
	=h_{\rm top}(\cL^+_{\mathbf A}(\alpha_+)) .
\] 
\end{enumerate}
\end{mtheorem}
 
The proof of Theorem~\ref{teo:SL2Rskewproduct} has two parts. First, for $\mathbf A\in \fE_{N,\rm shyp}$ the skew-product $F_{\mathbf A}$ satisfies the hypotheses of Theorems~\ref{main1} and~\ref{main2}. Second,  there is done a careful analysis of the relation between the spectra of exponents of the cocycle and the ones of the associated skew-product (see Theorem~\ref{theoprop:onesidedspectrum}), which is done in Section~\ref{sec:cocyles}. 

\section{Setting}\label{sec:Axioms}

We  recall the precise setting of our Axioms CEC$\pm$ and Acc$\pm$ and their main consequences, established in~\cite{DiaGelRam:}. The step skew-product structure of $F$ allows us to reduce the study of its dynamics to the study of the IFS generated by the fiber maps $\{f_i\}_{i=0}^{N-1}$. In what follows we always assume that $F$ is transitive.

Given a point $x\in\bS^1$, consider and define its \emph{forward} and \emph{backward orbits}
by
\[
	\mathcal{O}^+(x)	
	\eqdef \bigcup_{n\ge 1}
	\bigcup_{(\theta_0\ldots\theta_{n-1})}f_{[\theta_0\ldots\,\theta_{n-1}]}(x)
	\,\,\mbox{and} \,\,
	\mathcal{O}^{-}(x) \eqdef  \bigcup_{m\le 1}
	\bigcup_{(\theta_{-m}\ldots\theta_{-1})}f_{[\theta_{-m}\ldots\,\theta_{-1}.]}(x),
\]
respectively. Consider also the \emph{full orbit} of $x$
$$
	\mathcal{O}(x) \eqdef \mathcal{O}^+(x) \cup \mathcal{O}^-(x).
$$
Similarly, we define the orbits $\mathcal{O}^+(J),\mathcal{O}^-(J)$, and $\mathcal{O}(J)$ for any subset $J\subset\bS^1$.

In requiring that the step skew-product $F$ with fiber maps $\{f_i\}_{i=0}^{N-1}$ \emph{satisfies the Axioms CEC$\pm$ and Acc$\pm$} we mean that there are so-called  (closed) \emph{forward} and \emph{backward blending intervals} $J^+, J^-\subset\bS^1$ such that the following properties hold.

\medskip\noindent
\textbf{CEC+($J^+$) (Controlled Expanding forward Covering  relative to $J^+$).} There exist positive constants $K_1,\ldots,K_5$ such that for every interval $H\subset\bS^1$ intersecting $J^+$ and satisfying $\lvert H\rvert<K_1$ we have
\begin{itemize}
\item  (controlled covering)  there exists a finite sequence $(\eta_0\ldots\eta_{\ell-1})$ for some positive integer $\ell\le  K_2\,\lvert\log\,\lvert H\rvert\rvert +K_3$ such that
\[
	f_{[\eta_0\ldots\,\eta_{\ell-1}]}(H)\supset B(J^+,K_4),
\]	
where $B(J^+,\delta)$ is the $\delta$-neighborhood of the set $J^+$.
\item  (controlled expansion) for every $x\in H$ we have
\[
	\log \,\lvert  (f_{[\eta_0\ldots\,\eta_{\ell-1}]})'(x)\rvert\ge \ell K_5.
\]	
\end{itemize}

\medskip\noindent
\textbf{CEC$-(J^-$) (Controlled Expanding backward Covering relative to $J^-$).} The step skew-product  $F^{-1}$ satisfies the Axiom CEC$+(J^-)$.

\medskip\noindent
\textbf{Acc$+$($J^+$) (forward Accessibility  relative to $J^+$).} $\mathcal{O}^+(\interior J^+)=\bS^1$.

\medskip\noindent
\textbf{Acc$-$($J^-$) (backward Accessibility relative to $J^-$).} $\mathcal{O}^-(\interior J^-)=\bS^1$.

\medskip
When the step skew-product $F$ is transitive then there is a common interval $J\subset\bS^1$ satisfying CEC$\pm(J)$ and Acc$\pm(J)$ (see Lemma~\ref{l:minimalcommon} and detailed discussion in~\cite[Section 2.2]{DiaGelRam:}).

\begin{remark}[Remark~\ref{r.Gorod} continued]\label{rem:example}
	Consider an IFS $\{f_i\}_{i=0}^{N-1}$ of diffeomorphisms $f_0,\dots, f_{N-1}\colon \mathbb{S}^1\to \mathbb{S}^1$ and assume that there are finite sequences $(\xi_0\ldots\xi_r)$ and $(\zeta_0\ldots\zeta_t)$ such that $f_{[\xi_0\ldots\,\xi_r]}$ is Morse-Smale with exactly one  attracting fixed point and one  repelling fixed point and $f_{[\zeta_0\ldots\,\zeta_t]}$ is an irrational rotation. Then by~\cite[Proposition 8.8]{DiaGelRam:},  every $C^1$-small perturbation of this IFS satisfies Axioms CEC$\pm$ and Acc$\pm$. Moreover, the system is proximal (recall Remark~\ref{rem:proximal}).

	Also note that it is enough to assume that $f_{[\zeta_0\ldots\,\zeta_t]}$ is only $C^2$ conjugate to an irrational rotation (this avoids Denjoy-like counterexamples guaranteeing that every orbit is dense).
\end{remark}	
	
\section{Ergodic approximations}\label{ss.previous}

We recall some technical results from~\cite{DiaGelRam:}. The first one claims that any  nonhyperbolic ergodic measure $\mu$ (that is, with exponent $\chi(\mu)=0$) is  weak$\ast$ and in entropy approximated by hyperbolic ergodic measures. 

\begin{lemma}[Rephrasing partially {\cite[Theorem 1]{DiaGelRam:}}] 
\label{lem:main30}
For every ergodic measure $\mu$ with zero Lyapunov exponent $\chi (\mu)= 0$ there is a sequence of ergodic measures $\nu_i$ with Lyapunov exponents $\chi (\nu_i) =\beta_i$ such that $\beta_i>0$, $\beta_i\to 0$, $\lim_{i\to\infty}\nu_i=\mu$ in the weak$\ast$ topology, and
\[
	\lim_{i\to \infty} h(\nu_i) = h(\mu).
\]
The same holds true with ergodic measures $\nu_i$ satisfying $\chi(\nu_i)=\alpha_i$ such that $\alpha_i<0$ and $\alpha_i\to0$.
\end{lemma}

A further result claims that given an ergodic measure
$\mu$ with exponent  $\chi(\mu)=\alpha>0$ and entropy $h(\mu)>0$, for every small $\beta<0$ there are ergodic measures with exponents close to $\beta$ and positive entropy, but in this construction some entropy is lost.  \cite[Theorem 5]{DiaGelRam:} bounds the amount of lost entropy that is related to the
size of $\alpha+\lvert\beta\rvert$. A specially interesting case occurs when the exponent $\beta$ is taken arbitrarily close to $0^-$. The estimates are summarized  in the next lemma.

\begin{lemma}[Rephrasing partially {\cite[Theorem 5]{DiaGelRam:}}] \label{lem:main3}
There exists $c>0$ such that for every ergodic measure $\mu$ with nonzero Lyapunov exponent $\chi (\mu)=\alpha\neq 0$ there is a sequence of ergodic measures $\nu_i$ with Lyapunov exponents $\chi (\nu_i) =\beta_i$,  $\sgn \alpha \neq \sgn \beta_i$, such that $\beta_i\to 0$ and
\[
	\lim_{i\to \infty} h(\nu_i) 
	\geq \frac {h(\mu)} {1+c|\alpha|}.
\]
\end{lemma}

This result also implies the following.

\begin{corollary}\label{cor:maisemenos}
	There exist ergodic measures with negative/positive exponents arbitrarily close to $0$.
\end{corollary}

The systems considered in this paper satisfy the so-called {\emph{skeleton property}}
which implies the existence of orbit pieces that allow to approximate entropy and Lyapunov exponent,
see \cite[Section 4]{DiaGelRam:} for details.
The skeleton property is referred to some blending interval  and to quantifiers corresponding to the
entropy and a level set for the Lyapunov exponent.
An important property is that if $\cL (\alpha)\ne 0$ then the skeleton property holds relative to
$h=h_{\mathrm{top}} (\cL(\alpha))$ and $\alpha$.
Based on the skeleton property, we have the following.

Given a compact $F$-invariant set $\Gamma\subset\Sigma_N\times\bS^1$, we say that $\Gamma$ has {\emph{uniform fiber expansion}} (\emph{contraction}) if every ergodic measure $\mu\in \cM_{\rm erg}(\Gamma)$ has a positive (a negative)  Lyapunov exponent. It is \emph{hyperbolic} if it either has uniform fiber expansion or uniform fiber contraction. We say that a set is \emph{basic} (with respect to $F$) if it is compact, $F$-invariant, locally maximal, topologically transitive, and hyperbolic%
\footnote{This definition mimics the usual definition of a basic set in a differentiable setting.}.

\begin{proposition}[{\cite[Theorems 4.3 and 4.4 and Proposition 4.8]{DiaGelRam:}}]\label{newprop:skeleton}
Given $\alpha\le0$ such that $\cL(\alpha)\ne\emptyset$ and $h=h_{\rm top}(\cL(\alpha))>0$, for every $\gamma\in(0,h)$ and every small $\lambda>0$ there is a basic set  $\Gamma\subset\Sigma_N\times\bS^1$ such that
\begin{enumerate}
\item
$h_{\rm top}({\Gamma})\in[ h-\gamma,h+\gamma]$ and
\item
every $\nu\in\cM_{\rm erg}({\Gamma})$ satisfies $\chi(\nu)\in(\alpha-\lambda,\alpha+\lambda)\cap\bR_-$. 
\end{enumerate}
The analogous result holds for any Lyapunov exponent $\alpha\ge0$.
\end{proposition}

A further consequence of the Axioms CEC$\pm$ and Acc$\pm$ is that the IFS $\{f_i\}$ is forward and backward minimal.
 \cite[Lemma~2.2]{DiaGelRam:} states a quantitative version of this minimality.
We also will use the following results which are  simple consequences of these axioms.

\begin{lemma}[{\cite[Lemmas 2.2 and 2.3]{DiaGelRam:}}]
\label{l:minimalcommon}
Every nontrivial interval $I\subset \bS^1$ contains a subinterval $J\subset I$  such that $F$ satisfies
Axioms CEC$\pm(J)$ and Acc$\pm(J)$.
Moreover, there is a number $M=M(I)\ge1$ such that for every point $x\in \bS^1$ there are finite sequences $(\theta_1\ldots \theta_r)$ and $(\beta_1\ldots \beta_s)$ with $r,s\le M$ such that
$$
f_{[\beta_1\ldots \beta_s]} (x)\in I \quad \mbox{and} \quad
f_{[\theta_1\ldots \theta_r.]} (x)\in I
$$
\end{lemma}

\begin{lemma}[{\cite[Lemma 2.4]{DiaGelRam:}}] \label{lem:connect}
For every interval $I\subset\bS^1$ there exist $\delta=\delta(I)>0$ and $M=M(I)\ge1$ such that for any interval $J\subset\bS^1$, $\lvert J\rvert<\delta$, there exists a finite sequence $(\tau_1\ldots\tau_m)$, $m\leq M$, such that $f_{[\tau_1\ldots\,\tau_m]}(J)\subset I$.
\end{lemma}

We finish this section with one further conclusion which we will use in Sections~\ref{subsec:maxent1} and~\ref{sec:synch}.

\begin{lemma}\label{lem:common}
	There does not exist a Borel probability measure $m$ on $\bS^1$ which is $f_i$-invariant for every $i=0,\ldots,N-1$.
\end{lemma}

\begin{proof}
By contradiction, assume that there is a Borel probability measure  $m$ on $\bS^1$ which is simultaneously $f_i$-invariant for all $i$.
	Let $J\subset\bS^1$ be a blending interval and consider two closed disjoint small sub-intervals $J_1,J_2\subset J$.
By Axiom CEC$+(J)$, there is some sequence $(\eta_0\ldots \eta_{\ell-1})$ such that $f_{[\eta_0\ldots \eta_{\ell-1}]}(J_1)\supset J$. From this we can conclude that $m(J\setminus J_1)=0$.  Similarly, $m(J\setminus J_2)=0$. This implies $m(J)=0$. 	
Hence, by Acc$\pm(J)$ we have that $m(\bS^1)=0$.
	But this is a contradiction.
\end{proof}

\section{Entropy, pressures, and variational principles}\label{sec:entpres}

In this section, we collect some general facts about entropy and pressure. We consider a  general setting of a compact metric space $(\mathbf X,d)$, a continuous map $F\colon \mathbf X\to \mathbf X$,
 and a continuous function $\varphi\colon \mathbf X\to\bR$.
 
\subsection{Entropy: restricted variational principles}\label{subsec:ent}

Given $\alpha\in\bR$ consider the level sets
\[
	\cL(\alpha)
	\eqdef \big\{x\in \mathbf X\colon
		\overline\varphi(x)=\alpha\big\},
	\quad\text{ where }\quad
		\overline\varphi(x)
		\eqdef\lim_{n\to\infty}\frac1n\sum_{k=0}^{n-1}\varphi(F^k(x)),
\]
whenever this limit exists.
We study the \emph{topological entropy of $F$} on the set $\cL(\alpha)$ and consider the function
\[
	\alpha\mapsto  h_{\rm top}(\cL(\alpha)).
\]

We will now recall some results which are  known for such general setting. An upper bound for the entropy
$ h_{\rm top}(\cL(\alpha))$ (which, in fact, is sharp in many cases) is easily derived  applying a general result by Bowen~\cite{Bow:73}. Denote by $\cM(\mathbf X)$ the set of all $F$-invariant probability measures and by $\cM_{\rm erg}(\mathbf X)\subset\cM(\mathbf X)$ the subset of ergodic measures. We equip this space with the weak$\ast$ topology. Given $x\in \mathbf X$, let $V_F(x)\subset \cM(\mathbf X)$ be the set of ($F$-invariant) measures which are weak$\ast$ limit points as $n\to\infty$ of the empirical measures $\mu_{x,n}$
\[
	\mu_{x,n}\eqdef \frac1n\sum_{k=0}^{n-1}\delta_{F^k(x)},
\]
where $\delta_x$ is the Dirac measure supported on the point $x$.
Given $\mu\in\cM(\mathbf X)$, denote  by $G(\mu)$ the set of \emph{$\mu$-generic points}
\[
	G(\mu)\eqdef
	\big\{x\colon \lim_{n\to\infty}\,\mu_{x,n}=\{\mu\}\big\}\,.
\]
Given $c\ge0$, define the set of its ``quasi regular'' points  by
\[
	QR(c)\eqdef
	\big\{y\in \mathbf X\colon \text{ there exists } \mu\in V_F(y)\text{ with }h(\mu)\le c\big\}.
\]
	
\begin{proposition}\label{pro:Bowen} $\,$
	\begin{itemize}
	\item[i)] $h_{\rm top}(QR(c))\le c$ (\cite[Theorem 2]{Bow:73}).
	\item[ii)] For $\mu$ ergodic we have $h(\mu)= h_{\rm top}(G(\mu))$ (\cite[Theorem 3]{Bow:73}).
	\item [iii)] If $F$ satisfies the specification property, then for every $\mu\in\cM(\mathbf X)$ we have $h(\mu)= h_{\rm top}(G(\mu))$ (\cite[Theorem 1.2]{PfiSul:07} or~\cite[Theorem 1.1]{FanLiaPey:08}).%
\footnote{Note that, in fact, this result holds true for any map which has the
so-called  {\emph{$g$-almost product property}} which is implied by the specification property (see~\cite[Proposition 2.1]{PfiSul:07}). The specification property is satisfied for example for every basic  set (see~\cite{Sig:74}). We emphasize that the skew-product systems we study in this paper do not satisfy the specification property.}
	\end{itemize}		
\end{proposition}

We have the following simple consequence.
Let
\[
	\varphi(\mu)
	\eqdef \int\varphi\,d\mu.
\]

\begin{lemma}\label{lem:morning}
	For every $\alpha$ such that $\cL(\alpha)\ne\emptyset$ we have
	\[
	\begin{split}
		\sup\big\{h(\mu)\colon \mu\in\cM_{\rm erg}(\mathbf X),\varphi(\mu)= \alpha\big\}
		&\le h_{\rm top}(\cL(\alpha))
		\\
		&\le
		\sup\big\{h(\mu)\colon \mu\in\cM(\mathbf X),\varphi(\mu)= \alpha\big\}.
	\end{split}
	\]
Moreover, for $\alpha=\sup\{\varphi(\mu)\colon\mu\in\cM_{\rm erg}(\mathbf X)\}$ we have
\[
	h_{\rm top}(\cL(\alpha))
	=\sup\big\{h(\mu)\colon \mu\in\cM_{\rm erg},\varphi(\mu)= \alpha\big\}.
\]	
Analogously for $\alpha=\inf\{\varphi(\mu)\colon\mu\in\cM_{\rm erg}(\mathbf X)\}$.
\end{lemma}

\begin{proof}
To prove the first inequality, observe that for $\mu$ ergodic with $\varphi(\mu)=\alpha$ we have $G(\mu)\subset\cL(\alpha)$ and by Proposition~\ref{pro:Bowen} ii) and monotonicity of topological entropy with respect to inclusion we obtain $h(\mu)=h_{\rm top}(G(\mu))\le h_{\rm top}(\cL(\alpha))$.
	
To prove the second inequality, denote
\[
	H(\alpha)
	\eqdef \sup\{h(\mu)\colon \mu\in\cM(\mathbf X),\varphi(\mu)= \alpha\}.
\]
Note that for every $x\in\cL(\alpha)$ we have $\overline\varphi(x)=\alpha$ and hence for every $\mu\in V_F(x)$ we have $\varphi(\mu)=\alpha$ and thus $h(\mu)\le H(\alpha)$. Hence, $\cL(\alpha)\subset QR(H(\alpha))$ and again by monotonicity and Proposition~\ref{pro:Bowen} i) we obtain
$$
h_{\rm top}(\cL(\alpha))\le h_{\rm top}(QR(H(\alpha)))\le H(\alpha),
$$
proving the first part of the lemma.

It remains to consider the extremal exponent $\alpha=\sup\{\varphi(\mu)\colon\mu\in\cM_{\rm erg}(\mathbf X)\}$. By the ergodic decomposition, any invariant measure with extremal exponent $\alpha$ has almost surely only ergodic measures with that exponent in its decomposition. Hence we have
\[
	h_{\rm top}(\cL(\alpha))
	\le
	 \sup\big\{h(\mu)\colon \mu\in\cM_{\rm erg},\varphi(\mu)= \alpha\big\},
\]	
ending the proof.
\end{proof}

We recall the following classical restricted variational principle
strengthening the above lemma which will play a central role in our arguments.
We point out that it requires $\varphi$ to be continuous, only.

\begin{proposition}[{\cite[Theorem 6.1 and Proposition 7.1]{PfiSul:07} or \cite[Theorem 1.3]{FanLiaPey:08}} and~\cite{Sig:74}]\label{prop:spcifcc}
	If $F\colon \mathbf X\to \mathbf X$ satisfies the specification property then for every $\alpha$ such that $\cL(\alpha)\ne\emptyset$ we have
\[
	h_{\rm top}(\cL(\alpha))
	=\sup\big\{h(\mu)\colon\mu\in\cM(\mathbf X),\varphi(\mu)=\alpha\big\}.
\]	
Moreover, $\big\{\varphi(\mu)\colon\mu\in\cM_{\rm erg}(\mathbf X)\big\}$ is  an interval.
\end{proposition}

\subsection{Pressure functions}
\label{ss.pressuref}

For a measure $\mu\in\cM(\mathbf X)$ we define the affine functional $P(\cdot,\mu)$ on the space of continuous functions by
\[
	P(\varphi,\mu)
	\eqdef h(\mu)+\int\varphi\,d\mu.
\]
Given an $F$-invariant compact subset $Y\subset \mathbf X$, we define the \emph{topological pressure of $\varphi$} with respect to $F|_Y$ by
\begin{equation}\label{eq:vp}
	P_{F|Y}(\varphi)
	\eqdef \sup_{\mu\in\cM(Y)}P(\varphi,\mu)
	= \sup_{\mu\in\cM_{\rm erg}(Y)}P(\varphi,\mu)
\end{equation}
and we simply write $P(\varphi)=P_{F|\mathbf X}(\varphi)$ if $Y=\mathbf X$ and $F|_{\mathbf X}$ is clear from the context.
Note that  definition and equality in~\eqref{eq:vp} are nothing but the \emph{variational principle} of the topological pressure (see~\cite[Chapter 9]{Wal:82} for a proof and a purely topological and equivalent definition of pressure).
A measure $\mu\in\cM(Y)$ is an \emph{equilibrium state} for $\varphi$ (with respect to $F|_Y$) if it realizes the supremum in~\eqref{eq:vp}.%
\footnote{Note that in the context of the rest of the paper, skew-product maps with one-dimensional fibers,
such equilibrium states indeed exist by~\cite[Corollary 1.5]{DiaFis:11} (see also~\cite{CowYou:05}). However, in a slightly different skew-product setting, they are not unique in general, see for instance the examples in~\cite{LepOliRio:11, DiaGel:12}.}
Recall that $h_{\rm top}(Y)=P_{F|Y}(0)$ is the \emph{topological entropy} of $F$ on $Y$.

We now continue by considering a decomposition of the set of ergodic measures and studying  corresponding pressure functions. Given a subset $\cN\subset\cM(\mathbf X)$, define
\[
	P(\varphi,\cN)
	\eqdef \sup_{\mu\in\cN}P(\varphi,\mu).
\]
Given  $\cN\subset\cM(\mathbf X)$, consider its \emph{closed convex hull} $\clocon \cN$, defined as the smallest closed convex set containing $\cN$. It is an immediate consequence of the affinity of $\mu\mapsto P(\varphi,\mu)$ that
\[
	P(\varphi,\cN)
	=P\big(\varphi,\clocon(\cN)\big).
\]
A particular consequence of this equality and the ergodic decomposition of non-ergodic measures is the fact that for $\cN=\cM_{\mathrm{erg}}(\mathbf X)$  and hence $\clocon(\cN)=\cM(\mathbf X)$
in~\eqref{eq:vp} it is irrelevant if we take the supremum over all measures in $\cM(\mathbf X)$ or over the \emph{ergodic} measures only (used to show the equality in~\eqref{eq:vp}). The
case of a general subset $\cN$ of $\cM(\mathbf X)$, however, will be quite different and is precisely our focus of interest.

We now analyze the pressure function for a subset of \emph{ergodic} measures $\cN\subset\cM_{\rm erg}(\mathbf X)$.%
\footnote{In the rest of this paper, we will study the decomposition~\eqref{eq:ergdecompo} and have in mind the particular subset of measures
$	\cM_{\rm erg,<0}$ and 
$	\cM_{\rm erg,>0}
$.
}
Let $q\in\bR$ and consider the parametrized family $q\varphi\colon \mathbf X\to\bR$ and the function
\[
	\cP_\cN(q)
	\eqdef P(q\varphi,\cN).
\]
For each $\mu\in\cN$ we simply write $\cP_\mu(q)=\cP(q,\{\mu\})$.  We call $\mu\in\cM(\mathbf X)$ an \emph{equilibrium state} for $q\varphi$, $q\in\bR$, (with respect to $\cN$) if $\cP_\cN(q)=\cP_\mu(q)$.
 Let also
\begin{equation}
\label{e.supinfvarphi}
	\varphi(\cN)
	\eqdef\Big\{\int\varphi\,d\mu\colon\mu\in\cN\Big\},\quad
	\underline\varphi_\cN
	\eqdef \inf\varphi(\cN)
	,\quad
	\overline\varphi_\cN
	\eqdef \sup\varphi(\cN).
\end{equation}

We list the following general properties which are easy to verify
(most of these properties and the ideas behind their proofs can be found in
\cite[Chapter 9]{Wal:82}).

\begin{enumerate}
\item[(P1)] The function $\cP_\mu$ is affine and satisfies $\cP_\mu\le\cP_\cN$ and $\cP_\mu(0)=h(\mu)$.
\item[(P2)] Given a subset $\cN'\subset\cN$,  then $\cP_{\cN'}\le\cP_\cN$.
\item[(P3)] $\cP_\cN(0)=\sup\{h(\mu)\colon\mu\in\cN\}$.
\item[(P4)] The function $\varphi\mapsto P(\varphi,\cN)$ is continuous and $q\mapsto P(q\varphi,\cN)$ is uniformly Lipschitz continuous.
\item[(P5)]  The function $\cP_\cN$ is convex.
 	Consequently, $\cP_\cN$ is differentiable at
all but at most countably many $q$'s and the left and right derivatives $D_L\cP_\cN(q)$ and
$D_R\cP_\cN(q)$ are defined for all $q\in\bR$.
\item[(P6)] We have
\[\begin{split}	
	\underline\varphi_\cN
	&=\lim_{q\to\infty}\frac{\cP_\cN(q)}{q}
	=\lim_{q\to\infty}D_L\cP_\cN(q)
	=\lim_{q\to\infty}D_R\cP_\cN(q),
	\\
	\overline\varphi_\cN
	&= \lim_{q\to-\infty}\frac{\cP_\cN(q)}{q}
	=\lim_{q\to-\infty}D_L\cP_\cN(q)
	=\lim_{q\to-\infty}D_R\cP_\cN(q).
\end{split}\]
\item[(P7)] The graph of $\cP_\cN$ has a supporting straight line of slope $\varphi(\mu)$ for every $\mu  \in \cN$. Thus,  for any $\alpha \in (\underline\varphi_\cN,\overline\varphi_\cN)$ it has a supporting straight line of slope $\alpha$.
\item[(P8)]   If the entropy map $\mu\mapsto h(\mu)$ is upper semi-continuous on $\cM(\mathbf X)$ then for any number $\alpha \in (\underline\varphi_\cN,\overline\varphi_\cN)$ there is a measure $\mu_\alpha\in\cM(\mathbf X)$ (not necessarily ergodic and not necessarily in $\cN$) such that $\varphi(\mu_\alpha) = \alpha$ and $q \mapsto \cP_{\mu_\alpha}(q)$ is a supporting straight line for $\cP_\cN$.
\item[(P9)] If $\mu\in\cM(\mathbf X)$ is an equilibrium state for $q\varphi$ for some $q\in\bR$ (with respect to $\cN$), then $D_L\cP_\cN(q) \leq \varphi(\mu) \leq D_R\cP_\cN(q)$. Moreover, the graph of $\cP_\mu$ is a supporting straight line for the graph of $\cP_\cN$ at $(q,\cP_\cN(q))$.
\item[(P10)] If the entropy map $\mu\mapsto h(\mu)$ is upper semi-continuous, then for any  $q$ there are equilibrium states $\mu_{L,q}$ and $\mu_{R,q}$  for $q\varphi$ (with respect to $\cN$) such that  $\varphi(\mu_{L,q}) = D_L\cP_\cN(q)$ and $\varphi(\mu_{R,q}) = D_R\cP_\cN(q)$. Moreover, $\mu_{L,q}$ and $\mu_{R,q}$ can be chosen to be ergodic (but not necessarily in $\cN$).
\item[(P11)] $\cP_\cN$ is differentiable at $q$ if and only if all equilibrium states for $q\varphi$ (with respect to $\cN$) have the same exponent and this exponent is $\cP'_\cN(q)$. In particular, if there is a unique  equilibrium state for $q\varphi$ (with respect to $\cN$) then $\cP_\cN$ is differentiable at $q$.
\item[(P12)] If $\mu\in\clocon(\cN)$ is not ergodic and $\cP_\mu(q) = \cP_\cN(q)$ for some $q$, then almost all measures in the  ergodic decomposition of $\mu$ are equilibrium states for  $q\varphi$ (with respect to $\cN$). 	
\end{enumerate}

\subsection{The convex conjugates of pressure functions}
\label{ss.convexconjugate}
One of our goals is to express the topological entropy $h_{\rm top}(\cL(\alpha))$ of each level set $\cL(\alpha)$ in terms of a restricted variational principle and in terms of a Legendre-Fenchel transform of an appropriate pressure function. Let us hence recall some simple facts about such transforms.

Given a subset of ergodic measures $\cN\subset\cM_{\rm erg}(\mathbf X)$, we define
\[
	\cE_\cN(\alpha)
	\eqdef \inf_{q\in\bR}\big(\cP_\cN(q)-q\alpha\big)
\]
on its domain
\[
	D(\cE_\cN)
	\eqdef \Big\{\alpha\in\bR\colon \inf_{q\in\bR}(\cP_\cN(q)-q\alpha)>-\infty\Big\}.
\]
Observe that $(\cP_\cN,\cE_\cN)$ forms a \emph{Legendre-Fenchel pair}.%
\footnote{The \emph{Legendre-Fenchel transform} of  a convex function $\beta\colon\bR\to\bR\cup\{\infty\}$ is defined by
\[
	\beta^\star(\alpha)\eqdef\sup_{q\in\bR}\big(\alpha q-\beta(q)\big),
\] and is convex on its domain $D(\beta^\star)=\{\alpha\in\bR\colon\beta^\star(\alpha)<\infty\}$. In particular, the convex function $\beta$ is differentiable at all but at most countably many points and
\[
	\beta^\star(\alpha)
	= \beta'(q)q-\beta(q)\quad\text{ for }\quad\alpha=\beta'(q).
\]
On the set of strictly convex functions the transform is involutive $\beta^{\star\star}=\beta$. Formally, it is the function $\alpha\mapsto-\cE_\cN(-\alpha)$ which is the Legendre-Fenchel transform of $\cP_\cN(q)$, but it is common practice  in the context of this paper (that we will also follow) to address $\cE_\cN$ by this name.}
We list the following general properties.
\begin{enumerate}
\item [(E1)] The function $\cE_\cN$ is concave (and hence continuous). Consequently, it is  differentiable at all but at most countably many $\alpha$, and the left and right derivatives are defined for all $\alpha\in D(\cE_{\cN})$.
\item [(E2)] We have
\[
	D(\cE_\cN)\supset (\underline\varphi_\cN,\overline\varphi_\cN).
\]
\item [(E3)] If $\mu$ is an equilibrium state for $q\varphi$ for some $q\in\bR$ (with respect to $\cN$) and $\alpha=\varphi(\mu)$, then $h(\mu)=\cE_\cN(\alpha)$.
\item [(E4)] We have
\[
	\max_{\alpha\in D(\cE_\cN)}\cE_\cN(\alpha)
	= \cP_\cN(0).
\]
Moreover, this maximum is attained at exactly one value of $\alpha$ if, and only if, $\cP_\cN$ is differentiable at $0$.
\item [(E5)] For every $\alpha\in D(\cE_\cN)$ we have
\[
	\cE_\cN(\alpha)
	\ge \sup\big\{h(\mu)\colon\mu\in\cN,\varphi(\mu)=\alpha\big\}.
\]
\end{enumerate}

For completeness, we give the short proof of (E5).
\begin{proof}[Proof of property (E5)]
Let $\alpha\in\interior D(\cE_\cN)$.
Fix any $q\in\bR$. Observe that
\[\begin{split}
	\sup\big\{h(\mu)\colon \mu\in\cN,\varphi(\mu)=\alpha\big\}
	&= \sup\big\{h(\mu)+q\varphi(\mu)\colon \mu\in\cN,\varphi(\mu)=\alpha\big\}	-q\alpha\\
	&\le \sup\big\{h(\mu)+q\varphi(\mu)\colon \mu\in\cN\big\}	-q\alpha\\
	&= \cP_{\cN}(q)-q\alpha.
\end{split}\]	
Since $q$ was arbitrary, we can conclude
\[\begin{split}
	\sup\big\{h(\mu)\colon \mu\in\cN,\varphi(\mu)=\alpha\big\}
	&\le \inf_{q\in\bR}\big(\cP_\cN(q)-q\alpha\big)	
	= \cE_\cN(\alpha)
\end{split}\]	
proving the property.
\end{proof}

\begin{proposition}\label{prolem:basresvarprinc}
	Assume that  $\mathbf X$ is a basic set of the skew-product map $F\colon\Sigma_N\times\bS^1\to\Sigma_N\times\bS^1$.
	Let $\varphi\colon \mathbf X\to\bR$ be a continuous potential. Then for  $\cN=\cM_{\rm erg}(\mathbf X)$ and every $\alpha\in\interior D(\cE_\cN)$ we have
\[
	\sup\big\{h(\mu)\colon \mu\in\cN,\varphi(\mu)=\alpha\big\}
	=\sup\big\{h(\mu)\colon \mu\in\cM(\mathbf X),\varphi(\mu)=\alpha\big\}
	= \cE_\cN(\alpha).	
\]	
\end{proposition}

Note that to show the inequality $\le$ in the proposition we, in fact, do not need the hypothesis of a basic set.

\begin{proof}
Let $\alpha\in\interior D(\cE_\cN)$.
Note that $\cN\subset\cM(\mathbf X)$, the above proof of (E5), and the fact that for $\cN=\cM_{\rm erg}(\mathbf X)$ we have $\cP_\cN(q)=\sup\{h(\mu)+q\varphi(\mu)\colon \mu\in\cM(\mathbf X)\}$ (see~\cite[Corollary 9.10.1 i)]{Wal:82}) immediately implies the inequalities $\le$.

It remains to prove the inequality $\cE_\cN(\alpha)\le \sup\{h(\mu)\colon\mu\in\cN,\varphi(\mu)=\alpha\}$ and hence the proposition. First recall \cite{Bow:08} that for any H\"older continuous potential $\tilde\varphi\colon \mathbf X\to\bR$ and $\tilde q\in\bR$ there is a unique equilibrium state for $\tilde q\tilde\varphi$ for a basic set of a diffeomorphism. Note that this hypothesis naturally translates to our skew-product setting.
By property (P8) applied to $\mathbf X$ and $\cN$, there is a measure $\mu_\alpha\in\cM(\mathbf X)$ (not necessarily ergodic) such that $\varphi(\mu_\alpha)=\alpha$ and $q\mapsto \cP_{\mu_\alpha}(q)$ is a supporting straight line for $\cP_\cN$. Hence, there is $q=q(\alpha)$ such that $\cP_\cN(q)=h(\mu_\alpha)+qh(\mu_\alpha)$.
If $\mu_\alpha$ was already ergodic then we are done.  Otherwise, note that we can find $\tilde\varphi\colon \mathbf X\to\bR$ H\"older continuous and arbitrarily close to the continuous potential $\varphi\colon \mathbf X\to\bR$ and  $\tilde q$ arbitrarily close to $q$ and an ergodic equilibrium state $\tilde\nu\in\cN$ for $\tilde q\tilde\varphi$ such that $\varphi(\tilde\nu)=\alpha$. By (P4) we have that $P(\tilde q\tilde\varphi,\cN)$ is arbitrarily close to $P(q \varphi,\cN)$.
Hence, for such $\tilde\nu$ we have
\[
	h(\tilde\nu)
	= P(\tilde q\tilde\varphi,\cN)-\tilde q\alpha
	= \big(P(q\varphi,\cN)- q\alpha\big) + \big(P(\tilde q\tilde\varphi,\cN)-P(q\varphi,\cN)\big)
		+  \big(q\alpha - \tilde q\alpha  \big) .
\]
Thus, we can conclude
\[
	\sup\big\{h(\nu)\colon \nu\in \cN,\varphi(\nu)=\alpha\big\}
	\ge \big(P(q\varphi,\cN)- q\alpha\big) .
\]
Taking the infimum over all $q\in\bR$ we obtain
\[
	\sup\big\{h(\nu)\colon \nu\in \cN,\varphi(\nu)=\alpha\big\}
	\ge \inf_{q\in\bR}\big(\cP_\cN(q)- q\alpha\big)
	=\cE_\cN(\alpha).
\]
This finishes the proof of the lemma.
\end{proof}

\section{Exhausting families}\label{sec:exhau}

In this section, we present a general principle to perform a multifractal analysis. It was already used in several  contexts having some hyperbolicity (see, for example,~\cite{GelRam:09} for Markov maps on the interval, \cite{GelPrzRam:10} for non-exceptional rational maps of the Riemann sphere, or \cite{BurGel:14} for geodesic flows of rank one surfaces). As the system as a whole does not satisfy the specification property,  we consider certain families of subsets (basic sets, see Section~\ref{sec:homrel}) on which we do have specification.
The general theory of \emph{restricted pressures} presented here allows us to obtain dynamical properties of the full system knowing the properties of those subsets.

\subsection{General framework}\label{sec:exhau}

Let $(\mathbf X,d)$ be a compact metric space, $F\colon \mathbf X\to \mathbf X$ a continuous map,
and $\varphi \colon \mathbf X\to \bR$ a continuous potential. Fix a
set of ergodic measures $\cN\subset\cM_{\rm erg}(\mathbf X)$.
Recall that we defined for $\alpha\in D(\cE_\cN)$
\[
	\cE_\cN(\alpha)
	\eqdef \inf_{q\in\bR}\big(\cP_\cN(q)-q\alpha\big).
\]

 A sequence of  compact $F$-invariant sets $\mathbf X_1, \mathbf X_2,\ldots\subset \mathbf X$  is said to be \emph{$(\mathbf X,\varphi,\cN)$-exhausting}  if the following holds: for every $i\ge1$ we have
\begin{itemize}
\item[(exh1)] $ \cM_{\rm erg}(\mathbf X_i)\subset \cN$,
\item[(exh2)]  $F|_{\mathbf X_i}$ has the specification property,
\item[(exh3)]
Given $\cM_i=\cM_{\rm erg}(\mathbf X_i)$
let $\cP_i=\cP_{\cM_i}$  and
\[
	\cE_i(\alpha)
	\eqdef \inf_{q\in\bR}\big(\cP_i(q)-q\alpha\big).
\]
Then for every $\alpha \in \interior D(\cE_i)$ the \emph{restricted variational principle} holds
\[
	\cE_i(\alpha)
	=\sup\big\{h(\mu)\colon \mu\in\cM_i,\varphi(\mu)=\alpha\big\}.
\]

\item[(exh4)]
for every $q\in\bR$ we have
\[
	\lim_{i\to\infty}P_{F|\mathbf X_i}(q\varphi)
	=\cP_\cN(q).
\]	
\item[(exh5)]
Let $\underline\varphi_\cN$ and $\overline\varphi_\cN$ be as in \eqref{e.supinfvarphi}, then
\[
	\underline\varphi_\cN=\lim_{i\to\infty}\underline\varphi_{\cM_i},\quad
	\overline\varphi_\cN=\lim_{i\to\infty}\overline\varphi_{\cM_i}.
\]
\end{itemize}

Note that $(\cP_i,\cE_i)$ forms a Legendre-Fenchel pair for every $i\ge1$.

The exhausting property for appropriate $\cN$ is the essential step to relate the lower bound in the restricted variational principle~\eqref{lem:morning} to the Legendre-Fenchel transform of the restricted pressure function $\cP_\cN$. This is the requirement (exh3).

\begin{lemma}\label{l.conjugated}
It holds $\lim_{i\to\infty}\cE_i(\alpha)=\cE_\cN(\alpha)$. In particular, $\interior D(\cE_\cN) = (\underline\varphi_\cN,\overline\varphi_\cN).$
\end{lemma}

\begin{proof}
Note that property (exh4) of pointwise convergence
of convex functions of pressures $\cP_i$ to the convex function
of pressure $\cP_\cN$ and the fact that
$\cE_i$ and $\cE_\cN$ are their Legendre-Fenchel transforms
 imply the claim, see for instance ~\cite{Wij:66}.
\end{proof}

The following result will be the main step in establishing the lower bounds for entropy in Theorem~\ref{main1}.
We derive it in the general setting of this subsection.

\begin{proposition}\label{pro:Climen}
Assume that there exists an increasing family of sets $(\mathbf X_i)_i\subset \mathbf X$ which is $(\mathbf X,\varphi,\cN)$-exhausting.
Then
\begin{itemize}
\item we have
\[
	(\underline\varphi_\cN,\overline\varphi_\cN)
	\subset \varphi(\cN)
	\subset [\underline\varphi_\cN,\overline\varphi_\cN].
\]	
In particular, $\varphi(\cN)$ is an interval.
\item
For every $\alpha\in(\underline\varphi_\cN,\overline\varphi_\cN)$ we have $\cL(\alpha)\ne\emptyset$ and
\[
	h_{\rm top}(\cL(\alpha))
	\ge 
	  \cE_\cN(\alpha)\\
	= \lim_{i\to\infty}\sup\big\{h(\mu)\colon\mu\in\cM_i
	,\varphi(\mu)=\alpha\big\}.
\]
\end{itemize}
\end{proposition}

\begin{proof}
By condition (exh4) and the property of pointwise convergence of convex functions to a convex function (see (P5)), we can conclude that for every $i$
\[
	P_{F|\mathbf X_{n(i)}}(q\varphi)\ge \cP_\cN(q)-\frac1i
\]
for all $q\in[-i,i]$ and some sequence $(n(i))_i$. For simplicity, allowing a change of indices, we will assume that $n(i)=i$.

 A particular consequence of  specification of $F|_{\mathbf X_i}$ is that by  Proposition~\ref{prop:spcifcc} the set $\varphi(\cM_i)$ is an interval.
 Together with (exh5) this implies that $\varphi(\cN)$ is an  interval and we have
 \begin{equation}\label{equa:spec}
 	(\underline\varphi_\cN,\overline\varphi_\cN)
	\subset \varphi(\cN)
	=\bigcup_{i\ge1}\varphi(\cM_i)
	\subset[\underline\varphi_\cN,\overline\varphi_\cN],
 \end{equation}
proving the first item.

Let $\alpha\in(\underline\varphi_\cN,\overline\varphi_\cN)$. For every index $i$, by  Proposition~\ref{prop:spcifcc},  we have
\[
	h_{\rm top}(\cL(\alpha)\cap \mathbf X_i)
	=\sup\big\{h(\mu)\colon\mu\in\cM_i,\varphi(\mu)=\alpha\big\}
	\le h_{\rm top}(\cL(\alpha)),
\]
where for the inequality we use monotonicity of entropy.
By~\eqref{equa:spec}, there is $i=i(\alpha)\ge1$ such that $\alpha\in\varphi(\cM_i)$ and, in particular, we have $\cL(\alpha)\ne\emptyset$.
By (exh3),  for every $\alpha\in(\underline\varphi_\cN,\overline\varphi_\cN)$ and $i$ sufficiently big, we have
$$
	\cE_i(\alpha)
	= \sup\big\{h(\mu)\colon\mu\in\cM_i,\varphi(\mu)=\alpha\big\}.
$$

By Lemma~\ref{l.conjugated} we have $\lim_{i\to\infty}\cE_i(\alpha)=\cE_\cN(\alpha)$, concluding  the proof of the proposition.
\end{proof}

\subsection{Existence of exhausting families in our setting}\label{sec:bridging}

In this section, we return to consider a transitive step skew-product map $F$ as in~\eqref{eq:fssp} whose fiber maps are $C^1$ and satisfies Axioms CEC$\pm$ and Acc$\pm$.
Recall  that the map $F$ has ergodic measures with negative/positive exponents arbitrarily close to $0$, see Corollary~\ref{cor:maisemenos}.
The goal of this section is to prove the following proposition.

\begin{proposition}\label{pro:estamosehaustos}
Consider the set of ergodic measures $\cN= \cM_{\rm{erg},<0}$ and the potential $\varphi\colon \Sigma_N\times \bS^1\to \bR$ in \eqref{def:potential}.
Then  there is a $(\Sigma_N\times \bS^1, \varphi, \cN)$-exhausting family consisting of nested basic sets and $\varphi(\cN)=[\alpha_{\min}, 0)$.

The analogous statement is true for $\cN= \cM_{\rm{erg},>0}$ with $\varphi(\cN)=(0, \alpha_{\max}].$
\end{proposition}

\subsubsection{Homoclinic relations}\label{sec:homrel}
We say that a periodic point of $F$ is \emph{hyperbolic} or a \emph{saddle}  if its (fiber) Lyapunov exponent is nonzero. In our partially hyperbolic setting with one-dimensional central bundle, there are only two possibilities: a saddle has either a negative or positive (fiber) Lyapunov exponent. We say that two saddles are of  \emph{the same type} if either both have negative exponents or both have positive exponents. Note that all  saddles in a basic set are of the same (contracting/expanding) type. We say that two basic sets are of the \emph{same type} if their saddles are of the same type.

Given a  saddle $P$ we define the \emph{stable} and \emph{unstable sets} of its orbit $\cO(P)$
by
$$
	W^{\mathrm{s}} (\cO(P))
	\eqdef \{ X \colon \lim_{n\to \infty} d(F^n(X), \cO(P))=0\},
$$
and
$$
	W^{\mathrm{u}} (\cO(P))
	\eqdef \{ X \colon \lim_{n\to \infty} d(F^{-n}(X), \cO(P))=0\},
$$
respectively.

We say that a point $X$ is a {\emph{homoclinic point}} of $P$ if
$X\in W^{\mathrm{s}} (\cO(P))\cap W^{\mathrm{u}} (\cO(P))$.
Two saddles $P$ and $Q$ of the same index are \emph{homoclinically related}  if
the stable and unstable sets of their orbits intersect cyclically, that is, if
$$
	W^{\mathrm{s}} (\cO(P))\cap W^{\mathrm{u}} (\cO(Q))
	\ne\emptyset \ne
	W^{\mathrm{s}} (\cO(Q))\cap W^{\mathrm{u}} (\cO(P)).
$$
In our context, homoclinic intersections behave the same as transverse homoclinic intersections in the differentiable setting.
As in the differentiable case, to be homoclinically related defines an equivalence relation on the set of saddles of $F$. The {\emph{homoclinic class}} of a saddle $P$, denoted by $H(P,F)$, is the closure of the set of saddles which are homoclinically related to $P$. A homoclinic class can be also defined as the closure of the homoclinic points of $P$. As in the differentiable case, a homoclinic class is an $F$-invariant and transitive set.%
\footnote{These assertions are folklore ones, details can be found, for instance,  in \cite[Section 3]{DiaEstRoc:16}.
Note that in our skew-product  context the standard transverse intersection condition between the invariant sets of the saddles in the definition of a homoclinic relation is not required  and does not make sense.  However,  since  the dynamics in the central direction is  non-critical (the fiber maps are diffeomorphisms and hence have no critical points)  the intersections between invariant sets of saddles of the same type behave as ``transverse" ones and the arguments in the differentiable setting can be translated to the skew-product setting (here the fact that the fiber direction is one-dimensional is essential).}

\begin{lemma}
	Any pair of saddles $P,Q\in\Sigma_N\times\bS^1$ of the same type are homoclinically related.
\end{lemma}

\begin{proof}
Let us assume that $P$ and $Q$ both have negative exponents.
The proof of the other case is analogous and  omitted.
Let $P=(\xi,p)$ and $Q=(\eta,q)$, where $\xi=(\xi_0\ldots\xi_{n-1})^\bZ$ and $\eta=(\eta_0\ldots\eta_{m-1})^\bZ$.
By hyperbolicity, there is $\delta>0$ such that
\[
	f_\xi^n\big([p-\delta,p+\delta]\big)\subset (p-\delta,p+\delta)
	\quad\text{ and }\quad
	f_\eta^m\big([q-\delta,q+\delta]\big)\subset (q-\delta,q+\delta)
\]	
and such  that those maps are uniformly contracting on those intervals.
This immediately implies that
\[
\begin{split}
&
[. (\xi_0\dots \xi_{n-1})^\bN  ] \times
[p-\delta,p+\delta]
\subset W^{\mathrm{s}} (\cO(P)),\\
&
[. (\eta_0\dots \eta_{m-1})^\bN  ] \times
[q-\delta,q+\delta]
\subset W^{\mathrm{s}} (\cO(Q)).
\end{split}
\]
Similarly we get
$$
[(\xi_0\dots \xi_{n-1})^{-\bN}. ] \times
\{p\}
\subset W^{\mathrm{u}} (\cO(P)),\quad
[(\eta_0\dots \eta_{m-1})^{-\bN}.  ] \times
\{q\}
\subset W^{\mathrm{u}} (\cO(Q)).
$$
By Lemma~\ref{l:minimalcommon} there are $(\beta_0\ldots\beta_s)$ and $(\gamma_0\ldots\gamma_r)$
such that
\[
	f_{[\beta_0\ldots\,\beta_s]}(q)\in (p-\delta, p+\delta)
	\quad
	\mbox{and}
	\quad
	f_{[\gamma_0\ldots\,\gamma_r]}(p)\in (q-\delta, q+\delta).
\]
By construction, this implies that
\[
\begin{split}
&\big( (\eta_0\dots \eta_{m-1})^{-\bN}. \beta_0\ldots\,\beta_s
(\xi_0\dots \xi_{n-1})^\bN, q \big) \in
 W^{\mathrm{u}} (\cO(Q)) \cap  W^{\mathrm{s}} (\cO(P)),\\
 & \big( (\xi_0\dots \xi_{n-1})^{-\bN}. \gamma_0\ldots\,\gamma_r
(\eta_0\dots \eta_{m-1})^\bN, p \big) \in
 W^{\mathrm{s}} (\cO(P)) \cap  W^{\mathrm{u}} (\cO(P)).
\end{split}
\]
This proves that $P$ and $Q$ are homoclinically related.
\end{proof}

\subsubsection{Existence of exhausting families: Proof of Proposition~\ref{pro:estamosehaustos}}

We recall the following well-known fact about homoclinically related basic sets. For a proof we refer to~\cite[Section 7.4.2]{Rob:95}, where the hypothesis of a basic set of a diffeomorphism naturally translates to our skew-product setting. 

\begin{lemma}[Bridging]\label{lem:bridge}
Consider two basic sets $\Gamma_1,\Gamma_2\subset \Sigma_N\times\bS^1$ of $F$ which are homoclinically related. Then there is a basic
set
 $\Gamma$ of $F$ containing $\Gamma_1\cup\Gamma_2$.
 In particular, for every continuous potential $\varphi$, we have

\[
	\max\big\{P_{F|\Gamma_1}(\varphi),P_{F|\Gamma_2}(\varphi)\big\}
	\le P_{F|\Gamma}(\varphi).
\]	
\end{lemma}

We will base our arguments also on the following result
that translates results of from Pesin-Katok theory to our setting.

\begin{lemma}\label{lem:PesKat}
	Let $\mu\in\cM_{\rm erg,<0}$ with $h=h(\mu)>0$ and $\alpha=\chi(\mu)<0$.
	
Then for every $\gamma\in(0,h)$ and every $\lambda\in(0,\alpha)$ there exists a  basic set $\Gamma=\Gamma(\gamma,\lambda)\subset\Sigma_N\times\bS^1$ such that for all $q\in\bR$ we have
\[
	P_{F|\Gamma}(q\varphi)\ge h(\mu)+q\int\varphi\,d\mu-\gamma-q\lambda.
\]

The analogous statement is true for $\cM_{\rm{erg},>0}$.
\end{lemma}

\begin{proof}
	By Proposition~\ref{newprop:skeleton},  there exists a  basic set $\Gamma$ such that $h_{\rm top}(\Gamma)\ge h-\gamma$ and that for every $\nu\in\cM_{\rm erg}(\Gamma)$ we have $\chi(\nu)\in(\alpha-\lambda,\alpha+\lambda)$.
The variational principle~\eqref{eq:vp} immediately implies the lemma.
\end{proof}

We are now prepared to prove Proposition~\ref{pro:estamosehaustos}.

\begin{proof}[Proof of Proposition~\ref{pro:estamosehaustos}]
We first construct an exhausting family  for $\cN=\cM_{\rm erg,<0}$. Given $i\ge1$, let us first construct a basic set $X_i$ of contracting type such that
\begin{equation}\label{eq:above}
	P_{F|X_i}(q\varphi)\ge \cP_\cN(q)-\frac1i
\end{equation}
for all $q\in[-i,i]$. By Lipschitz continuity property (P4) of pressure,
there are a Lipschitz constant  $\Lip$ and a finite subset $q_1,\dots, q_\ell$ of $[-i,i]$ such that
for every $q\in [-i,i]$ there is $q_k$ with
$$
\Lip |q_k-q| \lVert \varphi \rVert < \frac{1}{4i}.
$$
To prove \eqref{eq:above}, given $q_k$, by Lemma~\ref{lem:PesKat} there is a basic set $X_{i,k}$ such that
\[
	P_{F|X_{i,k}}(q_k \varphi)\ge \cP_\cN(q_k)-\frac{1}{4i}.
\]
Applying Lemma~\ref{lem:bridge} consecutively to the finitely many basic sets $X_{i,1},\ldots,X_{i,\ell}$, we obtain a basic set $X_i$ containing all these sets and satisfying~\eqref{eq:above}.  This shows (exh4) and (exh5).

By construction, all basic sets are of contracting type and hence all ergodic measures have negative  Lyapunov exponent and we have (exh1). Each of them clearly satisfies (exh2) (basic sets have the specification property~\cite{Sig:74}). By Proposition~\ref{prolem:basresvarprinc} we have the restricted variational principle (exh3) on each of them.

What remains to prove is that $\varphi(\cN)=[\alpha_{\rm min},0)$.
By Corollary~\ref{cor:maisemenos}, the Lyapunov exponents of ergodic measures extend all the way to $0$, that is, $\overline\varphi_\cN=0$. On the other hand, note that by (P5) we can choose an increasing sequence $(q_j)_j$ tending to $-\infty$ such that $\cP_\cN$ is differentiable at all such $q_j$. By (P11) and (P12) for every $j$ there is an ergodic equilibrium state $\mu_j$ for $q_j\varphi$ and $\varphi(\mu_j)\to\underline\varphi_\cN$. Taking $\mu'$ which is a weak$\ast$ limit of $(\mu_j)_j$ as $j\to\infty$, then there is an ergodic measure $\mu''$ in its ergodic decomposition such that $\varphi(\mu'')=\underline\varphi_\cN$. In particular, we can conclude $\cL(\underline\varphi_\cN)\ne\emptyset$ and $\alpha_{\rm min}=\underline\varphi_\cN$. This concludes the proof that $\varphi(\cN)=[\alpha_{\rm min},0)$.

The statement for $\cN=\cM_{{\rm erg},>0}$ is proved  analogously.

The proof of the proposition is now complete.
\end{proof}

\section{Entropy of the level sets: Proof of Theorem~\ref{main1}}\label{sec:proofmain1}
In this section, we collect the ingredients required to prove Theorem~\ref{main1}. Section~\ref{subsec:maxent1} deals with the measures of maximal entropy. Section~\ref{sec:71} provides upper bounds for the entropy of level sets with exponents of the interior of the spectrum. 
Section \ref{ss.lowerbounds} deals with lower bounds. Sections~\ref{ss.coincidence} and~\ref{ss.upperbounds} deal with the boundary of the spectrum and with exponent zero. Here the main technical result is Theorem~\ref{theoprop:zero} whose proof will be postponed to Section~\ref{sec:proofoflowerbound}. The proof of Theorem~\ref{main1} is concluded in Section \ref{ss.endoftheproofofmain}.

\subsection{Measure(s) of maximal entropy}\label{subsec:maxent1}

Note that any measure of maximal entropy projects to the $(1/N,\ldots,1/N)$-Bernoulli measure in the base. Hence, we can use the known results about the behavior of Bernoulli measures for random dynamical systems.
By  \cite[Theorem 8.6]{Cra:90}  (stated for products of independently and identically distributed (i.i.d.) diffeomorphisms on a compact manifold)  for every Bernoulli measure $\mathfrak b$ in $\cM(\Sigma_N)$ there exists a (at least one) $F$-ergodic measure $\mu_+^{\mathfrak b}$ with positive exponent and a (at least one)
$F$-ergodic measure $\mu_-^{\mathfrak b}$ with negative exponent, both projecting to ${\mathfrak b}=\pi_\ast\mu_\pm^{\mathfrak b}$. Indeed, note that our axioms rule out the possibility of a measure being simultaneously preserved by all the fiber maps, see Lemma~\ref{lem:common}. When $\mathfrak b$ is the $(1/N,\ldots,1/N)$-Bernoulli measure  we simply write $\mu_\pm$.

There are various ways to prove that there are only finitely many hyperbolic ergodic $F$-invariant measures projecting to the same Bernoulli measure. For example, in our setting it is a consequence of~\cite[Theorem 1]{RodRodTahUre:12}.

\subsection{Negative/positive exponents in the interior of the spectrum}\label{sec:71}

We will analyze the negative part of the spectrum, the analysis of the positive part is analogous and it will be omitted.

By Proposition~\ref{pro:estamosehaustos} there is a $(\Sigma_N\times\bS^1,\varphi,\cM_{{\rm erg},<0})$-exhausting family $\{\mathbf X_i\}_i$. Hence, in particular, for every $\alpha\in(\alpha_{\rm min},0)$ we have $\cL(\alpha)\ne\emptyset$ and together with Proposition~\ref{pro:Climen} and writing $\varphi(\mu)=\chi(\mu)$ we have
\[
	h_{\rm top}(\cL(\alpha))
	\ge   \cE_\cN(\alpha)\\
	= \lim_{i\to\infty}\sup\big\{h(\mu)\colon\mu\in\cM(X_i),\chi(\mu)=\alpha\big\}.
\]
By Lemma~\ref{lem:morning}, we have
\[
	h_{\rm top}(\cL(\alpha))
	\ge \sup\big\{h(\mu)\colon \mu\in\cM_{\rm erg}
	,\chi(\mu)= \alpha\big\}.
\]

\begin{lemma}\label{lem:below}
	For every $\alpha\in(\alpha_{\min},0)$ we have $h_{\rm top}(\cL(\alpha))= \cE_{<0}(\alpha)$. 
\end{lemma}

\begin{proof}
By Proposition~\ref{pro:Climen}, we already have $h_{\rm top}(\cL(\alpha))\ge\cE_{<0}(\alpha)$ and it is hence enough to prove the other inequality.
Recall that by (E5) for every $\alpha<0$ we have
\begin{equation}\label{eq:levent}
	\cE_{<0}(\alpha)
	\ge \sup\big\{h(\mu)\colon\mu\in\cM_{\rm{erg},<0},\chi(\mu)=\alpha\big\}.
\end{equation}

	Arguing by contradiction,  let us assume that there are $\alpha\in(\alpha_{\min},0)$ and $\delta>0$ so that
\[
	h_{\rm top}(\cL(\alpha))
	\ge \cE_{<0}(\alpha)+2\delta.
\]	
Then, by continuity of $\cE_{<0}(\cdot)$, property (E1), there exists $\varepsilon>0$ such that for every $\alpha'\in(\alpha-2\varepsilon,\alpha+2\varepsilon)$ we have
\[
	h_{\rm top}(\cL(\alpha))
	\ge \cE_{<0}(\alpha')+\delta.
\]	
By Proposition~\ref{newprop:skeleton}, there exists a basic set $\Gamma\subset\Sigma_N\times\bS^1$ such that
\[
	h_{\rm top}(\Gamma)
	> h_{\rm top}(\cL(\alpha))-\delta,
\]
and that for every $\nu\in\cM_{\rm erg}(\Gamma)$ we have $\chi(\nu)\in(\alpha-\varepsilon,\alpha+\varepsilon)$.
Taking the measure of maximal entropy $\nu\in\cM_{\rm erg}(\Gamma)$, with the above, for every $\alpha'\in(\alpha-2\varepsilon,\alpha+2\varepsilon)$ we have
\[
	h(\nu)
	= h_{\rm top}(\Gamma)
	>  \cE_{>0}(\alpha').
\]
However, $\alpha'=\chi(\nu)\in (\alpha-\varepsilon,\alpha+\varepsilon)$ would then contradict~\eqref{eq:levent}. This proves the lemma.
\end{proof}

\begin{lemma}\label{lem:keyvarprinc}
For every $\alpha\in(\alpha_{\rm min},0)$ we have
\[
	h_{\rm top}(\cL(\alpha))
	= \cE_{<0}(\alpha)
	= \sup\{h(\mu)\colon\mu\in\cM_{\rm erg}\colon\chi(\mu)=\alpha\}.
\]
\end{lemma}

\begin{proof}
By Lemma~\ref{lem:below} we have $h_{\rm top}(\cL(\alpha))=  \cE_{<0}(\alpha)$. 
With Lemma~\ref{lem:morning}, what remains to show is that 
\[
	\cE_{<0}(\alpha)\le \sup\{h(\mu)\colon\mu\in\cM_{\rm erg},\chi(\mu)=\alpha\}.
\]	
By contradiction, assume that there is $\alpha$ such that $\cE_{<0}(\alpha)> \sup\{h(\mu)\colon\mu\in\cM_{\rm erg},\chi(\mu)=\alpha\}$ and let $\delta>0$ such that for every $\mu\in\cM_{\rm erg}$ with $\chi(\mu)=\alpha$ we have $\cE_{<0}(\alpha)-3\delta>h(\mu)$. By property (exh5) of the exhausting family $\{\mathbf X_i\}_i$, there exists $i_0\ge1$ such that for every $i\ge i_0$ we have that $\alpha\in(\underline\varphi_{\cM_i},\overline\varphi_{\cM_i})$. With Lemma~\ref{l.conjugated}, we also can assume that for every $i\ge i_0$ we have $\cE_i(\alpha)\ge\cE_{<0}(\alpha)-\delta$. Applying Proposition~\ref{prolem:basresvarprinc} to a basic set $\mathbf X_i$, there exists $\mu\in\cM_i\subset\cM_{\rm erg}$ with $\chi(\mu)=\alpha$ satisfying $h(\mu)\ge \cE_i(\alpha)-\delta$ and hence $h(\mu)\ge \cE_{<0}(\alpha)-2\delta$, a contradiction.
\end{proof}

\subsection{Coincidence of one-sided limits of the spectrum at zero}
\label{ss.coincidence}

\begin{lemma} \label{lem:welldefined}
$h_0\eqdef \lim_{\alpha \searrow 0} h_{\rm top}(\cL(\alpha))
	= \lim_{\alpha \nearrow 0} h_{\rm top}(\cL(\alpha))$.
\end{lemma}	
	
\begin{proof}
By Lemma \ref{lem:below}, for $\alpha\in ( \alpha_{\min},0)$
we have  $h_{\rm top}(\cL(\alpha))=\cE_{<0}(\alpha)$. Hence, by (E1), this is a concave function in $\alpha$. Similarly for   $\alpha\in (0,\alpha_{\max})$. So we can define the numbers $h_0^\pm =\lim_{\alpha\to 0\pm} h_{\rm top}(\cL(\alpha))$.

By the restricted variational principle in Lemma \ref{lem:keyvarprinc}, for every sequence $\alpha_k\nearrow 0$ there is a sequence of ergodic measures $(\mu_k)_{k\ge0}$
such that $\chi(\mu_k)=\alpha_k$ and $h(\mu_k)\to h_0^-$.
 As a consequence of Lemma~\ref{lem:main3} there is a corresponding sequence
 $(\nu_{k})_{k\ge1}$ with  $\chi(\nu_{k})\searrow 0$ and $h(\nu_{k})\to h_0^-$. This implies that $h_0^-\le h_0^+$.
 Reversing the roles of the  negative and positive exponents we get $h_0^-\le h_0^+$ and hence
 $h_0^-=h_0^+$, proving the lemma.
\end{proof}

\subsection{Zero and extremal exponents:  Upper bounds}
\label{ss.upperbounds}

\begin{lemma}\label{lem:somenumber}
	$ h_{\rm top}(\cL(0))\le h_0$.
\end{lemma}

\begin{proof}
	As a consequence of Proposition~\ref{newprop:skeleton} together with Lemma~\ref{lem:welldefined}, for every $\gamma>0$ and $\lambda >0$ there exists $\alpha \in (-\lambda,0)$ such that $h_{\rm top}(\cL(\alpha)) \geq h_{\rm top} (\cL(0)) - \gamma$. The assertion then follows.
\end{proof}

\begin{lemma}\label{lem:someothernumber}
	For $\alpha\in\{\alpha_{\min},\alpha_{\max}\}$ we have $h_{\rm top}(\cL(\alpha))\le \lim_{\beta\to\alpha}h_{\rm top}(\cL(\beta))$ and 
\[
	h_{\rm top}(\cL(\alpha))
	= \sup\{h(\mu)\colon\mu\in\cM_{\rm erg},\chi(\mu)=\alpha\}.
\]	
\end{lemma}

\begin{proof}
	We consider $\alpha=\alpha_{\min}$, the other case is analogous. By Lemma \ref{lem:below}, for every $\beta\in(\alpha_{\min},0)$ we have already $h_{\rm top}(\cL(\beta))=\cE_{<0}(\beta)$.
	By the second part in Lemma~\ref{lem:morning}, we have
\[
	h_{\rm top}(\cL(\alpha_{\min}))
	 =\sup\big\{h(\mu)\colon \mu\in\cM_{\rm erg},\varphi(\mu)= \alpha_{\min}\big\},
\]	
proving the second assertion in the lemma.
Hence, for every $q\in\bR$ we have
\[
	\cP_{<0}(q)
	=\sup\{h(\mu)+q\chi(\mu)\colon\mu\in\cM_{\rm erg,<0}\}
	\ge h_{\rm top}(\cL(\alpha_{\min})) 
	+q\alpha_{\min},
\]
which implies
\[
	\cE_{<0}(\alpha_{\rm min})
	= \inf_{q\in\bR}\big(\cP_{<0}(q)-q\alpha_{\min}\big)
	\ge h_{\rm top}(\cL(\alpha_{\min})) .
\]
By the definition of $\cE_{<0}$ in~\eqref{eq:defpres} and~\cite{Wij:66},  the left hand side is not larger than $ \lim_{\beta\to\alpha_{\min}}\cE_{<0}(\beta)$, proving the first assertion.
\end{proof}

\subsection{Whole spectrum: Lower bounds}
\label{ss.lowerbounds}

The following result is the final step needed to complete the proof of Theorem~\ref{main1}.
We postpone its proof to Section~\ref{sec:proofoflowerbound}.

\begin{theorem} \label{theoprop:zero}
For every $\alpha \in [\alpha_{\min}, \alpha_{\max}]$ we have
\[
	\limsup\limits_{\beta\to\alpha} \cE(\beta)\le h_{\rm top}(\cL(\alpha)).
\]		
\end{theorem}

\begin{remark}
For $\alpha\in(\alpha_{\min},\alpha_{\max})\setminus\{0\}$ the result of the above theorem follows already from Section~\ref{sec:71}. For $\alpha\in\{\alpha_{\min},\alpha_{\max}\}$ this result can be easily obtained by the following arguments: Take the weak$\ast$ limit $\mu$ of a sequence of measures $(\mu_k)_k$ converging in exponent to $\alpha$ and in entropy to $h=\limsup_{\beta\to\alpha} \cE(\beta)$. Indeed, such sequences exist by the already obtained description in the interior of the spectrum. The ergodic decomposition of $\mu$ contains an ergodic measure $\mu'$ with exponent $\alpha$ and entropy at least $h$. The set of $\mu'$-generic points is contained in $\cL(\alpha)$ which will imply the assertion.
So, we only need to prove Theorem~\ref{theoprop:zero} for $\alpha=0$. However, the proof is completely general.
\end{remark}

\subsection{Proof of Theorem~\ref{main1}}
\label{ss.endoftheproofofmain}

By the arguments in Section~\ref{sec:71}, for every $\alpha\in(\alpha_{\rm min},0)$ we have $\cL(\alpha)\ne\emptyset$, analogously for $\alpha\in(0,\alpha_{\rm max})$. The fact that $\cL(0)\ne\emptyset$ is a consequence of \cite{BocBonDia:16}. Note that by Theorem \ref{main3} item j) we have $\lim_{\alpha\to0}\cE(\alpha)>0$ and hence, by Theorem~\ref{theoprop:zero}, we obtain $h_{\rm top}(\cL(0))>0$ and, in particular, $\cL(0)\ne\emptyset$.  This already proves item d) of the theorem.
For any $\alpha\in\{\alpha_{\min},\alpha_{\max}\}$ take the weak$\ast$ limit $\mu$ of a sequence of measures $(\mu_k)_k$ converging in exponent to $\alpha$. Indeed, such sequences exist by the already obtained description in the interior of the spectrum. The ergodic decomposition of $\mu$ contains an ergodic measure $\mu'$ with exponent $\alpha$. This implies that $\cL(\alpha)\ne\emptyset$. 

Item a) (and analogously item b)) follows from Lemmas~\ref{lem:below} and~\ref{lem:keyvarprinc} for the interior of the spectrum. The restricted variational principle for $\alpha_{\rm min}$ and $\alpha_{\rm max}$ follows from Lemma~\ref{lem:morning}. 

To prove item c), for $\alpha\in\{\alpha_{\min},0,\alpha_{\max}\}$, again by Theorem~\ref{theoprop:zero} we have
\[
	\limsup\limits_{\beta\to\alpha} \cE(\beta)
	= \limsup\limits_{\beta\to\alpha} h_{\rm top}(\cL(\beta))\le h_{\rm top}(\cL(\alpha)).
\]		
Now if $\alpha=0$, the existence of the limit and the equality is a consequence of Lemmas~\ref{lem:welldefined} and~\ref{lem:somenumber}. If $\alpha\in\{\alpha_{\rm min},\alpha_{\rm max}\}$, then we apply Lemma~\ref{lem:someothernumber} instead proving item c).

Finally, the existence of finitely many ergodic measures of maximal entropy $\log N$ follows by the arguments in Section~\ref{subsec:maxent1}.

\section{Orbitwise approach and bridging measures: Proof of Theorem~\ref{theoprop:zero}}\label{sec:proofoflowerbound}

In this section, we  prove Theorem~\ref{theoprop:zero} which goes as follows. In Section~\ref{ss:apprxo} we identify a set of orbits with appropriate properties (``cardinality" and ``finite-time Lyapunov exponents" in Proposition~\ref{prolem:fluminense}). Based on those orbits, we construct a subset of $\pi(\cL(\alpha))$ in Section~\ref{sec:conlarsubset}. After some preliminary estimates in Section~\ref{subsec:prelimste}, in Section~\ref{subsec:estoftheentr} we show that this subset is ``large" by estimating its entropy following the approach of ``bridging measures". The proof of the theorem will be completed in Section~\ref{sec:endofproporrr}.

\subsection{Orbitwise approximation of ergodic measures}\label{ss:apprxo}

We will consider sets $\Xi^+(n)$ of finite sequences of length $n$
whose cardinalities grow exponentially fast. The following proposition provides precise estimates of how fast initial sequences  are  ``branching out'' to form  sequences in $\Xi^+(n)$. Given $1 \le \ell \le n$ we say $(\xi_0\ldots \xi_{\ell-1})$ is an 
{\emph{initial sequence of
length $\ell$}} of $\Xi^+(n)$ if there is $(\xi_{\ell}\ldots\xi_{n-1})$ such that  
 $(\xi_0\ldots \xi_{\ell-1} \xi_{\ell}\ldots\xi_{n-1}) \in \Xi^+(n)$.

\begin{proposition}\label{prolem:fluminense}
Given any $\mu\in\cM_{\rm{erg}}$, for every small $\varepsilon >0$ there exists a constant $K=K(\mu,\varepsilon)>1$ such that for every $n\ge1$ there exists
a set $\Xi^+(n)\subset\{0,\ldots,N-1\}^n$ of finite sequences of length $n$ satisfying:
\begin{itemize}
\item[1)] (Cardinality) The set $\Xi^+(n)$ has cardinality
$$
	K^{-1}e^{n(h(\mu) -\varepsilon)}
	\le \card (\Xi^+(n))\le
	Ke^{n(h(\mu) +\varepsilon)}	
$$
such that for every $j\in\{0,\ldots,n\}$ and  $\ell\in\{j,\ldots, n\}$ there exist at least 
\[
	\frac13K^{-2} e^{-2(\ell-j)\varepsilon} e^{j(h(\mu)-\varepsilon)}
\]
	 initial sequences of length $j$ of $\Xi^+(n)$ such that  each of them has between 
\[
	 \frac 13 K^{-2} e^{-2j\varepsilon} e^{(\ell-j)(h(\mu)-\varepsilon)}
	\quad \text{ and }\quad
	 K e^{(\ell-j)(h(\mu)+\varepsilon)}
\]
	  continuations to an initial sequence of length $\ell$ of $\Xi^+(n)$.
\item[2)] (Finite-time Lyapunov exponents)
There exists an interval $I=I(\mu,\varepsilon)\subset\bS^1$ such that for every $n\ge1$ for each $(\rho_0\ldots\rho_{n-1})\in \Xi^+(n)$ and each point $x\in I$ for every  $j\in\{1,\ldots,n\}$ we have
\begin{equation}\label{eq:claimedprop}
	K^{-1} e^{j(\chi(\mu) - \varepsilon)}
	\leq \lvert(f_{[\rho_0\ldots\,\rho_{j-1}]})'(x)\rvert
	\leq K e^{j(\chi(\mu) + \varepsilon)}.
\end{equation}
\end{itemize}
\end{proposition}

\begin{proof}
The  results are consequences of ergodicity, the definition of the Lyapunov exponent and the Brin-Katok, the Birkhoff ergodic, the Shannon-McMillan-Breiman, and the Egorov theorems (see also~\cite[Proposition 3.1]{DiaGelRam:}).  For completeness we provide the details. 

Given $\varepsilon_E\in(0,\chi(\mu)/2)$, $\varepsilon_H>0$, and $\kappa\in(0,1/6)$, there are a constant $K>1$ and a set $\Lambda\subset \Sigma_N\times\bS^1$ such that $\mu(\Lambda)>1-\kappa$ and for every $X=(\rho,x)\in\Lambda$ for every $n\ge1$ we have
\begin{equation}\label{eq:understand}
	K^{-1} e^{n(\chi(\mu) - \varepsilon_E/2)}
	\leq \lvert(f_{[\rho_0\ldots\,\rho_{n-1}]})'(x)\rvert
	\leq K e^{n(\chi(\mu) + \varepsilon_E/2)}.
\end{equation}
Let $S\eqdef\pi(\Lambda)\subset\Sigma_N$ and $\nu\eqdef\pi_\ast\mu$, where $\pi\colon\Sigma_N\times\bS^1\to\Sigma_N$ denotes the natural projection. Note that $\nu$ is $\sigma$-invariant  and ergodic, $\nu(S)\ge 1-\kappa$, and we can also assume that for every $\pi(\rho)\in S$ and every $n\ge1$
\begin{equation}\label{eq:understandc}
	K^{-1}e^{-n(h(\mu)+\varepsilon_H)}
	\le \nu([\rho_0\ldots\rho_{n-1}])
	\le Ke^{-n(h(\mu)-\varepsilon_H)}.
\end{equation}
For every $n\ge1$, define 
\[
	\Xi^+(n)
	\eqdef \{(\rho_0\ldots\rho_{n-1})\in\{0,\ldots,N-1\}^n\colon 
		[\rho_0\ldots\rho_{n-1}]\cap S\ne\emptyset\}.
\]	
Note that  by~\eqref{eq:understandc} (notice that those cylinders are pairwise disjoint and cover a set of measure at least $1-\kappa$) this set has cardinality $M_n$ bounded by
\[
	(1-\kappa)\cdot K^{-1}e^{n(h(\mu)-\varepsilon_H)}
	\le M_n
	\le Ke^{n(h(\mu)+\varepsilon_H)},
\]
which proves the first assertion of item 1).

To prove the second assertion of item 1), fix  positive integers $j,\ell$ such that $j\le \ell\le n$. 
Let $S(j)\eqdef S\cap\sigma^{-j}(S)$. Note that $\nu(S)>1-\kappa$ implies $1-2\kappa<\nu(S(j))\le1$. Observe that $S(j)$ consists of sequences $\xi=(\ldots\xi_{-1}.\xi_0\xi_1\ldots)$ which by~\eqref{eq:understandc} have the property that the initial $j$-cylinder $[\xi_0\ldots\xi_{j-1}]$ satisfies
\[
	K^{-1}e^{-j(h(\mu)+\varepsilon_H)}
	\le \nu([\xi_0\ldots\xi_{j-1}])
	\le Ke^{-j(h(\mu)-\varepsilon_H)}
\]	
 and that (because $\sigma^j(\xi)\in S$) has the property that its $(\ell-j)$-cylinder $[\xi_j\ldots\xi_{\ell-1}]$ satisfies
 \[
 	K^{-1}e^{-(\ell-j)(h(\mu)+\varepsilon_H)}
	\le \nu([\xi_j\ldots\xi_{\ell-1}])
	\le Ke^{-(\ell-j)(h(\mu)-\varepsilon_H)}.
\]	
These estimates have the consequence that we can choose a subset of $S(j)$ of sequences which is $(\ell-j,1)$-separated  (with respect to $\sigma$) and whose cardinality $M_{\ell-j}$ is bounded between
\begin{equation}\label{eq:flamengo}
	(1-2\kappa)\cdot K^{-1}e^{(\ell-j)(h(\mu)-\varepsilon_H)}
	\le M_{\ell-j}
	\le Ke^{(\ell-j)(h(\mu)+\varepsilon_H)}.
\end{equation}

Arguing in the same way for $j$ instead of $\ell-j$, we get a subset of $S(j)$ of sequences which is $(j,1)$-separated and whose cardinality $M_j$ is bounded between
\[
	(1-2\kappa)\cdot K^{-1}e^{j(h(\mu)-\varepsilon_H)}
	\le M_j
	\le Ke^{j(h(\mu)+\varepsilon_H)}.
\]
Note that to any such sequence with a fixed initial $j$-cylinder $[\xi_0\ldots\xi_{j-1}]\in S(j)$ there are at most $M_{\ell-j}$ continuations to a sequence with a corresponding $\ell$-cylinder $[\xi_0\ldots\xi_{j-1}\xi_j\ldots\xi_{\ell-1}]$ which has a continuation to some sequence in $S(j)$.

Assume now, by contradiction, that among those sequences with initial $j$-cylinder there are only less than 
\[
	M_j'\eqdef\frac13K^{-2} e^{-2(\ell-j)\varepsilon_H} e^{j(h(\mu)-\varepsilon_H)}
\]	 
of them which have more than $\frac 13 K^{-2} e^{-2j\varepsilon_H} e^{(\ell-j)(h(\mu)-\varepsilon_H)}$ continuations to a $\ell$-cylinder.  
Note that $M_j$ gives  a simple  estimate from above for the number of the remaining initial cylinders. Hence, the total number of $(\ell,1)$-separated sequences in $S(j)$ would be bounded from above by
\[\begin{split}
	M_j\,\cdot\,& \frac 13 K^{-2} e^{-2j\varepsilon_H} e^{(\ell-j)(h(\mu)-\varepsilon_H)}
	+ M_j'\cdot M_{\ell-j}\\
	&\le  K e^{j(h(\mu)+\varepsilon_H)} 
		\cdot \frac 13 K^{-2} e^{-2j\varepsilon_H} e^{(\ell-j)(h(\mu)-\varepsilon_H)}\\
	&\phantom{\le}	+ \frac13K^{-2} e^{-2(\ell-j)\varepsilon_H} e^{j(h(\mu)-\varepsilon_H)}\cdot Ke^{(\ell-j)(h(\mu)+\varepsilon_H)}\\
	&=\frac23K^{-1}e^{\ell(h(\mu)-\varepsilon_H)}.
\end{split}\]
However, from~\eqref{eq:flamengo} with $j=0$, there are at least $M_\ell\ge (1-2\varkappa)K^{-1}e^{\ell(h(\mu)-\varepsilon_H)}$ such sequences in $S(j)$, which leads us to a contradiction. Together with the upper estimate in~\eqref{eq:flamengo}, this ends the proof of  item 1).

To show item 2), what remains is to prove that~\eqref{eq:claimedprop} holds for \emph{any} point in $I$ and not only for a point $x$ such that $(\rho,x)\in \Lambda$ as in~\eqref{eq:understand}.
As any $X=(\rho,x)\in\Lambda$ is a point with uniform contraction (using that $\chi(\mu)<0$), we will obtain distortion control on some small neighborhood whose size depends on the constant $K$. For that we use the following.

\begin{claim}[{\cite[Proposition 3.4]{DiaGelRam:}}]\label{claim:here}
	Given $\varepsilon_D>0$, let $\delta_0>0$ be such that
\[
	\max_{i=0,\ldots,N-1}\,\,\,\max_{x,y\in\bS^1,\lvert y-x\rvert\le 2\delta_0}
	\Big\lvert\log\frac{\lvert f_i'(y)\rvert}{\lvert f_i'(x)\rvert}\Big\rvert
	\le \varepsilon_D.
\]	
Suppose that $x\in\bS^1$, $r>0$, $n\ge1$ are such that for every $\ell=0,\ldots,n-1$ we have
\[
	\lvert (f_\rho^\ell)'(x)\rvert<\frac1r\delta_0e^{-\ell\varepsilon_D},
\]
then for every $\ell=0,\ldots,n-1$  we have
\[
	\sup_{x,y\colon \lvert y-x\rvert\le r}
	\frac{\lvert (f_\rho^\ell)'(y)\rvert}{\lvert (f_\rho^\ell)'(x)\rvert}
	\le e^{\ell\varepsilon_D}.
\]
\end{claim}

Fixing some 
\[
	\varepsilon_D<\frac12\min\{\varepsilon_E,\lvert\chi(\mu)-\varepsilon_E\rvert\},
\]	 
let $\delta_0>0$ be as in Claim~\ref{claim:here} and choose also $r>0$ such that $K<\delta_0/r$. Thus, for every $X=(\rho,x)\in\Lambda$ and every $y\in (x-r,x+r)$ for every $\ell=0,\ldots,n-1$ with~\eqref{eq:understand} we obtain
\[
	 K e^{\ell(\chi(\mu) - \varepsilon_E/2)}e^{-\ell\varepsilon_D}
	 \le \lvert(f_{[\rho_0\ldots\,\rho_{\ell-1}]})'(y)\rvert
	\leq K e^{\ell(\chi(\mu) + \varepsilon_E/2)}e^{\ell\varepsilon_D}.
\]
Dividing now $\bS^1$  into intervals of length $r$, at least one interval of them, denoted $I$,  must contain  at least $K^{-1}re^{n(h(\mu)-\varepsilon)}$  starting points $x$ of $(n,1)$-separated trajectories  corresponding to $\Xi^+(n)$. Note that this way we perhaps disregard some of the elements in $\Xi^+(n)$, but we continue to denote the remaining set by $\Xi^+(n)$ and the proposition follows exchanging now $K^{-1}r$ for $K$.
\end{proof}

\subsection{Large subset of the level set}\label{sec:conlarsubset}

In this section, we will construct a large subset $\Xi\subset\pi(\cL(\alpha))$.
We start by fixing some quantifiers.

\medskip
\noindent\emph{Choice of quantifiers.} 
Given 
\begin{equation}\label{eq:let}
	h=\limsup_{\beta\to\alpha}\cE(\alpha),
\end{equation}	
 there is  a  sequence of ergodic measures $(\mu_k)_{k\ge0}$ with Lyapunov exponents converging to $\alpha$ and with the upper limit of entropies equal to $h$.
We aim to prove that $h_{\rm top} (\cL(\alpha))\geq h$.
Without weakening of assumptions, by passing to a subsequence, we can assume that all the measures $\mu_k$ have exponents of the same sign and that their entropies converge to $h$, recall the results in Section~\ref{ss.previous}. In the following, we will assume that $\alpha\le 0$ and that all the measures $\mu_k$ have negative exponents. 
The other case can be obtained by studying the map $F^{-1}$ instead of $F$.

Fix a sequence $(\varepsilon_k)_{k\ge0}$ with $\varepsilon_k \searrow 0$ and apply Proposition \ref{prolem:fluminense} item 1) to each measure $\mu_k$:
we get constants $K_k=K_k(\mu_k,\varepsilon_k)$ and for every $n\ge1$ a set $\Xi^+_k(n)$ of finite sequences of length $n$.
As, by our choice of sequences, $h(\mu_k)\to h$, we can assume that our constants $K_k$ and $\varepsilon_k$ are such that for every $k\ge0$ and for every $\ell\ge1$ 
we have
\begin{equation} \label{eqn:condskeleent}
	K_k^{-1}e^{\ell(h-\varepsilon_k)}
	\le\card \Xi^+_k(\ell)
	\le K_k e^{\ell(h+\varepsilon_k)}.
\end{equation}
We will choose a sequence $(n_k)_{k\ge0}$ with $n_k\nearrow\infty$, which will be further specified below.
By Proposition~\ref{prolem:fluminense} item 1), there is $\Xi^+_k(n_k)$ such that for every $j\in\{0,\ldots, n_k\}$ and $\ell\in\{j,\ldots,n_k\}$ there exist at least 
\begin{equation}\label{e:cjns}
\frac13K_k^{-2} e^{-2(\ell-j)\varepsilon_k} e^{j(h-\varepsilon_k)}
\end{equation}
 initial sequences of length $j$ of $\Xi^+_k(n_k)$, each of which has between 
 \begin{equation}
 \label{e.conos}
 \frac 13 K_k^{-2} e^{-2j\varepsilon_k} e^{(\ell-j)(h-\varepsilon_k)}
 \quad
 \text{and}
 \quad
 K_k e^{(\ell-j)(h+\varepsilon_k)}
 \end{equation} 
 continuations to an initial sequence of length $\ell$ of $\Xi^+_k(n_k)$. 

The same arguments but applied to $\sigma^{-1}$ instead of $\sigma$ provide sets $\Xi^-_k(n_k)$ of finite sequences with the very same properties.

Finally, by Proposition~\ref{prolem:fluminense} item 2) applied to $\mu_k$ and $\sigma$, there exist  intervals $I_k=I(\mu_k,\varepsilon_k)$ such that for every $x\in I_k$ we have
\[
	K_k^{-1} e^{\ell(\chi(\mu_k) - \varepsilon_k)}
	\leq \lvert(f_{[\rho_1\ldots\,\rho_\ell]})'(x)\rvert
	\leq K_k e^{\ell(\chi(\mu_k) + \varepsilon_k)}.
\]
To each interval $I_k$ we associate numbers $\delta_k>0$ and $M_k>0$ provided by Lemma~\ref{lem:connect}. 
This ends the choice of quantifiers.

\medskip
\noindent\emph{Construction of $\Xi$.}
We now are prepared to construct a subset $\Xi\subset\Sigma_N$ in the projection of $\cL(\alpha)$. First, we construct \emph{forward} orbits on which the (forward) Lyapunov exponent is $\alpha$. 

As the chosen orbit pieces, and their chosen neighborhoods, are uniformly contracting, we can fix a sequence of sufficiently fast increasing natural numbers $(n_k)_k$ such that for each $k$ for every $(\rho_1\ldots\rho_{n_k})\in\Xi^+_k(n_k)$ we have that
\[
	\lvert f_{[\rho_1\ldots\,\rho_{n_k}]}(I_k)\rvert<\delta_{k+1}.
\]
Note that $n_k$ can be chosen arbitrarily large. We will further specify the choice of this sequence in Section~\ref{subsec:prelimste}. Hence, by Lemma~\ref{lem:connect},  to  each $(\rho_1\ldots\rho_{n_k})$ we associate one (there may be several choices, we just pick one) finite sequence $(\tau_1\ldots \tau_m)$, $m\le M_{k+1}$, such that
\[
	(f_{[\tau_1\ldots\,\tau_m]}\circ f_{[\rho_1\ldots\,\rho_{n_k}]})(I_k)\subset I_{k+1}.
\]
We point out  that the sequence $(\tau_1\ldots\tau_m)$ depends on the initial sequence $(\rho_1\ldots\rho_{n_k})$, which is not reflected by the notation to simplify the exposition.

We consider now the set of all such concatenated finite sequences defined by
\begin{equation}\label{eq:laterneeded}
	\Xi_k'\eqdef \{(\rho_1\ldots\,\rho_{n_k}\tau_1\ldots\,\tau_{m})\colon
		(\rho_1\ldots\,\rho_{n_k})\in\Xi^+_k(n_k)\},
\end{equation}
note  again that here $(\tau_1\ldots\,\tau_{m})$ depends on $(\rho_1\ldots\,\rho_{n_k})$.
We write  $(\rho_1\ldots\,\rho_{n_k})=\varrho$ and  $(\tau_1\ldots\,\tau_{m})=\vartheta$ and say that $\varrho$ is a {\emph{main sequence}} and that $\vartheta$ is a {\emph{connecting sequence}}.
Finally, we consider the set
$\Xi^+$  of all one-sided infinite sequences
\[
	\Xi^+
	\eqdef \{\varrho_1\vartheta_1\varrho_2 \vartheta_2 \ldots\varrho_k\vartheta_k \ldots\colon
	\varrho_k\vartheta_k \in\Xi_k'\}.
\]
Note that by construction, for every $k\ge1$ we have
\[
	(f_{[\varrho_\ell \vartheta_\ell] }\circ\ldots\circ f_{[\varrho_1 \vartheta_1]})(I_1)
	\subset I_{\ell+1},
	\quad \ell=1,\ldots,k.
\]
By our choice of quantifiers,  for every $x\in I_1$ and every $\xi\in\Xi^+$ we have
\begin{equation} \label{eqn:forwad}
	\lim_{n\to\infty} \frac 1n \log \,\lvert(f_\xi^n)'(x)\rvert =\alpha.
\end{equation}

Until now we proved, that for some interval $I_1$ there exists a large set of forward-infinite symbolic sequences $\xi^+$ such that for \emph{every} point $x\in I_1$ the forward orbit of $(x,\xi^+)$ satisfies \eqref{eqn:forwad}. Now we can do the construction in the other (time) direction using the sets of finite sequences $\Xi^-_k(n)$ instead of $\Xi^+_k(n)$. It is somewhat analogous: we take the sequence $(\mu_{-k})_{k\ge1}$ of measures defined by $\mu_{-1}=\mu_1, \mu_{-2}=\mu_2,\ldots$. For every $k\ge1$ we apply Proposition \ref{prolem:fluminense} item 1) to the inverse map $F^{-1}$ and the $F^{-1}$-invariant measure $\mu_{-k}$ and obtain a set of sequences that we will denote by $\Xi^-_k(n_k)$. Then we connect them using Lemma \ref{lem:connect}.
Note that constructed backward itineraries are expanding, thus for a backward itinerary we just get \emph{one} point following it instead of \emph{an interval}
of points as in the forward itinerary. Note also that for each finite concatenation we get a closed interval of starting points and that these intervals form a nested sequence, the point is given by the intersection of these intervals.
In this way we obtain a set $\Xi^-$ of backward-infinite symbolic sequences such that for each of them there exists a corresponding backward orbit $(y,\xi^-)$ satisfying
\[
\lim_{n\to-\infty} \frac 1n \log \,\lvert(f_\xi^n)'(y)\rvert =\alpha.
\]
By Lemma \ref{lem:connect} we can make each of those backward orbits end at some point in $I_1$. Hence, each of those trajectories can be prolonged into the future by any $\xi\in\Xi^+$. 

To summarize, we obtained a set of two-sided infinite sequences
\[
	\Xi\eqdef\Xi^-.\Xi^+
	= \{\xi^-.\xi^+\colon\xi^\pm\in\Xi^\pm\}
\]	
such that  $\Xi\subset\pi(\cL(\alpha))$, where $\pi\colon\Sigma_N\times\bS^1\to\Sigma_N$ denoted the natural projection.
Note that $\Xi$ depends on the choice of the quantifiers $\Xi=\Xi((\varepsilon_k)_k, (n_k)_k)$. The sequence $(n_k)_k$, and hence the set $\Xi$, will be further specified in Section~\ref{subsec:prelimste}. 

\subsection{Specifying the set $\Xi=\Xi((\varepsilon_k)_k, (n_k)_k)$}\label{subsec:prelimste}

We need the following technical lemma which estimates how rapidly sequences in $\Xi=\Xi^-.\Xi^+$ are ``branching out''. However, for technical reasons (see the proof of Lemma~\ref{lem:entropyestimates}), we are only interested in some specific type of branching (sequences at length $\iota$ branch out to sequences of length $N$, where $\iota$ satisfies $N-\iota\le\iota$).
This will be used in Section~\ref{subsec:estoftheentr} to construct a special probability measure.

Recall the quantifiers fixed in Section~\ref{sec:conlarsubset}.

\begin{lemma} \label{lem:skelestrong}
	There are a sequence $(n_k)_k$, a constant $K>1$, and a function $\delta\colon\bN\to(0,1)$ with
\begin{equation}\label{eq:cond0}
	n_k>2^k
	\quad	\text{ and }\quad
	\lim_{N\to\infty}\delta(N)=0
\end{equation}
such that for every $N\ge1$ there exist between $K^{-1} \cdot e^{N(h-\delta(N))}$ and $K \cdot e^{N(h+\delta(N))}$ different sequences of length $N$ which admit continuations to a one-sided infinite in $\Xi^+$.
Moreover, 
\[
	\text{ for every }
	\iota\in\{0,\ldots,N\} \text{ satisfying }N-\iota\le \iota
\]
there exist at least $K^{-1} \cdot e^{\iota h}e^{-\iota \delta(\iota)}$ different finite sequences of length $\iota$ such that  each of them has at least 
$K^{-1} \cdot e^{(N-\iota)h}e^{-N \delta(\iota)}$ 
 and at most $K \cdot e^{(N-\iota) h }e^{N \delta(\iota)}$ different continuations to a sequence in $\Xi^+$.
 
The same statement also holds for $\Xi^-$, modulo time reversal.
\end{lemma}

\begin{proof}
We will present the proof for $\Xi^+$, the case of $\Xi^-$ is analogous.
For the sequences $(\varepsilon_k)_k$, $(K_k)_k$, and $(M_k)_k$ chosen in Section~\ref{sec:conlarsubset} we choose a sequence $(n_k)_k$ such that
\begin{equation}\label{eq:conddd1}
	\prod_{i=1}^{k+1}K_i
	\le e^{n_{k-1}\varepsilon_k},
\end{equation}
\begin{equation}\label{eq:conddd2}
	\sum_{i=1}^{k-1}n_i\varepsilon_i
	\le n_k\varepsilon_k
	\quad\text{ or, equivalently, }\quad 
	\sum_{i=1}^kn_i\varepsilon_i\le 2n_k\varepsilon_k,
\end{equation}
and
\begin{equation}\label{e.KKKK}
	M_1 + \sum_{i=1}^{k-1}(n_i+M_{i+1})
	<n_k\quad
	\text{ and }\quad
	e^{M_{k+1}h}
	\le e^{n_{k-1}\varepsilon_k}.
\end{equation}
Note that these conditions hold simultaneously if $(n_k)_k$ grows sufficiently fast. 

By the construction in Section~\ref{sec:conlarsubset}, any sequence in $\xi\in \Xi^+$ is a concatenation of main and connecting finite sequences $\varrho_1\vartheta_1\varrho_2 \vartheta_2 \ldots\varrho_k\vartheta_k \ldots$ such that $\varrho_k\in\Xi^+_k(n_k)$ and $|\vartheta_k| \le M_{k+1}$. To shorten notation, we say that an index $\iota\in \varrho_k$ provided that it enumerates a symbol in the block $\varrho_k$ within the sequence $\xi=\varrho_1 \vartheta_1 \varrho_2 \vartheta_2 \ldots \varrho_k \vartheta_k \ldots$, similar for any of the other blocks. Denote by 
\[
	m_1=m_1(\xi)\eqdef0,\quad
	m_k=m_k(\xi)\eqdef\sum_{i=1}^{k-1}(n_i+ \lvert\vartheta_i\rvert)
\]	 
the position of the beginning of the block $\varrho_k$.

To show the first claim in the lemma, let $N\ge1$. Let $k\ge1$ be such that $m_k<  N\le m_{k+1}$ and write $N=m_k+\ell$, $\ell\ge 1$. Using~\eqref{eqn:condskeleent} we have that the number of sequences of length $N$ is given by
\[
	\card\Xi^+_1(n_1)\cdot\ldots\cdot\card\Xi^+_{k-1}(n_{k-1})\cdot\card\Xi^+_k(\ell)
	\le \prod_{i=1}^{k-1} K_i \cdot e^{Nh} 
	\cdot e^{\sum_{i=1}^{k-1}n_i\varepsilon_i}e^{\ell\varepsilon_k}.
\]
By~\eqref{eq:conddd1} and~\eqref{eq:conddd2} the latter can be estimated from above by
\[
	e^{n_{k-1}\varepsilon_k}
	\cdot e^{Nh} 
	\cdot e^{2n_{k-1}\varepsilon_{k-1}}e^{\ell\varepsilon_k}
	\le e^{Nh}e^{3n_{k-1}\varepsilon_{k-1}+\ell\varepsilon_k}
	\le e^{N(h+3\varepsilon_{k-1})},
\]
with an analogous lower bound proving the claimed property.

To show the second claim in the lemma, let now $N\ge 1$ and $\iota\in\{1,\ldots,N\}$ satisfying $N-\iota\le\iota$. First, note that either $\iota\in\varrho_k$ or $\iota\in\vartheta_k$ for appropriate $k$.  We will only discuss the  case $\iota\in\varrho_k$, the other one is simpler and hence omitted. 
Note that $\iota\in\varrho_k$ implies 
\begin{equation}\label{eq:formula}
	m_{k-1}
	\le \iota
	\le \sum_{i=1}^k(n_i + \lvert\vartheta_i\rvert)
	\le \sum_{i=1}^k(n_i + M_{i+1})
	< n_{k+1},
\end{equation}
where the last inequality follows from~\eqref{e.KKKK}. 
Now note that \eqref{eq:formula} and our hypothesis $N\le 2\iota$ imply that  either $N\in\varrho_k$ or $N\in\vartheta_k$ or $N\in\varrho_{k+1}$. 
This leads to the three cases studied below. In the course, we will be implicitly defining constants $c^{(i)}(\iota,N), C^{(i)}(\iota,N)$, $i=1,2,3$, needed at the end of the proof.

\smallskip\noindent\textbf{Case 1: $\iota\in\varrho_k$ and $N\in\varrho_k$.}
\begin{figure}[h] 
 	\begin{overpic}[scale=.55]{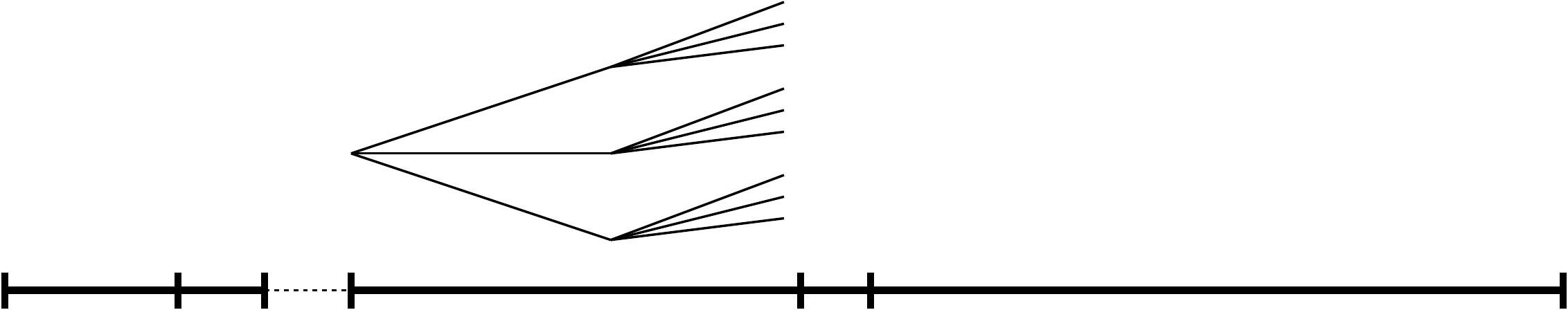}
		\put(21.2,2.8){\small $m_k$}
		\put(39,2.8){\small $\iota$}
		\put(48.5,2.8){\small $N$}
 		\put(0.1,-3){\small $\varrho_1$}
 		\put(12,-3){\small $\vartheta_1$}
 		\put(23,-3){\small $\varrho_k$}
 		\put(51,-3){\small $\vartheta_k$}
 		\put(75,-3){\small $\varrho_{k+1}$}
 	\end{overpic}
\caption{Case 1 with $j,N\in\varrho_k$}
\label{fig.proofCase1}
\end{figure}
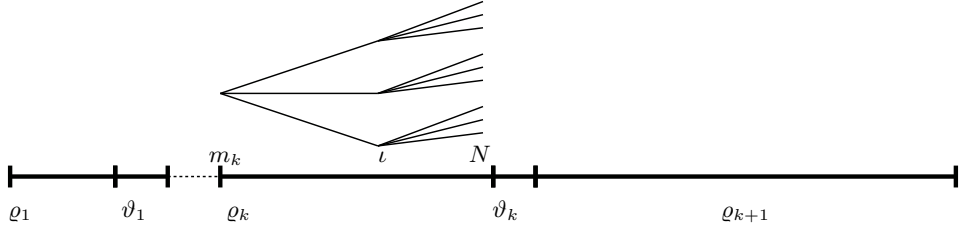

Write $\iota=m_k+j$ and $N=m_k+\ell$ and note that $N-\iota =j-\ell$. Compare Figure~\ref{fig.proofCase1}.
By our choice of the quantifiers (see 
\eqref{e:cjns} in Section~\ref{sec:conlarsubset} applied to $\Xi^+_k(\ell))$), there exist at least 
\begin{equation}\label{eq:LDM1a}\begin{split}
	\frac13K_k^{-2}e^{-2(\ell-j)\varepsilon_k}e^{j(h-\varepsilon_k)}
	&= \frac13K_k^{-2}e^{-2(N-\iota)\varepsilon_k}e^{j(h-\varepsilon_k)}\\
	&= \frac13K_k^{-2}e^{-(j+2(N-\iota))\varepsilon_k}\cdot e^{jh}
	=H(\iota,N)\cdot e^{jh},
\end{split}\end{equation}
where
\begin{equation}\label{eq:sacoscheio}
	H(\iota,N)
	\eqdef \frac13K_k^{-2}e^{-(j+2(N-\iota))\varepsilon_k},
\end{equation}
initial sequences within the block $\varrho_k$ each of which has between (see \eqref{e.conos})
\begin{equation}\label{eq:LDM1b}\begin{split}
	\frac13 K_k^{-2}e^{-2j\varepsilon_k}e^{(\ell-j)(h-\varepsilon_k)} 
	&=	\frac13 K_k^{-2}e^{-2j\varepsilon_k}e^{(N-\iota)(h-\varepsilon_k)} \\
	&= \frac13 K_k^{-2}e^{-(2j+N-\iota)\varepsilon_k}\cdot e^{(N-\iota)h}\\
	&\eqdef c^{(1)}(\iota,N)\cdot e^{(N-\iota)h}
\end{split}\end{equation}
and
\begin{equation}\label{eq:LDM1c}\begin{split}
K_k e^{(\ell-j) (h +\varepsilon_k)}
	&= K_k e^{(N-\iota)(h+\varepsilon_k)}
	= K_ke^{(N-\iota)\varepsilon_k} \cdot e^{(N-\iota)h}\\
	&\eqdef C^{(1)}(\iota,N)\cdot e^{(N-\iota)h}
	\end{split}
\end{equation}
continuations to a sequence of length $\ell$ within this block.

\smallskip\noindent\textbf{Case 2: $\iota\in\varrho_k$ and $N\in\vartheta_k$.}
Write $\iota=m_k+j$ and $N=m_k+n_k+\ell$ and note that $N-\iota =n_k-j+\ell$. By our choice of the constants (see
\eqref{e:cjns} in Section~\ref{sec:conlarsubset} now applied to $\Xi^+_k(n_k)$ with $\ell=n_k$) and also using $n_k-j\le N-\iota$, there exist at least 
\begin{equation}\label{eq:LDM2a}\begin{split}
	\frac13K_k^{-2}e^{-2(n_k-j)\varepsilon_k}e^{j(h-\varepsilon_k)}
	&\ge \frac13K_k^{-2}e^{-2(N-\iota)\varepsilon_k}e^{j(h-\varepsilon_k)}\\
	&= \frac13K_k^{-2}e^{-(j+2(N-\iota))\varepsilon_k}\cdot e^{jh}
	= H(\iota,N)\cdot e^{jh}
\end{split}\end{equation}
initial sequences of length $j$ within the block $\varrho_k$. By \eqref{e.conos} each of them has between
\begin{equation}\label{eq:LDM2b}
	\frac13 K_k^{-2}e^{-2j\varepsilon_k}e^{(n_k-j)(h-\varepsilon_k)} 
\end{equation}
and
\begin{equation}\label{eq:LDM2c}
	K_k e^{(n_k-j) (h +\varepsilon_k)}
\end{equation}
continuations to a sequence of length $n_k$ within this block. Each of those sequences can then be continued (without further branching) to a sequence of length $N-m_k$ by some connecting sequence whose length is between $0$ and $M_k$. We now estimate the terms in \eqref{eq:LDM2b} and \eqref{eq:LDM2c}. Using that  $n_k-j\le N-\iota$ and $\ell \le M_{k+1}$ we get
\begin{equation}\label{eq:wrong1}\begin{split}
	\text{eq.}\eqref{eq:LDM2b}
	&= \frac13 K_k^{-2} e^{-2j\varepsilon_k}  
		e^{- (n_k-j)\varepsilon_k }  
		e^{-\ell h}e^{(n_k+\ell-j)h}\\
	&\ge \frac13 K_k^{-2} e^{-(2j+N-\iota)\varepsilon_k}  e^{-M_{k+1}h}
		\cdot e^{(N-\iota)h}
				\eqdef c^{(2)}(\iota,N)\cdot e^{(N-\iota)h},\\
	\text{eq.}\eqref{eq:LDM2c}
	&= K_k e^{(n_k-j)\varepsilon_k} e^{-\ell h} e^{(n_k+\ell-j)h}\\
	&< K_k e^{(N-\iota)\varepsilon_k} 
	\cdot e^{(N-\iota)h}
			\eqdef C^{(2)}(\iota,N)\cdot e^{(N-\iota)h} .
\end{split}\end{equation}

\smallskip\noindent\textbf{Case 3: $\iota\in\varrho_k$ and $N\in\varrho_{k+1}$.}
Write $\iota=m_k+j$ and $N=m_k+n_k+\lvert \vartheta_k\rvert +\ell=m_{k+1}+\ell$. First, as in Case 2, there exist initial sequences of length $j$ within $\varrho_k$ which can be estimated as in \eqref{eq:LDM2a} and each of them has between \eqref{eq:LDM2b} and \eqref{eq:LDM2c} continuations to a sequence of length $n_k$ within this block. Then each of them is continued (without further branching) by some connecting sequence of length $\lvert\vartheta_k\rvert$ up to length $m_{k+1}$ at the beginning of block $\varrho_{k+1}$. Compare Figure~\ref{fig.proofCase3}.
Then, by~\eqref{eqn:condskeleent} in Section~\ref{sec:conlarsubset} applied to $\Xi^+_{k+1}(\ell)$, for each such sequence there exist between
\begin{equation}\label{e:ptqp}
	K_{k+1}^{-1}e^{\ell(h-\varepsilon_{k+1})}
	= K_{k+1}^{-1}e^{(N-m_{k+1})(h-\varepsilon_{k+1})}
	\text{ and }
		K_{k+1}e^{(N-m_{k+1})(h+\varepsilon_{k+1})}
\end{equation}
continuations to a sequence of length $N-m_k$.
\begin{figure}[h] 
 	\begin{overpic}[scale=.55]{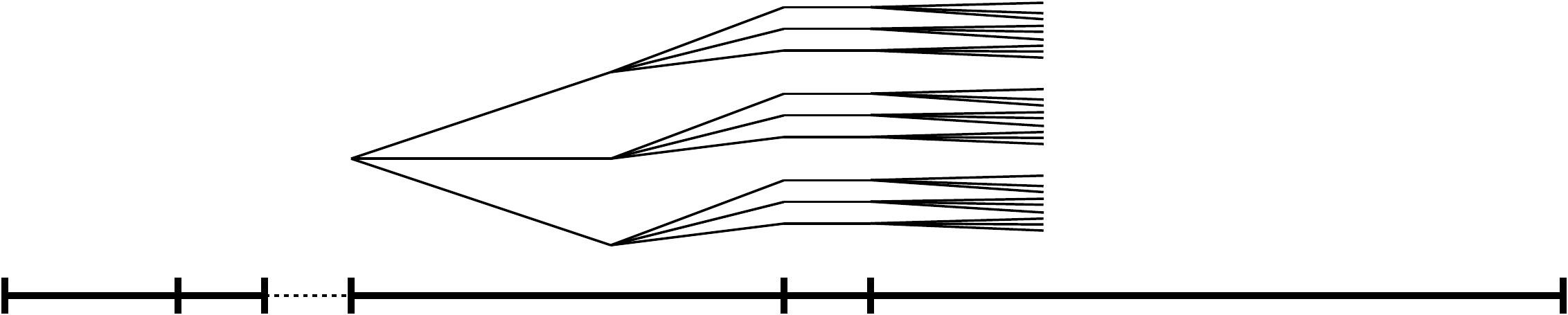}
		\put(21.2,2.8){\small $m_k$}
		\put(39,2.8){\small $\iota$}
		\put(54,2.8){\small $m_{k+1}$}
		\put(66,2.8){\small $N$}
 		\put(0.1,-3){\small $\varrho_1$}
 		\put(12,-3){\small $\vartheta_1$}
 		\put(23,-3){\small $\varrho_k$}
 		\put(51,-3){\small $\vartheta_k$}
 		\put(75,-3){\small $\varrho_{k+1}$}
 	\end{overpic}
\caption{Case 3 with $j\in\varrho_k$ and $N\in\varrho_{k+1}$}
\label{fig.proofCase3}
\end{figure}
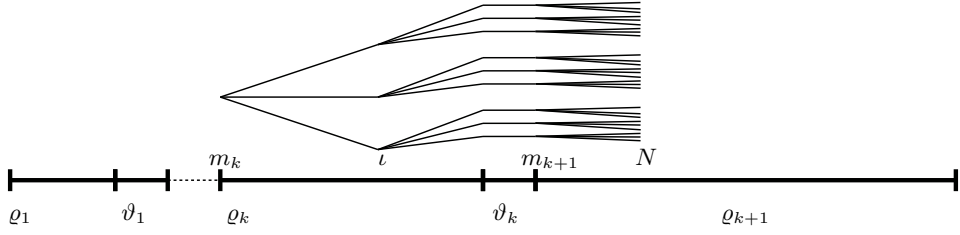
Thus,  as for \eqref{eq:LDM2a} and using $n_k-j\le N-\iota$, there exist at least
\begin{equation}\label{eq:LDM2d}\begin{split}
	\frac13 K_k^{-2}e^{-2(n_k-j)\varepsilon_k} e^{j(h-\varepsilon_k)}
	&\ge \frac13 K_k^{-2} e^{-2(N-\iota)\varepsilon_k} e^{j(h-\varepsilon_k)}\\
	&= \frac13 K_k^{-2}  e^{-(j+2(N-\iota))\varepsilon_k}
		\cdot e^{jh}
		=H(\iota,N)\cdot e^{jh}  
\end{split}\end{equation}
initial sequences of length $j$ within $\varrho_k$. Moreover, multiplying the terms in 
\eqref{e:ptqp} with~\eqref{eq:LDM2b} and~\eqref{eq:LDM2c}, respectively, each of them has between
\begin{equation}\label{eq:LDM3b}\begin{split}
	&\frac13 K_k^{-2}e^{-2j\varepsilon_k}e^{(n_k-j)(h-\varepsilon_k)} 
	\cdot
	K_{k+1}^{-1}e^{(N-m_{k+1})(h-\varepsilon_{k+1})}\\
&= \frac13 K_k^{-2}K_{k+1}^{-1}
			e^{-2j\varepsilon_k-(n_k-j)\varepsilon_k-(N-m_{k+1})\varepsilon_{k+1}}
			e^{(n_k-j+N-m_{k+1})h} \\
&= \frac13 K_k^{-2}K_{k+1}^{-1}
			e^{-2j\varepsilon_k-(n_k-j)\varepsilon_k-(N-m_{k+1})\varepsilon_{k+1}}
			e^{-\lvert\vartheta_k\rvert h}
			e^{(n_k-j+\lvert\vartheta_k\rvert+N-m_{k+1})h} \\
&\ge \frac13 K_k^{-2}K_{k+1}^{-1}
			e^{-(2j+N-\iota)\varepsilon_k} 
			\cdot e^{-M_{k+1}h}
			e^{(N-\iota)h}			
\eqdef c^{(3)}(\iota,N)\cdot e^{(N-\iota)h} 		
\end{split}\end{equation}
and
\begin{equation}\label{eq:LDM3c}\begin{split}
	K_k &e^{(n_k-j) (h +\varepsilon_k)}
	\cdot
	K_{k+1}e^{(N-m_{k+1})(h+\varepsilon_{k+1})}\\
	&= K_kK_{k+1}
		e^{(n_k-j)\varepsilon_k+(N-m_{k+1})\varepsilon_{k+1}}
		e^{(n_k-j+N-m_{k+1})h}\\
	&\le  K_kK_{k+1}
		e^{(N-\iota)\varepsilon_k}
		\cdot e^{(N-\iota)h}
	\eqdef C^{(3)}(\iota,N)\cdot e^{(N-\iota)h} 
\end{split}\end{equation}
continuations each of which can then be continued to an infinite sequence on $\Xi^+$. This ends the  estimates in Case 3.

\smallskip
We are now ready to estimate the cardinality of the set of finite compound sequences $\varrho_1\vartheta_1\ldots\varrho_{k-1}\vartheta_{k-1}$ which is given by 
\[
	\card\Xi^+_1(n_1)\cdot\ldots\cdot\card\Xi^+_{k-1}(n_{k-1}).
\]	
Joining the above Cases 1, 2, and 3,  let us now estimate all possible continuations to a sequence $\varrho_1\vartheta_1\ldots\varrho_{k-1}\vartheta_{k-1}\xi_{m_k}\ldots\xi_\iota$ of length $\iota$ which then has continuations to an infinite sequence $\xi\in \Xi^+$. Writing $\iota=m_k+j$, there exist (using~\eqref{eq:LDM1a} in Case 1,~\eqref{eq:LDM2a} in Case 2, and~\eqref{eq:LDM2d} in Case 3) at least
\begin{equation}\label{eq:defL}
	L
	= \card\Xi^+_1(n_1)\cdot\ldots\cdot\card\Xi^+_{k-1}(n_{k-1})
		\cdot 
		H(\iota,N)\cdot e^{jh}
\end{equation}
initial finite sequences of length $\iota$ such that each of them has at least (using~\eqref{eq:LDM1b},~\eqref{eq:LDM2b}, and~\eqref{eq:LDM3b})
\[
	 \min_{r=1,2,3}c^{(r)}(\iota,N)\cdot 
		e^{(N-\iota)(h-\varepsilon_k)}
\]
and at most (using~\eqref{eq:LDM1c},~\eqref{eq:LDM2c}, and~\eqref{eq:LDM3c})
\[
	\max_{r=1,2,3}C^{(r)}(\iota,N)\cdot 
	e^{(N-\iota)(h-\varepsilon_k)}
\]
continuations to an infinite sequence in $\Xi^+$.
To conclude the proof of the lemma, we finally estimate the above terms.
Recalling (see \eqref{eq:laterneeded}) that $\card\Xi_\ell'=\card\Xi^+_\ell(n_\ell)$, with~\eqref{eqn:condskeleent} and \eqref{eq:defL} we have
\[\begin{split}
	L
	&\ge \prod_{i=1}^{k-1}\Big(K_i^{-1}e^{n_i(h-\varepsilon_i)}\Big)
	\cdot 	H(\iota,N)
	\cdot e^{jh}\\
	&= \prod_{i=1}^{k-1}K_i^{-1} 
		\cdot e^{-\sum_{i=1}^{k-1}n_i\varepsilon_i}
		\cdot H(\iota,N)
	\cdot e^{\iota h}
	 \eqdef \widetilde H(\iota,N)
	\cdot e^{\iota h}.
\end{split}\]
We can conclude, recalling the definition of $ H(\iota,N)$ in~\eqref{eq:sacoscheio}
\[\begin{split}
	\widetilde H(\iota,N)
	&= \prod_{i=1}^{k-1}K_i^{-1} 
		\cdot e^{-\sum_{i=1}^{k-1}n_i\varepsilon_i}
		\cdot \frac13 K_k^{-2}
		\cdot e^{-(j+2(N-\iota))\varepsilon_k}\\
	\small{(\text{using }N-\iota\le \iota\text{ and }j\le\iota)}\quad
	&\ge 	\prod_{i=1}^{k}K_i^{-2} 
		\cdot e^{-\sum_{i=1}^{k-1}n_i\varepsilon_i}
		\cdot \frac13 
		\cdot e^{-3\iota\varepsilon_k}\\
	\small{(\text{using }\eqref{eq:conddd1}\text{ and }n_{k-1}\le\iota)}\quad
	&\ge 	e^{-2\iota\varepsilon_k} 
		\cdot e^{-\sum_{i=1}^{k-1}n_i\varepsilon_i}
		\cdot \frac13 
		\cdot e^{-3\iota\varepsilon_k}\\
	\small{(\text{using }\eqref{eq:conddd2}\text{ and }n_{k-1}+j<\iota)}\quad
	&> 	e^{-2\iota\varepsilon_k} 
		\cdot e^{-2\iota\varepsilon_{k-1}}
		\cdot \frac13 
		\cdot e^{-3\iota\varepsilon_k}.
\end{split}\]
Moreover, using the definitions of $c^{(r)}(\iota,N)$ in~\eqref{eq:LDM1b},~\eqref{eq:wrong1}, and~\eqref{eq:LDM3b}, we have 
\[\begin{split}
	\min_{r=1,2,3}c^{(r)}(\iota,N)=c^{(3)}(\iota,N)
	&= \frac13K_k^{-2}K_{k+1}^{-1} 
		\cdot e^{-(2j+N-\iota)\varepsilon_k} 
		\cdot e^{-M_{k+1}h}\\
	\small{(\text{using }N-\iota\le \iota\text{ and }j\le\iota)}\quad
	&\ge \frac13K_k^{-2}K_{k+1}^{-1} 
		\cdot e^{-3\iota\varepsilon_k} 
		\cdot e^{-M_{k+1}h}\\
	\small{(\text{using }\eqref{eq:conddd1},~\eqref{e.KKKK},\text{ and }n_{k-1}\le\iota)}\quad
	&\ge \frac13 e^{-2\iota\varepsilon_k}
		\cdot e^{-3\iota\varepsilon_k} 
		\cdot e^{-\iota\varepsilon_k} .
\end{split}\]
Finally, using~\eqref{eq:LDM1c},~\eqref{eq:wrong1}, and~\eqref{eq:LDM3c}, we have
\[\begin{split}
	\max_{r=1,2,3}C^{(r)}(\iota,N)
	=C^{(3)}(\iota,N)
	&= K_kK_{k+1}
		\cdot e^{(N-\iota)\varepsilon_k}\\
	\small{(\text{using }N-\iota\le \iota, \eqref{eq:conddd1},\text{ and }n_{k-1}\le\iota)}\quad		
	&\le K_kK_{k+1}
		\cdot e^{\iota\varepsilon_k}
	\le e^{\iota\varepsilon_k}			
	\cdot e^{\iota\varepsilon_k}.
\end{split}\]

Now, choosing the function $\delta$ appropriately, this implies the claimed estimate, where the constant $K$ in the lemma takes also care of the remaining (not considered) cases.
\end{proof}

\subsection{Estimate of the entropy of $\Xi$: Bridging measures}\label{subsec:estoftheentr}

In this section, we show that $\Xi$ has large entropy, bigger than $h$ defined in~\eqref{eq:let}. The heart of the argument is the construction of certain Borel probability measures $\nu^\pm$ and the application of the Mass distribution principle that we recall below (see~\cite{Mat:95}). Here we follow a general principle of so-called \emph{bridging measures} already used in other contexts. See for example \cite[Section 5.2]{GelRam:09} and \cite[Section 5.2]{GelPrzRam:10} (where those measures are called \emph{w-measures}). We will use Frostman's lemma again in Section~\ref{sssec:entcocy}.

\begin{lemma}[Mass distribution principle or Frostman's lemma]\label{l.frostman}
Consider a compact metric space $(X,d)$ and a subset $\Xi\subset X$. Let $\nu$ be a finite Borel measure such that $\nu(\Xi)>0$. Suppose that there exists $D>0$ such that for every $x\in \Xi$ it holds
$$
\liminf_{\varepsilon\to 0}
\frac{\log \nu ( B(x,\varepsilon))}{\log \varepsilon} \ge D.
$$
Then $\mathrm{HD} (\Xi) \ge D$, where
$\mathrm{HD}$ denotes the Hausdorff dimension.
\end{lemma}

\begin{remark}
\label{r.HD=h}
Consider the natural projections $\pi^\pm\colon\Sigma_N\to\Sigma_N^\pm$ and the shift maps $\sigma^\pm\colon\Sigma_N^\pm\to\Sigma_N^\pm$. Recall that, by construction, $\pi^\pm(\Xi)=\Xi^\pm$. Note that $h_{\rm top}(\sigma^\pm,\Xi^\pm)=h_{\rm top}(\sigma^\pm,\pi^\pm(\Xi))\le h_{\rm top}(\sigma,\Xi)$.

Note that for the standard metric the Hausdorff dimension of any set in $\Sigma_N^+$ is equal to its topological entropy relative to  $\sigma^+$.
\end{remark}

We will apply Lemma~\ref{l.frostman} and Remark~\ref{r.HD=h} to estimate the topological entropy of the sets $\Xi^\pm$ relative to $\sigma^\pm$ and hence the topological entropy of $\Xi$ relative to $\sigma$. 

\begin{lemma}\label{lem:entropyestimates}
	We have $h_{\rm top}(\sigma^+,\Xi^+)\ge h$ and $h_{\rm top}(\sigma^-,\Xi^-)\ge h$.
\end{lemma}

\begin{proof}
Consider the sequence $(r_k)_{k\ge0}$ given by $r_k=2^k$. Note that by~\eqref{eq:cond0} in  Lemma~\ref{lem:skelestrong}  the sequence $(n_k)_k$ grows  faster than the sequence $(r_k)_k$.

We define a probability measure $\nu^+$ depending only on the (forward) one-sided sequences $\Sigma_N^+$ and a measure $\nu^-$ depending only on the (backward) one-sided sequences $\Sigma_N^-$. The measure $\nu^+$ is constructed as follows (the measure $\nu^-$ is analogously defined): for every $k\ge1$
\begin{itemize}
\item[--] for every (``parent") cylinder of level $r_k$ intersecting $\Xi^+$, $\nu^+$ is uniformly subdistributed on its (``child")  subcylinders of level $r_{k+1}$ intersecting $\Xi^+$,
\item[--] as the family of cylinders of levels $r_k$, $k\ge1$, generate the Borel $\sigma$-algebra of $\Sigma_N^+$, we obtain a Borel probability measure on this $\sigma$-algebra.
\end{itemize}

\begin{figure}
\begin{overpic}[scale=.40 
  ]{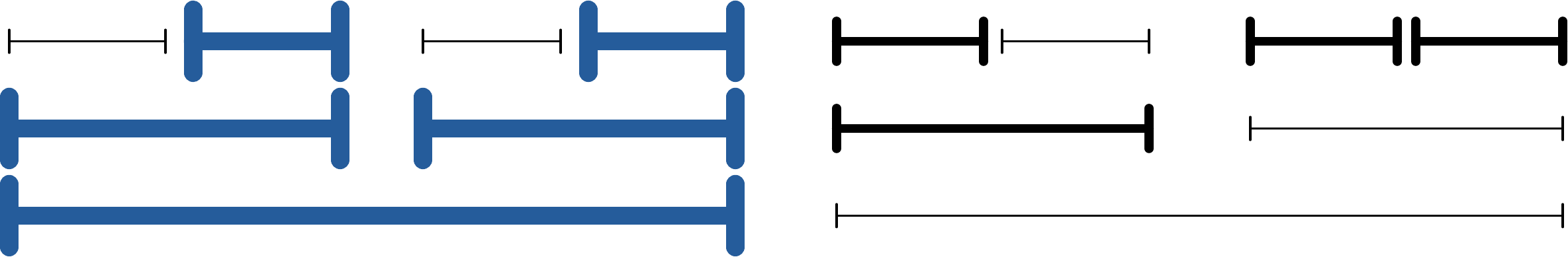}
        \put(102,13){\small$\Sigma^+_{r_{k+1}}$}
        \put(102,7){\small$\Sigma^+_{L}$}
        \put(102,1){\small$\Sigma^+_{r_k}$}
\end{overpic}
\caption{Schematic construction of $\nu^+$: $r_i$-cylinders which intersect $\Xi^+$, $i=k,k+1$, (bold), cylinders on which $\nu^+$ is distributed (heavy bold and blue)}
\label{fig:bridingmeasure}
\end{figure}

Given $L\ge1$ and $\xi^+\in \Xi^+$, denote by $\Delta^+_L(\xi^+)\eqdef [\xi_0\ldots\xi_{L-1}]$ the  cylinder of length $L$ containing $\xi^+$.

\begin{claim}
	For every $\xi^+\in\Xi^+$ we have $\lim_{L\to\infty}-\frac1L\log\nu^+(\Delta^+_L(\xi^+))\ge h$.
\end{claim}

\begin{proof}
First, let
\[
	C
	\eqdef \max_{\xi^+\in\Xi^+}\Delta^+_2(\xi^+)
	>0.
\]
In what follows we fix $\xi^+$ and omit in the notation the dependence of $\Delta_L^+$ on $\xi^+$.
Given $L\ge1$, consider some cylinder $\Delta^+_L$ of length $L$ which has nonempty intersection with $\Xi^+$. We are going to estimate $\nu^+(\Delta^+_L)$. 
There exists a unique index $k\ge1$ such that $r_k<L\le r_{k+1}$. We have
\[
	\nu^+(\Delta^+_L)
	= \nu^+(\Delta_{r_0}^+) \cdot
		\left(\frac {\nu^+(\Delta_{r_1}^+)} {\nu^+(\Delta_{r_0}^+)}  \cdots
		\frac {\nu^+(\Delta_{r_k}^+)} {\nu^+(\Delta_{r_{k-1}}^+)} \right)
		\cdot
		\frac {\nu^+(\Delta^+_L)} {\nu^+(\Delta_{r_k}^+)}
	= T_1\cdot T_2\cdot T_3	,
\]
where $\Delta_{r_i}^+$ is the corresponding parent $r_i$-cylinder of $\Delta^+_{r_{i+1}}$, $i=0,\ldots,k-1$. 
Let us now estimate the  three terms $T_1,T_2$, and $T_3$ from above. Compare Figure~\ref{fig:bridingmeasure}.

First note that, since $r_0=2$ is fixed,  $T_1=\nu^+(\Delta^+_2)\le C$ does not depend on $L$.

For every $i=1,\ldots,k$, by the second part of Lemma~\ref{lem:skelestrong} applied to $N=r_i$ and $\iota=r_{i-1}$ (note that the hypothesis $N-\iota= \iota$ is satisfied),  we obtain that each cylinder $\Delta_{r_i}^+$ contains at least
\[
	K^{-1} \cdot e^{(r_i-r_{i-1})h}
				e^{-r_i\delta(r_{i-1})}
\]
subcylinders of length $r_{i+1}$ intersecting $\Xi^+$ and hence (using $r_i=2r_{i-1}$) we have
\[
\begin{split}
	T_2
	&\le \prod_{i=1}^k \Big(K \cdot e^{-(r_i-r_{i-1})h}
				e^{r_i\delta(r_{i-1})}\Big)\\
	&= K^k \prod_{i=1}^k e^{r_i\delta(r_{i-1})}\cdot e^{(r_0-r_k)h}
	= K^k e^{\sum_{i=1}^kr_{i}\delta(r_{i-1})}e^{-(r_{k}-r_0)h} 
	= e^{-r_{k}h}S(k),			
\end{split}\]
where
\[
	S(k)
	\eqdef K^{k} e^{\sum_{i=1}^kr_{i}\delta(r_{i-1})}e^{r_0h}. 
\]
Recall the particular choice of $r_k=2^k$. For any $r>0$ let 
$$
d(r) \eqdef \max\{i\colon \delta(r_{i-1}) \geq r\}.
$$ 
Let $D=\max_i \delta(r_i)$. Then for any $k$ sufficiently large  so that $d(r)<k$ we have

\[
	\frac 1 {r_k} \sum_{i=1}^k r_i \delta(r_{i-1}) 
	< \frac 1 {r_k} \sum_{i=1}^{d(r)} r_i D + \frac 1 {r_k} \sum_{i=d(r)+1}^k r_i r 
	\leq 2^{d(r) + 1 - k} D + 2 r.
\]
The latter term we can make arbitrarily small by choosing a small $r$ and the former summand tends to zero with $k\to\infty$.
Hence, it follows that
\begin{equation}\label{eqeq:estiSk}
	\lim_{k\to\infty}\frac{1}{r_k}\log S(k)=0.
\end{equation}

The estimate of $T_3$ is done in two steps. First, since $r_k<L\le r_{k+1}$, there is a cylinder $\Delta^+_{r_k}$ which contains $\Delta^+_L$. Let $\ell$ be the number of  cylinders of level $r_{k+1}$ intersecting $\Xi^+$ and subdividing $\Delta^+_{r_k}$. Note that each such cylinder, by construction of $\nu^+$, has measure
\[
	\frac1\ell\nu^+(\Delta^+_{r_k}).
\]
 By Lemma~\ref{lem:skelestrong} applied to $N=r_{k+1}$ and $\iota=r_k$, we have
\[
	\ell 
	\ge K^{-1} \cdot e^{(r_{k+1}-r_k)h} e^{-r_{k+1}\delta(r_k)}
	= K^{-1} \cdot e^{(r_{k+1}-r_k)h} e^{- 2 r_{k}\delta(r_k)}.
\]
Let $\ell'\le\ell$ be the number of  cylinders of level $r_{k+1}$ intersecting $\Xi^+$ which are contained in $\Delta^+_L$. Again by Lemma~\ref{lem:skelestrong} applied now $N=r_{k+1}$ and $\iota=L$ we get
\[
	\ell'
	\le K \cdot e^{(r_{k+1}-L) h} e^{ r_{k+1} \delta(L)}.
\]
Hence 
\[
\begin{split}
	T_3
	&= \nu^+(\Delta^+_L)\cdot\frac{1}{\nu^+(\Delta^+_{r_k})}
	= \ell'\frac1\ell\nu^+(\Delta^+_{r_k})
		\cdot\frac{1}{\nu^+(\Delta^+_{r_k})}\\
	& \le \frac{K \cdot e^{(r_{k+1}-L) h} e^{ r_{k+1} \delta(L)}}
			{K^{-1} \cdot e^{(r_{k+1}-r_k)h} e^{- 2 r_{k}\delta(r_k)}}
	= e^{-(L-r_k)h}R(k),	
\end{split}\]
where 
\[
	R(k)
	\eqdef K^2e^{2 r_k (\delta(L) +\delta(r_k))}.
\]
In this case, it follows easily that
\begin{equation}\label{eq:estiRk}
	\lim_{k\to\infty}\frac{1}{r_k}\log R(k) = 0.
\end{equation}

This concludes the estimates of $T_1,T_2,T_3$.
Putting them together, we get
\[
	\nu^+(\Delta^+_L)
	\leq
		e^{-Lh}
		Q(k), \quad \text{where} \quad
	Q(k)
	= K^{-1} e^{-r_0(h+\delta(r_0))}\cdot e^{-r_kh}S(k)\cdot e^{r_kh}R(k).
\]
Using~\eqref{eqeq:estiSk} and~\eqref{eq:estiRk}, it follows immediately that
\[	
		\lim_{k\to\infty}\frac{1}{r_k}\log Q(k)=0.
\]		
As the cylinder $\Delta^+_L$ which intersects $\Xi^+$ was arbitrary, we get that
\[
	\lim_{L\to\infty}-\frac1L \log \nu^+(\Delta^+_L(\xi^+)) \geq h,
\]
proving the claim.
\end{proof}

Analogous arguments apply to $\nu^-$. 

Note that, by construction, we have $\nu^+(\Xi^+)=\nu^-(\Xi^-)=1$. 
By Lemma~\ref{l.frostman} applied to the probability measure $\nu^+$ on the space $\Sigma_N^+$ we obtain $h_{\rm top}(\sigma^+,\Xi^+)\ge h$. The same arguments give $h_{\rm top}(\sigma^-,\Xi^-)\ge h$ using $\nu^-$ instead of $\nu^+$. This concludes the proof of the lemma.
\end{proof}

\subsection{End of the proof of Theorem~\ref{theoprop:zero}}\label{sec:endofproporrr}

Recall that so far we worked under the hypothesis that the sign of the exponents of the measures $\mu_k$ was negative. To conclude the proof in this case, note that $\Xi\subset\pi(\cL(\alpha))$ and the  projections $\pi$ and $\pi^+$ do not increase entropy, hence 
\[
	h_{\rm top}(F,\cL(\alpha))
	\ge h_{\rm top}(\sigma,\pi(\cL(\alpha)))
	\ge h_{\rm top}(\sigma,\Xi)
	\ge h_{\rm top}(\sigma^+,\Xi^+)
	\ge h.
\]	
This ends the proof of Theorem~\ref{theoprop:zero} in this case. 

In the other case, when the exponents of the measures $\mu_k$ are positive, we can perform the same construction for $F^{-1}$ instead of $F$ and construct a set $\Xi=\Xi(F^{-1})$. However, since we want to determine the topological entropy with respect to $F$, we have to consider the ``inverse" set of $\Xi$, that is, the set of sequences $\Xi'\eqdef\{(\xi^+.\xi^-)\colon (\xi^-.\xi^+)\in\Xi\}$. To conclude, note that $h_{\rm top}(\sigma^+,\pi^+(\Xi'))=h_{\rm top}(\sigma^-,\pi^-(\Xi))$ and apply the second  assertion in Lemma~\ref{lem:entropyestimates}.

\section{Measures of maximal entropy: Proof of Theorem \ref{main2}}\label{sec:proofmain2} 

As explained in Section~\ref{subsec:maxent1}, any measure of maximal entropy (with respect to $F$) projects to the $(1/N,\ldots,1/N)$-Bernoulli measure $\mathfrak b$. In this section, we study its properties and conclude the  proof of Theorem \ref{main2}.

\subsection{Synchronization}\label{sec:synch}

We call a Bernoulli measure $\mathfrak b=(b_1,\ldots,b_N)$ \emph{nondegenerate} if all weights $b_i$, $i=1,\ldots,N$, are positive.
 By Lemma~\ref{lem:common} (there is no measure which is simultaneously $f_i$-invariant for every $i$) the assumptions in \cite[Theorem 8.6]{Cra:90} are satisfied. Hence  
 for the Bernoulli measure $\mathfrak b$ above there exists
a (at least one) $F$-ergodic measure  with positive exponent and a (at least one)
$F$-ergodic measure  with negative exponent, both projecting to $\mathfrak b$. In what follows, in our context, we will show that those measures are unique and will be denoted by $\mu_\pm^{\mathfrak b}$.

Following for example \cite{Mal:}, given a Bernoulli measure $\mathfrak b$, we say that an IFS $\{f_i\}_{i=0}^{N-1}$ with probabilities $\mathfrak b$ is \emph{forward synchronizing} if for every $x,y\in \bS^1$ for $\mathfrak b$-almost every one-sided sequence $\xi\in\Sigma_N^+$ we have
\begin{equation} \label{eqn:distal}
	\lvert f_{\xi}^n(x)- f_\xi^n(y)\rvert \to 0.
\end{equation}
\emph{Backward synchronization} is defined in the same way, but for the IFS $\{f_i^{-1}\}_{i=0}^{N-1}$. We say that it is \emph{synchronizing} if it is both backward and forward synchronizing.

The IFS $\{f_i\}_{i=0}^{N-1}$ is (\emph{forward}) \emph{proximal} if for every $x,y\in \bS^1$ there exists at least one sequence $\xi\in\Sigma_N^+$ such that \eqref{eqn:distal} holds, \emph{backward proximality} is defined analogously.
By \cite[Theorem E]{Mal:}, forward proximality of the IFS implies that every nondegenerate Bernoulli measure satisfies forward synchronization. Similarly, backward proximality implies backward synchronization.%
\footnote{In fact, \cite[Theorem E]{Mal:} (stated for groups of circle homeomorphisms) shows even that we have exponential synchronization, that is, for a given Bernoulli measure  convergence in~\eqref{eqn:distal} is exponential. However, we will not make use of this fact.}

\begin{lemma} \label{lem:synch}
For every Bernoulli measure $\mathfrak b$ satisfying synchronization there are
ergodic measures $\mu_+^{\mathfrak b}$ with positive exponent and $\mu_-^{\mathfrak b}$ with negative exponent such that $\mathfrak b=\pi_\ast\mu_\pm^{\mathfrak b}$.
Moreover, for every ergodic measure $\mu$ with 
 $\mathfrak b=\pi_\ast\mu$ we have that $\mu=\mu_+^{\mathfrak b}$
 or $\mu=\mu_-^{\mathfrak b}$.
\end{lemma}

\begin{proof}
Consider the set
\[
	B\eqdef \big\{(x,y,\xi) \in \bS^1\times \bS^1 \times\Sigma_N^+
		\colon \limsup_{n\to\infty}\, \lvert f_\xi^n(x)- f_\xi^n(y)\rvert>0\big\}
\]
and write
\[
	B_{(x,y)}
	\eqdef \big( \{(x,y)\} \times \Sigma_N^+ \big) \cap B.
\]
Note that forward synchronization means that for every $(x,y)\in \bS^{1} \times \bS^1$ it holds that $\mathfrak b (B_{(x,y)})=0$.

Given $\xi\in\Sigma_N^+$, divide $\bS^1$ into equivalence classes by the relation
\[
	x \sim_\xi  y
	\iff
	\lim_{n\to\infty}\,\lvert f_\xi^n(x)- f_\xi^n(y)\rvert = 0.
\]	
One can easily check that it is indeed an equivalence relation. As the fiber maps are homeomorphisms, those equivalence classes are simply connected, i.e. intervals or points. Note that in principle for $\xi$ there may exist uncountably many
classes. However, there can exist only countably many classes which are nontrivial intervals. Let us denote these
classes by $C_i(\xi)$ with $i\in I(\xi)$.

Let us denote by $\fm$ the Lebesgue measure.

\begin{claim}\label{claim:let}
	For $\mathfrak b$-almost every $\xi$ we have $I(\xi)=\{1\}$ and $\fm(C_1(\xi))=1$.
\end{claim}

\begin{proof}
Note that if $(x,y,\xi) \not\in B$, $x\ne y$, then there is an index $i=i(x,y,\xi)$ such that $x,y\in C_i(\xi)$. As the diagonal $\{(x,y)\in \bS^1\times \bS^1; x=y\}$ has $\Leb\times\Leb$ measure zero, we have
\begin{equation} \label{eqn:synch}
(\Leb \times \Leb \times \mathfrak b)(B^c) =
 \int_{\Sigma_N^+}
 \left(
 \sum_{i\in I(\xi)}
 \Leb(C_i(\xi)))^2  \right) d\mathfrak b(\xi).
\end{equation}
By the comments above, synchronization, and Fubini's Theorem we have that $(\Leb \times \Leb \times \mathfrak b)(B^c)=1$, hence the integrand in \eqref{eqn:synch} is $1$ almost everywhere. As $\sum \Leb C_i(\xi) \leq 1$, $\sum (\Leb C_i(\xi))^2=1$ can happen if, and only if, $I(\xi) = \{1\}$ and $\Leb C_1(\xi)=1$.
Therefore, for $\mathfrak b$-almost every $\xi$ the index set $I(\xi)$ consists of exactly one element
and there is exactly one class with full Lebesgue measure. This class is the whole circle except one point. 
\end{proof}

Using Claim~\ref{claim:let}, for $\mathfrak b$-almost every $\xi$, let $\{z^+(\xi)\}\eqdef\bS^1\setminus C_1(\xi)$ be this missing point and denote  $C^+(\xi)\eqdef C_1(\xi)$.

Clearly, the  synchronization in the set $C^+(\xi)$ implies that the disintegration of any positive exponent invariant ergodic measure projecting to $\mathfrak b$
is supported on $z^+(\xi)$ for almost all $\xi$. We will denote by $\mu_+^{\mathfrak b}$ this measure.

The same arguments applied to  the IFS $\{f_i^{-1}\}$ 
imply that for $\mathfrak b$-almost every $\xi$  there is exactly one class with full Lebesgue measure, called  $C^-(\xi)$, where the points synchronized as $n\to -\infty$. Analogously, $\bS^1= C^-(\xi)\cup\{z^-(\xi)\}$.
Arguing as above we prove that  
 $\mu_-^{\mathfrak b}$ is the only ergodic  measure with negative exponent that projects to $\mathfrak{b}$ and that its disintegration is supported on the points $z^-(\xi)$.
 
 It remains to exclude the case that there is a third ergodic measure $\mu$ (necessarily with zero exponent) projecting to $\mathfrak b$. Note that we have obtained a set $\Sigma^\pm$ of full measure $\mathfrak b$
 for which the points $z^+(\xi)$ and $z^-(\xi)$ are well defined and different. Consider now the disintegration  $\mu_\xi$ of $\mu$ for $\xi\in \Sigma^\pm$. The forward synchronization with the point $z^-(\xi)$ implies that the measure $\mu$ has negative exponent in those fibers. This leads to a contradiction.
 \end{proof}

\subsection{End of the proof of Theorem \ref{main2}}
We can now conclude the proof of Theorem \ref{main2}.
By Lemma~\ref{lem:synch} 
applied to the $(1/N,\ldots, 1/N)$-Bernoulli measure $\mathfrak{b}$ it follows that
 there is exactly a
pair of ergodic measures of
maximal entropy $\log N$ with positive and negative exponent
$\mu_+=\mu_+^{\mathfrak b}$ and $\mu_-=\mu_+^{\mathfrak b}$, respectively.

Assume that the second conclusion in the theorem is not true, that is, that there exists $\alpha\ne\alpha_\pm=\chi(\mu_\pm)$ such that $h_{\rm top}(\cL(\alpha))=\log N$.
Let us assume that $\alpha\ge0$, the proof of the other case is analogous.
By property (E4)  in Section~\ref{ss.convexconjugate} we have $\log N=\cE_\cN(\alpha')$ with $\cN=\cM_{{\rm erg},>0}$ for all $\alpha'$ between $\alpha$ and $\alpha_+$. Hence, by (E4) the function $\cP_\cN$ is not differentiable at $0$. Recall that a measure of maximal entropy is an ergodic equilibrium states for $q=0$. By property (P10) in Section~\ref{ss.pressuref}, there exist two ergodic measures of maximal entropy  (with respect to $\cN$) with exponents given by the (different) left and right derivatives $D_{L/R}\cP_\cN(0)$, these derivatives being nonnegative by the choice of $\cN$. Hence, there would exist two ergodic measures of maximal entropy with two distinct nonnegative Lyapunov exponents, contradicting Lemma~\ref{lem:synch}.
\qed

\subsection{Proof of Corollary~\ref{c.maximalentropy}} 
By  Lemma~\ref{lem:synch}, applied to the (Bernoulli) measure of maximal entropy we have that there are
exactly two ergodic measures of maximal entropy. Arguing by contradiction, suppose that  there is another invariant measure
$\mu$ of maximal entropy which is the weak$\ast$ and in entropy limit of a sequence of ergodic measures, which must be nonergodic. 
Then almost every measure in its (nontrivial) ergodic decomposition
has maximal entropy. Hence this measure is a (nontrivial) linear combination of $\mu_+$ and $\mu_-$ and, in particular,
$\alpha= \chi(\mu)\in (\alpha_- , \alpha_+)$. Without weakening of assumptions suppose that $\alpha\ge 0$. Recall that the function $\mathcal{E}_{>0}$
is continuous and has a unique global maximum at $\alpha_+$. Hence for $\delta$ small $\cE_{>0} (\alpha')<
\log N-\delta$
for all $\alpha'$ in a small neighborhood of $\alpha$.
If there would exist such a sequence of ergodic measures
weak$\ast$ (and hence in Lyapunov exponent)  and in entropy converging to $\mu$  then eventually the Lyapunov exponents of the measures would be arbitrarily close to $\alpha$
and their entropies arbitrarily close to $\log N$. This provides a contradiction with the above inequality.
\qed

\section{Shapes of pressure and Lyapunov spectrum: Proof of Theorem~\ref{main3}}\label{sec:proofmain3}

Recall the properties of pressure, Legendre-Fenchel transform, and convex functions given in Section \ref{sec:entpres}. For convenience, the items are proved in the following order: a), f), g), h), i), b), c), d), e), j), k).

\emph{Property a)}: Convexity follows from basic properties of pressure. The (one-sided) derivative(s) of the pressure function is equal to the Lyapunov exponent of a corresponding equilibrium state by the definition of Legendre-Fenchel transform. For the pressure $\cP_{>0}$ all the equilibrium states have nonnegative Lyapunov exponent, for the pressure $\cP_{<0}$ all the equilibrium states have nonpositive Lyapunov exponent.
Hence property a) follows from property (P9) in Section~\ref{ss.pressuref}.

\emph{Properties f), g), and h)} are formulated in Theorem \ref{main1}.

We need the following lemma.
\begin{lemma}\label{lem:stolen}
	There is $c>0$ such that $\frac {\cE(\alpha) - \cE(0)} {|\alpha|} \leq c \cE(0)$.
\end{lemma}

\begin{proof}
An immediate consequence of Lemma \ref{lem:main3} and property h) is that there exists $c>0$ such that
\[
	\cE(0) \geq \frac {\cE(\alpha)} {1+c|\alpha|}
\]
for all $\alpha \neq 0$. Hence,
\[
	\frac {\cE(\alpha) - \cE(0)} {|\alpha|} \leq c \cE(0).
\]
Taking the limit $\alpha\to0$, we hence get
\begin{equation} \label{eqn:main3}
	\max \{D_R \cE(0), -D_L \cE(0)\} \leq c \cE(0),
\end{equation}
which proves the finiteness of the derivatives $D_L \cE(0), D_R \cE(0)$. The other inequality follows from convexity of $\cE_{<0}$ and $\cE_{>0}$ and property f).
\end{proof}

\emph{Property i)} follows now from Lemma \ref{lem:stolen} and Property a).

\emph{Property b)} follows from the fact that the sets $\cM_{\rm erg, >0}$ and $\cM_{\rm erg, <0}$ contain measures with arbitrarily small Lyapunov exponent (Corollary~\ref{cor:maisemenos}) and hence the limit derivative of both $\cP_{>0}$ (as $q\to -\infty$) and $\cP_{<0}$ (as $q\to \infty$) is zero. The fact that those are indeed plateaus, not asymptotic behaviour, follows from property i) proved below. Indeed, by the definition of Legendre-Fenchel transform, $D_- = D_R \cE(0)$ and $D_+ = D_L \cE(0)$.

\emph{Property c)} follows from Theorem \ref{main1} item c). Indeed, by the definition of Legendre-Fenchel transform, $h_+= \lim_{\alpha \searrow 0} \cE(\alpha)$ and $h_-=\lim_{\alpha \nearrow 0} \cE(\alpha)$.

\emph{Property d)} follows from Theorem \ref{main1} and property a). Indeed, a concave function with a maximum in the interior of the domain is nonincreasing to the right of the maximum and nondecreasing to the left of the maximum.

\emph{Property e)} follows immediately from Theorem \ref{main1}, because by the basic properties of entropy $\cP_{>0}(0)$ is the supremum of entropies of ergodic measures with positive Lyapunov exponents (and similarly $\cP_{<0}(0)$ - negative Lyapunov exponents) and those classes of measures both contain a measure of maximal entropy.

\emph{Property j)} was proved in the course of  the proof of Property i). Indeed, by contradiction, assume that $\cE(0)= 0$. Then equation \eqref{eqn:main3} would imply that $D_L \cE(0)= D_R \cE(0)=0$. However, this would imply that $\cE$ attains its maximum at zero and hence $\cE(0)=\log N$, a contradiction. Together with Theorem~\ref{theoprop:zero}, we have $h_{\rm top}(\cL(0))>0$.

\emph{Property k)}: assume now that we are under assumptions of Theorem \ref{main2} and that hence there are exactly two maxima of the entropy spectrum $\cE(\alpha)$, achieved at points $\alpha_-<0$ and $\alpha_+>0$. As the concave function with a unique maximum in the interior of the domain has negative derivative (or one-sided derivatives in case that the derivative is not defined) to the right of the maximum and positive derivative to the left of the maximum, the required changes in items d) and i) follow immediately.  Property k) follows from (local) uniqueness of the maximum  by (P11) is Section~\ref{ss.pressuref}.
\qed

\section{One-step $2\times 2$ matrix cocycles: Proof of Theorem~\ref{teo:SL2Rskewproduct}}\label{sec:cocyles}

The goal of this section is to prove Theorem~\ref{teo:SL2Rskewproduct}. 
For that, using the notation in Section~\ref{subsec:cocyles},
for every $\mathbf A\in \mathrm{SL}(2,\bR)^N$
we  study the one-sided cocycle $A\colon\Sigma_N^+\to\mathrm{SL}(2,\bR)$ and consider the associated one-sided step skew-product $F_{\mathbf A}\colon\Sigma_N^+\times\bP^1\to\Sigma_N^+\times\bP^1$ defined in~\eqref{eq:defsteskecoc}. 
Recall that the \emph{Lyapunov exponents of the cocycle $\mathbf A$} at $\xi^+ \in \Sigma_N^+$ are the
limits
\[
\begin{split}
	\lambda_1 (\mathbf A,\xi^+)
	&\eqdef \lim_{n\to \infty} \frac{1}{n}\log \,\lVert \mathbf A^n(\xi^+)\rVert,\\
	\lambda_2 (\mathbf A,\xi^+)
	&\eqdef \lim_{n\to \infty} \frac{1}{n}\log\, \lVert (\mathbf A^n(\xi^+))^{-1}\rVert^{-1},
\end{split}\]
whenever they exist. Otherwise, we denote by $\underline\lambda_1$ and $\overline\lambda_1$ and $\underline\lambda_2$ and $\overline\lambda_2$ the lower and upper limits, respectively.
We analyze the spectrum of Lyapunov exponents of one-step cocycles.
 Note that for every $\xi^+$ we have $\lambda_2(\mathbf A,\xi^+)=-\lambda_1(\mathbf A,\xi^+)$ whenever those exponents are well defined.
The Lyapunov exponent $\lambda_1$ (and hence $\lambda_2$) of $\mathbf A$ are intimately related to the (one-sided) Lyapunov exponent $\chi^+$ of the step skew-product $F_{\mathbf A}$ defined as in~\eqref{def:exponentFfirst} taking only the limit $n\to\infty$, as we will see below. 
Given $\alpha\ge0$, similarly to the level sets in~\eqref{def:levelset} we will analyze the \emph{level sets of Lyapunov exponents} 
\[
	\cL^+_{\mathbf A}(\alpha)
	\eqdef \{\xi^+\in\Sigma_N^+\colon \lambda_1(\mathbf A,\xi^+)=\alpha\}.
\]

The following is our main translation step  from skew-products to cocycles.
\begin{theorem}\label{theoprop:onesidedspectrum}
	For every $\mathbf A\in\mathrm{SL}(2,\bR)^N$  we have the following:
\begin{enumerate}	
	\item
For every $\alpha>0$ it holds
\[\begin{split}
	\{\xi^+\in\Sigma_N^+\colon \lambda_1(\mathbf A, \xi^+) = \alpha\} 
	&= \{\xi^+\in\Sigma_N^+\colon \chi^+(\xi^+,v) 
		= -2\alpha\text{ for some }v\in\bP^1\} \\
	&= \{\xi^+\in\Sigma_N^+\colon \chi^+(\xi^+,v) 
		= 2\alpha\text{ for some }v\in\bP^1\}.
\end{split}\]
\item 
For $\alpha=0$ it holds
\[\begin{split}
	\{\xi^+\in\Sigma_N^+\colon \lambda_1(\mathbf A, \xi^+) = 0\} 
	&\subset \{\xi^+\in\Sigma_N^+\colon \chi^+(\xi^+,v) 
		= 0\text{ for some }v\in\bP^1\} \\
	&\subset \{\xi^+\in\Sigma_N^+\colon \underline{\lambda}_1(\mathbf A, \xi^+) = 0\}
\end{split}\]
and those  three sets have the same topological entropy. 
\end{enumerate}
\end{theorem}
We will prove this theorem in Section~\ref{sssec:entcocy} and conclude the proof of Theorem~\ref{teo:SL2Rskewproduct} in Section~\ref{ss5sec:entcocy}. In Sections~\ref{subsec:pelim}--\ref{sec:relations} we collect  preparatory results. 

\subsection{Preliminary steps}\label{subsec:pelim}

We first collect a series of results (see also Caveat~\ref{cav:convex}) relating the exponents of cocycles and skew-products.

In what follows, we use the standard metric generated by $\theta\mapsto (\cos(\theta\pi),\sin(\theta\pi))$ mapping $\bR/\bZ$ to $\bP^1$, we denote by $\fm$ the corresponding image of the Lebesgue measure.
Let us start from elementary linear algebra: 

\begin{lemma} \label{lem:linalg}
For every $A\in \mathrm{SL}(2,\mathbb{R})$ we have $\lVert A^{-1}\rVert=\lVert A\rVert$ and
\begin{itemize}
\item[i)] $\min_v \lvert f_A'(v)\rvert =||A||^{-2}$ and $\max_v |f_A'(v)| = ||A||^2$,
\item[ii)] for any $\delta>0$ the points $v\in \bP^1$ which satisfy
\[
|f_A'(v)| \leq \frac {(1+\delta^2)||A||^2}{1+\delta^2 ||A||^4}.
\]
form an interval of length $1-\arctan \delta$.
\end{itemize}
\end{lemma}

Let us write 
\[
	f_i\eqdef f_{A_i}
	\quad\text{ for every }i=0,\ldots,N-1.
\]	
Let 
\begin{equation}\label{def:MM}
	M\eqdef\max_{i=0,\ldots,N-1} \lVert A_i\rVert.
\end{equation}

\begin{caveat}\label{cav:convex}
Note that, given $\xi^+$ and $\ell$, unless $f_{[\xi_0\ldots\xi_{\ell-1}]}$ is an isometry, the function $|f_{[\xi_0\ldots\xi_{\ell-1}]}'|$ attains its unique maximum and minimum at some $v_+(\xi^+,\ell)\in\bP^1$ and $v_-(\xi^+,\ell)\in\bP^1$, respectively, and is monotone between them. This hypothesis is explicitly stated in Lemma~\ref{lem:distancevplus} and is a consequence of the 
hypotheses of the other lemmas in Section~\ref{sec:relations} (excluded the first part of Lemma~\ref{prolem:regular}).
\end{caveat}

\begin{lemma}\label{lem:distancevplus}
	For every $\xi^+$ and $\ell\ge1$ so that $f_{[\xi_0\ldots\xi_{\ell-1}]}$  and $f_{[\xi_0\ldots\xi_{\ell}]}$
	are not isometries we have
	\[
	|v_+(\xi^+,\ell) - v_+(\xi^+,\ell +1)| 
	\leq \arctan \left(\frac{1-M^{-4}}{M^{-4}\lVert\mathbf A^{\ell+1}(\xi^+)\rVert^4-1}
	\right)^{1/2}.
	\]
\end{lemma}

\begin{proof}
Note that the hypothesis implies that $v_+(\xi^+,\ell)$ and $v_+(\xi^+,\ell +1)$ are well defined. 
By Lemma~\ref{lem:linalg}\,i)
\[
	M^2\big|f_{[\xi_0\ldots\xi_{\ell}]}'(v_+(\xi^+,\ell))\big| 
	\geq \big|f_{[\xi_0\ldots\xi_{\ell-1}]}'(v_+(\xi^+,\ell))\big|  
	= \lVert \mathbf A^\ell(\xi^+)\rVert^2 
	\ge M^{-2} \lVert \mathbf A^{\ell+1}(\xi^+)\rVert^2 .
\]
Applying Lemma~\ref{lem:linalg}\,ii) to $A=\mathbf A^{\ell+1}(\xi^+)$ and $v=v_+(\xi^+,\ell)$, we
determine $\delta\eqdef |v_+(\xi^+,\ell) - v_+(\xi^+,\ell +1)|$ by solving the inequality in item ii) 
for $\delta$ and obtain
\[
	\delta^2
	\le \frac{1-\lvert f_A'(v)\rvert/\lVert A\rVert^2}{\lvert f_A'(v)\rvert\lVert A\rVert^2-1}
	\le \frac{1-M^{-4}}{M^{-4}\lVert A\rVert^4-1},
\]
where we applied the above estimates. Using Lemma~\ref{lem:linalg}, the definitions of $v_+(\xi^+,\ell)$ and $v_+(\xi^+,\ell+1)$ and the choice of $\delta$ imply the assertion.
\end{proof}

\subsection{Regular one-sided sequences}

Using the notation in Section~\ref{subsec:pelim},  define
\[
	v_0(\xi^+)
	\eqdef \lim_{\ell\to\infty}v_+(\xi^+,\ell)
\]	
whenever this limit exists.
We derive the following result about one-sided sequences $\xi^+\in\Sigma_N^+$ which are \emph{regular} for the cocycle, that is, for which $\lambda_1(\mathbf A,\xi^+)$ is well-defined.
Throughout this section, denote
\begin{equation}\label{eq:convergent}
	a_\ell 
	\eqdef \frac 1 \ell \log \,\lVert\mathbf A^\ell(\xi^+)\rVert.
\end{equation}

\begin{proposition} \label{prolem:regular}
Assume $\xi^+$ satisfies $\lambda_1(\mathbf A, \xi^+) = \alpha$. 
\begin{itemize}
\item If  $\alpha=0$ then $\chi^+(\xi^+,v)=0$ for all $v\in \bP^1$. 
\item If $\alpha>0$ then $v_0(\xi^+)$ is well-defined and we have 
\[
	\chi^+( \xi^+,v_0 (\xi^+))= 2\alpha
	\quad\text{ and }\quad
	\chi^+(\xi^+,v)= -2\alpha\text{ for all }v\ne v_0(\xi^+).
\]	 
\end{itemize}
\end{proposition}

\begin{proof}
The case $\alpha=0$ follows immediately from Lemma \ref{lem:linalg} i). 
So let us assume that $\alpha >0$. 
Let 
$$
\cL(-2\alpha,\xi^+)\eqdef\{v\in\bP^1\colon\chi^+(\xi^+,v)=-2\alpha\}.
$$
Note that, by definition of  $\lambda_1(\mathbf A, \xi^+)$, we have $\alpha = \lim_{\ell\to\infty} a_\ell$.

Let $(\varepsilon_\ell)_\ell$ be a sequence of positive numbers $\varepsilon_\ell<2a_\ell$ converging to $0$  such that
\[
	\sum_{\ell\ge1} \arctan e^{-\ell \varepsilon_\ell} < \infty.
\]
By Lemma \ref{lem:linalg} i), for all $v\in\bP^1$ for every $\ell\ge1$ we have that
\begin{equation}\label{eq:usedlaterr}
	e^{-2\ell a_\ell}
	\le \lvert f_{[\xi_0\ldots\xi_{\ell-1}]}'(v)\rvert.
\end{equation}
By Lemma \ref{lem:linalg} ii), for any $\ell$ there is an interval $I_\ell\subset \bP^1$
of length $1-\arctan e^{-\ell\varepsilon_\ell}$ such that for every point $v\in I_\ell$
\[
	|f_{[\xi_0\ldots\xi_{\ell-1}]}'(v)| 
	\le \frac {(1+e^{-2\ell \varepsilon_\ell}) e^{2\ell a_\ell}}
		{1+e^{2\ell (2a_\ell - \varepsilon_\ell)}} 
\]
(note that the right hand side is approximately $e^{-2\ell (a_\ell-\varepsilon_\ell)}$).
By the Borel-Cantelli Lemma,
almost every $v\in \bP^1$ belongs to infinitely many intervals $I_\ell$. Together with~\eqref{eq:usedlaterr} we   hence have $\chi^+(\xi^+,v)=-2\alpha$.
 Thus,  $\cL(-2\alpha,\xi^+)$ has full Lebesgue  measure. What remains to prove is that $\cL(-2\alpha,\xi^+)$ is the whole set $\bP^1$ minus one point which will turn out to be the point $v_0(\xi^+)$.

Let $v_1, v_2\in\cL(-2\alpha,\xi^+)$. Given $\delta>0$, let  $L\ge1$ be such that for all $\ell>L$ we have that $\lvert a_\ell-\alpha\rvert<\delta$ and for $i=1,2$ (for the second inequality again using~\eqref{eq:usedlaterr})
\begin{equation}\label{eq:formulaaaaa}
	-2(\alpha+\delta)
	\le -2\ell a_\ell
	\le \frac1\ell\log |f_{[\xi_0\ldots\xi_{\ell-1}]}'(v_i)|
	\le -2(\alpha-\delta).
\end{equation}
The points $v_1, v_2$ divide $\bP^1$ into two intervals.  Note that by definition 
 \[
 	\lvert f'_{[\xi_0\ldots\xi_{\ell-1}]}(v_+(\xi^+,\ell))\rvert 
	= \lVert \mathbf A^\ell(\xi^+)\rVert\ge 1
\]	 
and hence $v_+(\xi^+,\ell)\ne v_i$, $i=1,2$. Thus, let us denote by $J_1(\ell)=J_1(\ell,v_1,v_2)$ the interval which contains $v_+(\xi^+,\ell)$ and by 
 $J_2(\ell)= J_2(\ell, v_1,v_2)$ the other one.  
Note that by monotonicity of the derivative, we have
\begin{equation}\label{eq:estmiates}
	|f_{[\xi_0\ldots\,\xi_{\ell-1}]}'(v)| \le e^{-2\ell( \alpha-\delta)}
	\quad\text{ for every }\quad v\in J_2(\ell).
\end{equation}	 

\begin{claim}
\label{cl.J1J2}
For every
 $\ell$ large enough it holds 
 $J_1(\ell) = J_1(\ell+1)$. 
\end{claim} 
 
 \begin{proof}
 By Lemma~\ref{lem:linalg} i) and the second estimate in~\eqref{eq:formulaaaaa} we have
\[
	|f_{[\xi_0\ldots\,\xi_{\ell}]}'(v_+(\xi^+,\ell))| 
	\geq |f_{[\xi_0\ldots\,\xi_{\ell-1}]}'(v_+(\xi^+,\ell))| \cdot M^{-2} 
	\ge e^{2\ell (\alpha-\delta)}\cdot  M^{-2}\ge 1.
\]
By~\eqref{eq:estmiates} the interval $J_1(\ell)$ must contain $v_+(\xi^+,\ell+1)$.
\end{proof}

\begin{claim}\label{cl:claim108}
$v_0(\xi^+)$ is well-defined. 
\end{claim}
\begin{proof}
By contradiction, otherwise $(v_+(\xi^+,\ell))_\ell$ would have at least two accumulation points which would divide $\bP^1$ into two connected components.  Since by the above $\cL(-2\alpha,\xi^+)$ has full measure and hence is dense, there would exist points $u_1,u_2\in \cL(-2\alpha,\xi^+)$, one in each of them. Hence, each of the correspondingly defined intervals $J_1(\ell,u_1,u_2)$ and $J_2(\ell,u_1,u_2)$ would contain one of the accumulation points and hence eventually the accumulating points $v_+(\xi^+,\ell)$, for infinitely many times. This would contradict that $J_1(\ell,u_1,u_2)=J_1(\ell+1,u_1,u_2)$ for all large enough $\ell$ as in Claim~\ref{cl.J1J2}.
\end{proof}

\begin{claim} \label{cl.kat}
For every $v\ne v_0(\xi^+)$ we have $-2\alpha\le\chi^+ (\xi^+,v)\le-2(\alpha-\delta)$.
\end{claim}

\begin{proof}
The first inequality follows from~\eqref{eq:usedlaterr}. For the second, by the above  $\cL(-2\alpha,\xi^+)$ has full Lebesgue measure  and hence is dense in $\bP^1$. Hence, for any $v\in \bP^1 \setminus \{v_0(\xi^+)\}$ we can find $u_1, u_2 \in \cL(-2\alpha, \xi^+)$ such that the points $v$ and $v_0(\xi^+)$ are in different components of $\cL(-2\alpha, \xi^+)\setminus \{u_1, u_2\}$. Thus, $v_0(\xi^+)\in J_1(\ell,u_1, u_2)$ and hence $v\in J_2(\ell,u_1, u_2)$ for $\ell$ large enough. Now~\eqref{eq:estmiates} implies the claim. 
\end{proof}

By Claim~\ref{cl.kat}, as $\delta>0$ was arbitrary, we conclude that for every $v\ne v_0(\xi^+)$ we have $\chi^+ (\xi^+,v)=-2\alpha$.

The only thing that remains to prove is that the  Lyapunov exponent at $v_0(\xi^+)$ is $2\alpha$. For that we invoke
the following lemma whose proof we postpone. 

In what follows we use the Banach-Landau notation%
\footnote{Recall that for two real valued functions $f,g\colon\bR\to\bR$ we have $f=O(g)$ if and only if there exists a positive number $C$ and $x_0$ such that $\lvert f(x)\rvert \le C\lvert g(x)\rvert$ for every $x\ge x_0$.}.

\begin{lemma}\label{lem:landau}
For every $\xi^+$, if $(\ell_i)_i$ is a sequence of positive integers such that 
\begin{itemize}
\item[a)] the limit $\alpha\eqdef \lim_{i\to\infty}\frac{1}{\ell_i}\log\,\lVert \mathbf A^{\ell_i}(\xi^+)\rVert>0$ exists,
\item[b)] the limit $v_0(\xi^+)\eqdef \lim_{i\to\infty}v_+(\xi^+,\ell_i)$ exists, and
\item[c)]  for every $\delta>0$ and every $i$ we have $|v_+(\xi^+,\ell_i) - v_0(\xi^+)|=O(e^{-2\ell_i(\alpha-\delta)})$.
\end{itemize}
Then 
\[
	\lim_{i\to\infty}\frac{1}{\ell_i}\log\,\lvert f_{[\xi_0\ldots\,\xi_{\ell_i-1}]}'(v_0(\xi^+))\rvert =2\alpha.
\]	 
\end{lemma}

Let us show that the hypotheses of   Lemma \ref{lem:landau} are  verified.

Since we consider the case $\lambda_1(\mathbf A,\xi^+)=\alpha>0$, we have hypothesis a) in Lemma \ref{lem:landau}. By Claim \ref{cl:claim108}, $v_0(\xi^+)=\lim_{\ell\to\infty}v_+(\xi^+,\ell)$ is well defined and hence we have hypothesis b) in Lemma \ref{lem:landau}. What remains to verify is hypothesis c) in this lemma.
Observe that $\alpha>0$ implies that for every $\ell$ large enough the map
$f_{[\xi_0\ldots \xi_\ell]}$ is not an isometry. Hence  by Lemma~\ref{lem:distancevplus} with~\eqref{eq:convergent} we get
\[\begin{split}
	|v_+(\xi^+,\ell) - v_+(\xi^+,\ell +1)| 
	&\leq \arctan \left(\frac{1-M^{-4}}{M^{-4}e^{4(\ell+1)a_{\ell+1}}-1}
	\right)^{1/2}\\
	&\leq  \left(\frac{1-M^{-4}}{M^{-4}e^{4(\ell+1)a_{\ell+1}}-1}
	\right)^{1/2} ,
\end{split}\]
where the latter holds for $\ell$ sufficiently large.
Hence  given $\delta>0$ for every $\ell$  large we get (after some simple approximation steps)
\begin{equation}\label{e.todesacocheio}
	|v_+(\xi^+,\ell) - v_+(\xi^+,\ell +1)| 	
	\le e^{-2(\ell+1) (a_{\ell+1}-\delta)}
	\le e^{-2(\ell+1) (\alpha-2\delta)},
	\end{equation}
where we also used that $\alpha=\lim_\ell a_\ell>0$. Thus, for $\ell$ large enough we get 
\[
	|v_0(\xi^+)-v_+(\xi^+,\ell)|
	\le \sum_{k \ge\ell }\lvert v_+(\xi^+,k)-v_+(\xi^+,k+1)\rvert
	\le e^{-2(\ell+1) (\alpha-3\delta)}.
\]
This shows that for any $\delta>0$ we have $\lvert v_0(\xi^+)-v_+(\xi^+,\ell)\rvert=O(e^{-2\ell(\alpha-\delta)})$. Hence we have hypothesis c) in  Lemma \ref{lem:landau}.

We can now apply  Lemma \ref{lem:landau} to get $\chi^+(\xi^+,v_0(\xi^+))=2\alpha$. 

What remains is to give the postponed proof.
\begin{proof}[Proof of Lemma~\ref{lem:landau}]
Clearly, with Lemma~\ref{lem:linalg} i), we get 
\[
	\lim_{i\to\infty}\frac{1}{\ell_i}\log\,\lvert f_{[\xi_0\ldots\,\xi_{\ell_i-1}]}'(v_0)\rvert \le 2\alpha.
\]	 
What remains to show is the other inequality.
Observe that $\alpha>0$ implies that for every $\ell_i$ large enough the map
$f_{[\xi_0\ldots \xi_{\ell_i}]}$ is not an isometry.
Note that by the definition of $v_+(\xi^+,\ell)$ and by Lemma~\ref{lem:linalg} i) we have
\[
	|f_{[\xi_0\ldots\xi_{\ell-1}]}' (v_+(\xi^+,\ell))|
	= \max_{v} |f_{[\xi_0\ldots\xi_{\ell-1}]}'(v)|
	=  \lVert \mathbf A^\ell(\xi^+) \rVert^2 .
\]
Hence, by Lemma \ref{lem:linalg} ii) applied to $A=\mathbf A^\ell(\xi^+)$, for every $v$ with $|v-v_+(\xi^+,\ell)| \le \frac{1}{2}
\arctan \delta^\prime$ we have 
\[
|f_{[\xi_0\ldots\xi_{\ell-1}]}'(v)| \geq 
		\frac{(1+(\delta')^2)   \lVert \mathbf A^{\ell}(\xi^+) \rVert^2}
			{1+(\delta')^2  \lVert \mathbf A^{\ell}(\xi^+) \rVert^4}
\]
(note that the interval in Lemma \ref{lem:linalg} is the complement of an concentric interval centered at $v_+(\xi^+,\ell)$).
Applying the above for $v=v_0$ and $\delta'=\delta_i'=2\tan\delta_i$, where $\delta_i\eqdef |v_+(\xi^+,\ell_i) - v_0|$,  we have 
\[
		|f_{[\xi_0\ldots\xi_{\ell_i-1}]}'(v_0)| \geq 
		\frac{(1+(\delta_i')^2)   \lVert \mathbf A^{\ell_i}(\xi^+) \rVert^2}
			{1+(\delta_i')^2  \lVert \mathbf A^{\ell_i}(\xi^+) \rVert^4}.
\]
By hypothesis, for any $\delta>0$ we have $\delta_i=O(e^{-2\ell_i(\alpha-\delta)})$ and hence 
\[
	(\delta_i')^2
	= O(\delta_i^2)
	= O(\delta_i)^2
	= O(e^{-4\ell_i (\alpha-\delta)}).
\]
Recalling that $\alpha=\lim_i \frac{1}{\ell_i}\log \lVert \mathbf A^{\ell_i}(\xi^+) \rVert>0$, we conclude
\[\begin{split}
	\lim_{i\to\infty}\frac{1}{\ell_i}\log|f_{[\xi_0\ldots\xi_{\ell_i-1}]}'(v_0)| 
	&\geq \lim_{i\to\infty}\frac{1}{\ell_i}\log
		\frac{(1+(\delta_i')^2)   \lVert \mathbf A^{\ell_i}(\xi^+) \rVert^2}
			{1+(\delta_i')^2 \lVert \mathbf A^{\ell_i}(\xi^+) \rVert^4}\\
	&= - \lim_{i\to\infty}\frac{1}{\ell_i}\log
		((\delta_i')^2 \lVert \mathbf A^{\ell_i}(\xi^+) \rVert^2)\\
	&\ge 4(\alpha-\delta)-2\alpha=2\alpha-4\delta.		
\end{split}\]
As $\delta$ was arbitrary, this shows the inequality $\ge$ and hence proves the lemma.
\end{proof}
This finishes the proof of the proposition.
\end{proof}

\subsection{Nonregular one-sided sequences}

We now study one-sided sequences $\xi^+\in\Sigma_N^+$ which are \emph{nonregular} for the cocycle, that is, for which $\underline\lambda_1(\mathbf A,\xi^+)<\overline\lambda_1(\mathbf A,\xi^+)$.

\begin{lemma} \label{lem:irregularfacil}
Assume $\xi^+$ satisfies $\underline{\lambda}_1(\mathbf A, \xi^+) = 0$.
 Then there exists a sequence $(n_i)_i$ such that  for every point $v\in\bP^1$ it holds
\[
	\lim_{i\to\infty}\frac{1}{n_i} \log |f_{[\xi_0\ldots\xi_{n_i-1}]}'(v)| 
	= 0.
\]
\end{lemma}

\begin{proof}
There exists a subsequence  $(n_i)_i$ so that by Lemma~\ref{lem:linalg} i)
and with the notation in \eqref{eq:convergent}
 we have
 $\lim_ia_{n_i}=0$ and 
\[
	-2 a_{n_i}
	\le \frac{1}{n_i}\log\,\lvert f_{[\xi_0\ldots\,\xi_{n_i-1}]}'(v)\rvert
	\le 2 a_{n_i}
	\quad\text{ for every }\quad
	v\in\bP^1.
\]	
Since $\underline\lambda_1(\mathbf A,\xi^+)=0$ this immediately implies the lemma.
\end{proof}

\begin{lemma} \label{lem:irregular}
Assume $\xi^+$ satisfies $\underline{\lambda}_1(\mathbf A, \xi^+) = \alpha_1$, $\overline{\lambda}_1(\mathbf A, \xi^+) = \alpha_2$ for some $0<\alpha_1<\alpha_2$. Then for every $\alpha \in [\alpha_1, \alpha_2]$ there exists a sequence $(m_i)_i$ such that 
\[
	\lim_{i\to\infty}\frac{1}{m_i} \log \,\lVert \mathbf A^{m_i}(\xi^+)\rVert = \alpha.
\]
Moreover, for any such a sequence $(m_i)_i$
there are two cases:
\begin{enumerate}
\item[(a)]either  the limit $v_0\eqdef \lim_i v_+(\xi^+,m_i)$ exists and
\begin{equation} \label{eqn:irrcase1}
	\lim_{i\to\infty}\frac{1}{m_i} \log |f_{[\xi_0\ldots\xi_{m_i-1}]}'(v)| = -2\alpha
	\quad\text{ for all }\quad v\ne v_0,
\end{equation}
and for all $\alpha \in (\alpha_1,\alpha_0)$, where
$$
\alpha_0\eqdef\alpha_1\frac{ \log M +\alpha_2}{ \log M + \alpha_1},
$$ 
 there exists a subsequence $(n_i)_i$ of $(m_i)_i$ such that
\[
	\lim_{i\to\infty}\frac{1}{n_i} \log |f_{[\xi_0\ldots\xi_{n_i-1}]}'(v_0)| 
	= 2\alpha.
\]
\item[(b)] or for all $v\in\bP^1$ there exists a subsequence of $(m_i)_i$ for which~\eqref{eqn:irrcase1} holds.
\end{enumerate}
\end{lemma}

\begin{proof}
Let
$\alpha \in [\alpha_1, \alpha_2]$.
We note the following  simple fact.

\begin{claim}\label{clai:claim}
	Consider $\beta$ and $\gamma$ with $\beta>\gamma$.
	\begin{itemize} 
	\item
	If
	 $a_m\ge\beta$ and $a_{m+k}<\gamma$ then 
	 $\displaystyle{ m+k> m\frac{\log M+\beta}{\log M+\gamma}}$.
	 \vskip 0.1cm
	\item
	If $a_m\le\gamma$ and $a_{m+k}>\beta$
	then 
	$ \displaystyle{ m+k>m\frac{\log M-\gamma}{\log M-\beta}}.$
	\end{itemize}
\end{claim}

\begin{remark}
\label{r.smallsteps}
 Claim~\ref{clai:claim} implies that there exists $(m_i)_i$ so that $\lim_ia_{m_i}= \alpha$.
Indeed, by hypothesis this holds if $\alpha=\alpha_2$. Otherwise, if $\alpha<\alpha_2$ and 
$\delta$ small, it is enough to observe that if $a_m>\alpha_2-\delta\ge\alpha+\delta$ then $a_{m+k}<\alpha-\delta$ only if 
$$
m+k>\frac{m(\log M +(\alpha+\delta))}{(\log M+(\alpha-\delta))},
$$ 
and hence $k\ge 2$ when $m$ is large. 
\end{remark}

Arguing as  in the proof of Proposition \ref{prolem:regular}, we get that the set 
\[
	\cL'(-2\alpha,\xi^+, (m_i)_i)
	\eqdef \Big\{v\in\bP^1\colon
		\lim_{i\to\infty}\frac 1 {m_i} \log |f_{[\xi_0\ldots\xi_{m_i-1}]}'(v)| 
	= -2\alpha\Big\}
\]
has full  Lebesgue measure. Now choose some $v_1, v_2\in \cL'(-2\alpha,\xi^+, (m_i)_i)$ and as in
the proof of  Proposition \ref{prolem:regular} define the intervals $J_1(m_i)=J_1(m_i,v_1,v_2)$ and 
$J_2(m_i,v_1,v_2)$. Note that, as we consider a sequence for which in general $m_{i+1}\ne m_i+1$, we cannot proceed directly to get  $J_1(m_i)=J_1(m_{i+1})$ as in Claim~\ref{cl.J1J2}, however we can argue as follows. 

Either we do have $J_1(m_i)=J_1(m_{i+1})$ for all $i\ge1$ large enough, in which case we  continue exactly as in the proof of Proposition \ref{prolem:regular} and obtain the existence of the limit $v_0\eqdef \lim_iv_+(\xi^+,m_i)$ and that for every $v\ne v_0$ and for this very sequence the assertion~\eqref{eqn:irrcase1} holds true.  This proves the first part of the claim in Case (a).

Or we have $J_1(m_i)=J_1(m_{i+j})$ for infinitely many $j\ge 1$ and $J_2(m_i)=J_1(m_{i+\ell})$ for infinitely many $\ell$. Hence, taking any $v\in\bP^1$, there is some subsequence $(m_{i_k})_k$ such that $v\in J_1(m_{i_k})$ and hence for this subsequence  the assertion~\eqref{eqn:irrcase1} holds true, proving Case (b).

\smallskip
It remains to prove the remaining part of Case (a) where
$\alpha\in (\alpha_1,\alpha_0)$.
For that we will apply Lemma \ref{lem:landau}. Note that hypotheses a) and b) are satisfied by the hypotheses of the Case (a) we consider. It remains to check hypothesis c) of that lemma. 

\begin{claim}\label{cl:fundao}
	There exists a subsequence $(n_i)_i$ of $(m_i)_i$ such that for any sufficiently small $\delta>0$ and for every $i$ we have
\[\lvert v_+(\xi^+,n_i)-v_0\rvert=O(e^{-2n_i(\alpha-\delta)}).\]	
\end{claim}

With this claim at hand all hypotheses in Lemma \ref{lem:landau} are satisfied and hence we have
\[
	\lim_{i\to\infty}\frac{1}{n_i} \log |f_{[\xi_0\ldots\xi_{n_i-1}]}'(v_0)| 
	= 2\alpha
\]
which concludes the proof of the lemma.

\begin{proof}[Proof of Claim \ref{cl:fundao}]
First observe that the definition of $\alpha_0$ implies that $\alpha_1<\alpha_0\le \alpha_2$.
Consider  $\delta>0$ sufficiently small (we will specify this further) such that 
\[
	\delta < \min\{(\alpha_2-\alpha)/2, (\alpha-\alpha_1)/2\}.
\]	
By hypothesis $\overline\lambda_1(\mathbf A,\xi^+)=\alpha_2$, there is a sequence $(r_i)_i$ for which $a_{r_i}\ge \alpha_2-\delta$ for all $i$. To define  the subsequence $(n_i)_i$  of $(m_i)_i$ we  consider  an auxiliary strictly increasing sequence   $(t_i)_i$ 
given by the positive integers such that $a_{t_i}\ge \alpha_2-\delta$ for all $i\ge1$. For every $i\ge1$ let 
 \[
	n_i
	\eqdef \max \Big\{ n \colon 
		n\in \{m_j\}, n<t_i,  a_n < \alpha+\delta\Big\},
\]
that is, $a_{n_i}$ is the last sequence element which was below $\alpha+\delta$ before  approaching values close to $\alpha_2$. 
Let 
\begin{equation}\label{eq:notag}\begin{split}
 	L_i'(\delta)
	&\eqdef  n_i\frac{\log M+(\alpha_2-\delta)}{\log M+(\alpha+\delta)}\cdot
			\frac{\log M-(\alpha+\delta)}{\log M-(\alpha_2-\delta)}, \notag\\
	L_i''(\delta)
	&\eqdef 
	 n_i \frac {\log M + (\alpha-\delta)} {\log M + (\alpha_1+\delta)} \cdot
	\frac{\log M+(\alpha_2-\delta)}{\log M+(\alpha+\delta)}\cdot
			\frac{\log M-(\alpha+\delta)}{\log M-(\alpha_2-\delta)}\notag\\
	&= 		\frac {\log M + (\alpha-\delta)} {\log M + (\alpha_1+\delta)}\cdot L_i'(\delta).
\end{split}\end{equation}	
Clearly $L_i''(\delta)>L_i'(\delta)>n_i$.
Let
$$
\ell_i'(\delta)\eqdef \lceil L_i'(\delta)\rceil
\quad 
\mbox{and}
\quad
\ell_i''(\delta)\eqdef \lceil L_i''(\delta)\rceil.
$$
The above implies that there are an increasing function $\tau(\delta)$, $\tau(\delta) \to 1$ as $\delta\to 0$ such that  (for sufficiently large $i$)
\[
	\ell_i''(\delta)
	> n_i  \, \tau(\delta) \, \frac{\log M +\alpha_2}{\log M +\alpha_1}\, 
				\frac{\log M - \alpha}{\log M -\alpha_2}
	>n_i  \, \tau(\delta) \,\frac{\log M+\alpha_2}{\log M+\alpha_1}			,
\]
where for the last inequality we use that $\alpha_2>\alpha$.
Multiplying both sides by $\alpha_1-\delta$ and recalling the definition of $\alpha_0$, we get
\[
	\ell_i''(\delta)(\alpha_1-\delta)
	> n_i\left(\tau(\delta)\alpha_0-\tau(\delta)\frac{\log M+\alpha_2}{\log M+\alpha_1}\delta\right)	.
\]
Specifying now $\delta$, it will be enough that for given $\alpha$ we have that $\delta>0$ is sufficiently small such that 
\[
	\tau(\delta)\alpha_0-\tau(\delta)\frac{\log M+\alpha_2}{\log M+\alpha_1}\delta
	>\alpha.
\]
Hence, for $i$ large we obtain
\begin{equation}\label{e.dominates}
	n_i (\alpha-\delta) 
	<n_i \alpha 
	<\ell_i''(\delta) (\alpha_1-\delta).
\end{equation}
In what follows, we take $\delta$ in this way and further-on omit the dependence on $\delta$. 

\begin{figure}
\begin{overpic}[scale=.40 
  ]{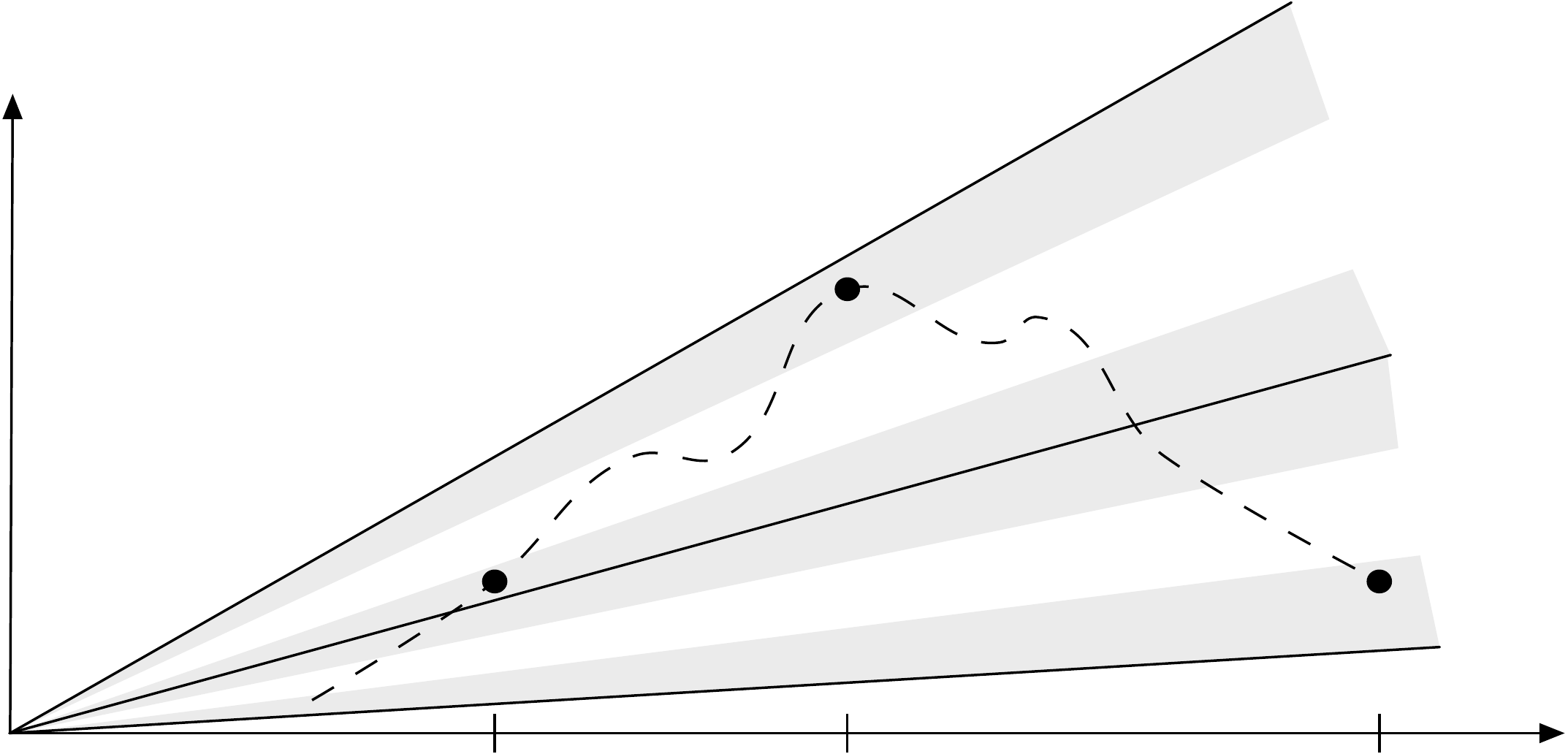}
        \put(84,47.5){\small$n\alpha_2$}
        \put(93.5,6){\small$n\alpha_1$}
        \put(90,25){\small$n\alpha$}
        \put(30,-3){\small$n_i$}
        \put(53,-3){\small$n_i'$}
        \put(87,-3){\small$n_i''$}
\end{overpic}
\caption{}
\label{fig:cones}
\end{figure}

Given $n_i$, let
\[
	n_i'
	\eqdef \min\{n\colon n>t_i,  a_n<\alpha+\delta\}
\]
be the smallest index $n>t_i$ at which $a_n$ drops below $\alpha+\delta$ again and let
\[
	n_i''
	\eqdef \min\{n\colon n>n_i', a_n<\alpha_1+\delta\}
\]
be the first $n>n_i'$ at which $a_n$ approaches the other accumulated exponent $\alpha_1$. Compare Figure~\ref{fig:cones}.
Note that by Claim~\ref{clai:claim} applied twice (to the pairs of times $(n_i,t_i)$ and  $(t_i, n_i')$)  we have $n_i'\ge  L_i'\ge\ell_i'$. In particular, invoking Remark \ref{r.smallsteps}, for any $n=n_i,\ldots,n_i'$ and hence  in particular for $n=\ell_i'$, we have
\[	a_n\ge \alpha-\delta.
\]
Consequently, for any $n\ge\ell_i'+1$ we have
\[
	na_n
	\ge \ell_i'(\alpha-\delta) - (n-\ell_i')\log M.
\]
Putting this together and noting also that we can assume that $a_n>\alpha_1-\delta$ for all $n$, by construction,  we have
\begin{equation}\label{eq:estimmates}\begin{cases}
	a_n\ge\alpha-\delta
	&\text{ for all }n\in I_1\eqdef \{n_i+1,\ldots, \ell_i'\}\\
	n a_n\ge \ell_i'(\alpha-\delta) -(n-\ell_i')\log M
	&\text{ for all }n\in I_2 \eqdef \{\ell_i'+1,\ldots, \ell_i''\}\\
	a_n\ge\alpha_1-\delta
	&\text{ for all }n\in I_3 \eqdef \{ \ell_i''+1, \ldots\}.
\end{cases}\end{equation}

Recall again that  the limit $v_0\eqdef \lim_i v_+(\xi^+,n_i)$ exists since we choose the terms $n_i$ in the sequence $(m_j)_j$ satisfying the hypothesis of Case (a). 
Thus,  applying Lemma~\ref{lem:distancevplus} to a telescoping sum we have (for $n_i$ large enough) 
\[\begin{split}
	|v_+(\xi^+,n_i)&-v_0| 
	\le \sum_{\ell=0}^\infty\lvert v_+(\xi,n_i+\ell)-v_+(\xi,n_i+\ell+1)\rvert\\
	&\le \sum_{\ell=0}^\infty\arctan\left(\frac{1-M^{-4}}
				{M^{-4}\lVert\mathbf A^{n_i+\ell+1}(\xi^+)\rVert^4-1}\right)^{1/2}
	\le \sum_{n= n_i+1}^\infty e^{-2 n (a_n -\delta)},
\end{split}\]
where the last inequality follows after some simple approximation steps as in \eqref{e.todesacocheio}.
To finish the proof of the claim, we  now divide the latter sum into three subsums over the index sets $I_1,
I_2,$ and $I_3$ defined in \eqref{eq:estimmates} and estimate these sums. 

\noindent
$\bullet$ First sum:
\[
	\sum_{n\in I_1} e^{-2 n (a_n -\delta)}	
	\le\sum_{n\ge n_{i}+1 } e^{-2 n \alpha}	
	\le e^{-2 n_i \alpha } \, \frac{ e^{-2\alpha } }{1-e^{-2\alpha}}
	< e^{-2 n_i (\alpha-\delta) } \, \frac{ e^{-2\alpha } }{1-e^{-2\alpha}}.	
\]
$\bullet$ Second sum: with \eqref{eq:estimmates} we have
\[\begin{split}
	&\sum_{n\in I_2} e^{-2 n (a_n -\delta)}	
	\le \sum_{n=\ell_i'+1}^{\ell_i''}e^{-2\ell_i'(\log M+(\alpha-\delta)) +2 n (\log M+ \delta)}\\
	&\le e^{-2\ell_i'(\log M+(\alpha-\delta))}
		\,\frac{e^{2(\ell_i''+1)(\log M+\delta)}-e^{2(\ell_i'+1)(\log M+\delta)}}
			{e^{2( \log M+\delta)}-1}	\\
	(\text{with }\eqref{eq:notag})\quad
	&= e^{-2\ell_i''(\log M+(\alpha_1+\delta))}
	e^{2\ell_i''(\log M+\delta)}
		\,\frac{e^{2(\log M+\delta)}-e^{-2(\ell_i''-\ell_i'-1)(\log M+\delta)}}
			{e^{2( \log M+\delta)}-1}\\
	&\le e^{-2\ell_i''\alpha_1}
		\, \frac{e^{2(\log M+\delta)}}{e^{2(\log M+\delta)}-1}.
\end{split}\]
Note that in the but last step, by a slight abuse of notation, we apply \eqref{eq:notag} for $\ell_i'$ and $\ell_i''$ which asymptotically coincide with $L_i'$ and $L_i''$, respectively. \\
$\bullet$ Third sum: with \eqref{eq:estimmates} we have
\[
	\sum_{n\in I_3}e^{-2na_n}
	\le  \sum_{n\ge \ell_i''+1}e^{-2n(\alpha_1-\delta)} 
	\le e^{-2\ell_i''(\alpha_1-\delta)}
		\frac{e^{-2(\alpha_1-\delta)}}{1-e^{-2(\alpha_1-\delta)}}.
\]

Finally, recalling that \eqref{e.dominates},  we have $n_i(\alpha-\delta)<\ell_i'' (\alpha_1 -\delta) <\ell_i''\alpha_1$ for every sufficiently large $i$. Hence the first sum dominates the second and the third ones and hence, for any $\delta>0$ small enough we have $\lvert v_+(\xi^+,n_i)-v_0\rvert=O(e^{-2n_i(\alpha-\delta)})$. This proves the claim.\end{proof}

The proof of the lemma is now complete.
\end{proof}

\subsection{Relations between exponents of cocycles and skew-products}\label{sec:relations}

We have the following consequences of Lemmas~\ref{lem:irregularfacil} and \ref{lem:irregular}.

\begin{corollary} \label{cor:irregular}
Assume that $\xi^+$ satisfies $\underline{\lambda}_1(\mathbf A, \xi^+) = \alpha_1$, $\overline{\lambda}_1(\mathbf A, \xi^+) = \alpha_2$ for some $\alpha_1<\alpha_2$. 
\begin{enumerate}
\item  If $\alpha_1=0$ then if $\chi^+(\xi^+,v)$  is well defined (i.e., the limit exists) then it is equal to zero.
\item If $\alpha_1>0$ then there is no $v\in \bP^1$ for which the Lyapunov exponent $\chi^+(\xi^+,v)$ is well defined.
\end{enumerate}
\end{corollary}

\begin{proof}
Item 1 follows immediately from Lemma \ref{lem:irregularfacil}. To show Item 2 observe that by Lemma \ref{lem:irregular} for any $\alpha\in(\alpha_1,\alpha_0)$ and for every $v\in \bP^1$ there is a sequence $(m_i)_i$   for which 
\[
	\lim_{i\to\infty} \frac 1 {m_i} \log |f_{[\xi_0\ldots\xi_{m_i-1}]}'(v)| 
	\in\{-2\alpha,2\alpha\}.
\]
This concludes the proof.
\end{proof}

\begin{corollary}\label{cor:limexireg}
	Assume that $(\xi^+,v)\in\Sigma_N^+\times\bP^1$ is such that $\chi^+(\xi^+,v)$ is well defined and nonzero, then $\lambda_1(\mathbf A,\xi^+)=2\lvert\chi^+(\xi^+,v)\rvert$.
\end{corollary}
\begin{proof}
By Corollary \ref{cor:irregular} we obtain $\underline\lambda_1(\mathbf A,\xi^+)=\overline\lambda_1(\mathbf A,\xi^+)=\lambda_1(\mathbf A,\xi^+)>0$  and by Proposition~\ref{prolem:regular} $\lambda_1(\mathbf A,\xi^+)=2\lvert\chi^+(\xi^+,v)\rvert$.
\end{proof}

\subsection{Entropy spectrum: Proof of Theorem~\ref{theoprop:onesidedspectrum}}\label{sssec:entcocy}

We finally study the topological entropy of several level sets. 

\begin{proposition} \label{prolem:exp0}
The sets 
\[
	S_0 
	= \{\xi^+\in\Sigma_N^+\colon \lambda_1(\mathbf A,\xi^+)=0\}
	\quad\text{and}\quad
	S_1 = \{\xi^+\in\Sigma_N^+\colon \underline{\lambda}_1(\mathbf A,\xi^+)=0\}
\]
have the same topological entropy.
\end{proposition}

\begin{proof}
We clearly have $S_0\subset S_1$ and hence $h_{\rm top}(S_0)\le h_{\rm top}(S_1)$. 

It only  remains to prove the other inequality, for which we will invoke again Frostman's Lemma~\ref{l.frostman}. 
Let $h=h_{\rm top}(S_1)$.
For any $\varepsilon>0$ consider the sets
\[
	X_{n, \varepsilon} 
	\eqdef \left\{\xi^+\in\Sigma_N^+\colon 
	\left\lvert\frac 1n \log\, \lVert A_{\xi_{n-1}}\circ\ldots\circ A_{\xi_0}\rVert \right\rvert\leq \varepsilon\right\}.
\]
Note that for any $N\ge1$ we obtain the corresponding cover by open (cylinder) sets
\[
	S_1
	\subset \bigcup_{n\ge N}
				\bigcup_{\xi^+\in X_{n,\varepsilon}}[\xi_0\ldots\xi_{n-1}].
\]

Recalling Appendix \ref{App:B}, fix $\mathscr{A}=\{[i]\colon i=0,\ldots,N-1\}$ the cover by cylinders of level $1$. Note that there exists $N=N(\varepsilon)\ge1$ such that for any open cover $\mathcal U$ of $S_1$, which satisfies $n_{\mathscr A}(U)\ge N$ for every $U\in\mathcal U$, we have 
\[
	\sum_{U\in\mathcal U}e^{-(h-\varepsilon)n_{\mathscr A}(U)}
	>C(\varepsilon)\eqdef\frac{1}{1-e^{-\varepsilon}}.
\]
Note that $n_{\mathscr A}(U)$ for a cylinder set $U=[\xi_0\ldots\xi_{n-1}]$ is just its length $n$. 

\begin{claim}
There exist $n\ge N$ and $\widehat{\mathcal U}=\{[\xi_0\ldots\xi_{n-1}]\colon \xi^+\in X_{n,\varepsilon}\}$ a family of cylinders of \emph{equal} length $n\ge N$ each one intersecting $S_1$ such that 
\[
	\card(\widehat{\mathcal U})
	>e^{(h-2\varepsilon)n}.
\]
\end{claim}

\begin{proof}
By the above, there exists a cover $\mathcal U$ of $S_1$ by cylinders 
\[
	\mathcal U
	= \{ [\xi_0\ldots\xi_{n-1}]\colon n\ge N,\xi^+\in X_{n,\varepsilon}\}
\]
such that
\[
	\sum_{U\in\mathcal U}e^{-(h-\varepsilon)n_{\mathscr A}(U)}
	>C(\varepsilon).
\]
We will show that we can choose $\widehat{\mathcal U}$ being a subfamily of $\mathcal U$.
Indeed, by contradiction, suppose that for every $n\ge N$, denoting by $\mathcal U(n)\subset\mathcal U$ the subfamily of cylinders of length $n$, we would have $\card(\mathcal U(n))\le e^{(h-2\varepsilon)n}$.  This would imply that
\[
	\sum_{U\in\mathcal U}e^{-(h-\varepsilon)n_{\mathscr A}(U)}
	=\sum_{n\ge N}\sum_{U\in\mathcal U(n)}e^{-(h-\varepsilon)n}
	\le \sum_{n\ge N}e^{(h-2\varepsilon)n}e^{-(h-\varepsilon)n}
	<C(\varepsilon),	
\]
contradiction.
\end{proof}

Now we take a sequence $(\varepsilon_i)_i$ decreasing to zero and apply the above to each $\varepsilon_i$. This provides a sequence $n_i=n(\varepsilon_i)\ge1$ and families $\widehat{\mathcal U}_i$ of cylinder sets of \emph{equal} length $n_i$ each satisfying
\[
	\card(\widehat{\mathcal U}_i)
	>e^{(h-2\varepsilon_i)n_i}.
\]
Note that by the standard property of cylinders,  the elements in $\widehat{\mathcal U}_i$ are pairwise disjoint.

Now, given $m\ge1$, for each $i$ we consider the family of cylinders of length $m\,n_i$ which are formed by all possible cylinders which are $m$ concatenated elements from the family $\widehat{\mathcal U}_i$, denote this family by $\widehat{\mathcal U}_i^{m}$. Again,  this is a family of pairwise disjoint cylinders.

Choose now a fast growing sequence $(m_i)_i$ satisfying
\begin{equation} \label{eqn:growthmi}
	\lim_{k\to\infty}\frac {n_{k+1}} {\sum_{i=1}^k m_i n_i} 
	=0
	\quad\text{ and }\quad
	\max_{i=1,\ldots,k-1}\frac{m_in_i}{m_{k}n_{k}}<\frac{1}{k^2}.
\end{equation}
Let $X$ be the set of one-sided infinite sequences of the form
\[
	X
	= \{\xi^+=\varrho_1\varrho_2\ldots\colon \varrho_i\in\widehat{\mathcal U}_i^{m_i}, i=1,2,\ldots\}.
\]

\begin{claim}
$X\subset S_0$.
\end{claim}

\begin{proof}
Each $\ell\ge1$ we can write as $\ell=\sum_{i=1}^k m_i n_i + j n_{k+1}+r$  for some $j\in\{0,\ldots,  m_{k+1}-1\}$ and $r\in\{0,\ldots,n_{k+1}-1\}$.  Hence, recalling~\eqref{def:MM}, from~\eqref{eqn:growthmi} we obtain
\[\begin{split}
	\left\lvert\frac 1 \ell \log \,\lVert \mathbf A^{\ell}(\xi^+)\rVert\right\rvert
	&\leq \frac {\sum_{i=1}^k m_i n_i \varepsilon_i + j n_{k+1} \varepsilon_{k+1}+r\log M} 						{\sum_{i=1}^k m_i n_i + j n_{k+1}+r}\\
	&\le k\frac{1}{k^2}+\varepsilon_{k+1}+\frac{n_{k+1}}{\sum_{i=1}^km_in_i}\log M \to 0,
\end{split}\]
as $\ell\to\infty$ (and hence $k\to\infty$).
\end{proof}

\begin{claim}
	$h_{\rm top}(X)\ge h$.
\end{claim}

\begin{proof}
The construction of the set $X$ can be described as the intersection of an infinite nested family $X(\ell)$, each being a finite union of cylinders. For any $\ell_k \eqdef \sum_{i=1}^k m_i n_i$ denote by $X(\ell_k)$ the union of all $\ell_k$th level cylinders which intersect $X$. By construction of $X$, each cylinder in $X(\ell_k)$ contains at least $e^{(h-2\varepsilon_{k+1})n_{k+1}}$ cylinders from $X(\ell_{k+1})$.

We will equidistribute on $X$ a probability measure $\nu$, estimate its local dimension and apply Frostman's Lemma~\ref{l.frostman}. For every cylinder in $X(\ell_k)$ the measure $\nu$ is to be equidistributed on its subcylinder from $X(\ell_{k+1})$. 
Denote by $\Delta^+_\ell(\xi^+)$ the cylinder of length $\ell$ containing $\xi^+$.
By induction, we can prove that for every $\ell_k$th level cylinder $C \in X(\ell_k)$ we have
\[
	\nu(C) 
	\leq e^{-\sum_{i=1}^k m_i n_i (h-2\varepsilon_i) }.
\]
Hence, for every $\xi^+\in X$ we have
\[
	\liminf_{k\to\infty} - \frac{1}{\ell_k} \log \nu(\Delta^+_{\ell_k}(\xi^+)) 
	\geq \liminf_{k\to\infty} 
		\frac {\sum_{i=1}^k m_i n_i (h-2\varepsilon_i)} 
				{\sum_1^k m_i n_i} = h,
\]
where we used \eqref{eqn:growthmi}. Now Frostman's Lemma~\ref{l.frostman} implies  $h_{\rm top}(X) \geq h$. Since  $S_0\supset X$ and entropy is monotone, this proves the claim.
\end{proof}
This proves $h_{\rm top}(S_0) \ge h_{\rm top}(S_1)$ and finishes the proof of the proposition.
\end{proof}

We can now conclude the Proof of Theorem~\ref{theoprop:onesidedspectrum}.

\begin{proof}[Proof of Theorem~\ref{theoprop:onesidedspectrum}]
Consider first the case of $\alpha>0$. By Proposition \ref{prolem:regular} we have
\[
	\{\xi^+\in\Sigma_N^+\colon \lambda_1(\mathbf A, \xi^+) = \alpha\} 
	\subset \{\xi^+\in\Sigma_N^+\colon \chi^+(\xi^+,v) 
		= 2\alpha\text{ for some }v\in\bP^1\}.
\]
To obtain the other inclusion, consider $\xi^+\in\Sigma_N^+$ such that there is a Lyapunov regular point $(\xi^+,v)$ for some $v\in\bP^1$ with exponent $\chi^+(\xi^+,v)=2\alpha$ in the fiber. By  Corollary \ref{cor:limexireg}, $\lambda_1(\mathbf A,\xi^+)$ is well defined.
Again applying Proposition \ref{prolem:regular}, we obtain that $\alpha= \lambda_1(\mathbf A,\xi^+)$, proving the other inclusion.
The case $-2\alpha$ is analogous and hence omitted.  \\
Now consider the case $\alpha=0$. Again by Proposition \ref{prolem:regular} we have
\[
	\{\xi^+\in\Sigma_N^+\colon \lambda_1(\mathbf A, \xi^+) = 0\} 
	\subset \{\xi^+\in\Sigma_N^+\colon \chi^+(\xi^+,v) 
		= 0\text{ for all }v\in\bP^1\}.
\]
To show the second inclusion, assume that $\chi^+(\xi^+,v)=0$ for all $v$. Then either $\lambda_1(\mathbf A,\xi^+)$ exists and hence by the first claim in Proposition~\ref{prolem:regular} must be equal to $0$. Or $\alpha_1=\underline\lambda_1(\mathbf A,\xi^+)<\overline\lambda_1(\mathbf A,\xi^+)$ and then by Case (b) in
Lemma~\ref{lem:irregular} we can exclude that $\alpha_1>0$, hence proving $\underline\lambda_1(\mathbf A,\xi^+)=0$ and thus the other inclusion. The assertion about the entropy is just Proposition~\ref{prolem:exp0}.
This proves the theorem.
\end{proof}

\subsection{Entropy spectrum: Proof of Theorem~\ref{teo:SL2Rskewproduct}}\label{ss5sec:entcocy}

Now we are ready to prove Theorem~\ref{teo:SL2Rskewproduct}. Note the differences in some statements in which we jump from studying the one-sided shift space $\Sigma_N^+$ to the two-sided one $\Sigma_N$.

\begin{proof}[Proof of Theorem~\ref{teo:SL2Rskewproduct}]

After Theorem~\ref{theoprop:onesidedspectrum}, it remains to see the properties of the level sets of two-sided sequences. 
Note that by Corollary~\ref{cor:limexireg} for every $\alpha\ne0$ we have
\[
\pi (\cL^+(\alpha)) \subset \cL^+_{\mathbf A}(\lvert \alpha\rvert/2),
\]
where 
\[
	\cL^+(\alpha)
	\eqdef \{(\xi^+,v)\in\Sigma_N^+\times\bS^1\colon\chi^+(\xi^+,v)=\alpha\}
\]
and  $\pi$ denotes the natural projection. By Proposition~\ref{prolem:regular} for every $\alpha$ 
\[
 \cL^+_{\mathbf A}(\lvert\alpha\rvert/2) \subset 
 \pi (\cL^+(\alpha)). 
\]
Hence, for every $\alpha> 0$ we have
\[
	\pi( \cL^+(-\alpha))=
	\cL^+_{\mathbf A}(\alpha/2)=
	\pi(\cL^+(\alpha)).
\]
Finally, recalling that for every set $\Theta\subset \Sigma_N \times \bP^1$ we have  $h_{\rm top} (F_{\mathbf{A}}, \Theta)=h_{\rm top} (\sigma, \pi(\Theta))$, see for instance \cite[Lemma 4.9]{DiaGelRam:}, for every $\alpha>0$ we have 
\[
	h_{\rm top}(\cL^+(-\alpha))
	= h_{\rm top}(\cL^+_{\mathbf A}(\alpha/2))
	= h_{\rm top}(\cL^+(\alpha)).
\]
By the same argument, applying Theorem~\ref{theoprop:onesidedspectrum} item 2., we obtain 
\[
	h_{\rm top}(\cL^+_{\mathbf A}(0))
	= h_{\rm top}(\pi(\cL^+(0)))
	= h_{\rm top}(\cL^+(0)).
\]

We now relate the one-sided and the two-sided spectra.
Note that any $(\xi,v)$ with $\chi(\xi,v)=\alpha$ satisfies $\chi^+(\xi^+,v)=\alpha$. This immediately implies that for all $\alpha$ we have $h_{\rm top}(\cL(\alpha))\le h_{\rm top}(\cL^+(\alpha))$.

\begin{claim} \label{cl:parte1} 
For every $\alpha\ge0$ it holds
$h_{\rm top}(\cL^+_{\mathbf{A}}(\alpha/2))
	\le h_{\rm top}(\cL(\alpha))$.	
\end{claim}

\begin{proof}
The case $\alpha=0$ follows immediately from Proposition~\ref{prolem:regular}.

Consider now the case $\alpha>0$.
Again by Proposition~\ref{prolem:regular}, for every $\xi^+\in\cL^+_{\mathbf{A}}(\alpha/2)$ the vector $v_0(\xi^+)\in\bP^1$ is well defined and  $\chi^+(\xi^+,v_0(\xi^+))=\alpha$. 
We need to prove the following statement: for every $\xi^+ \in \cL^+_{\mathbf A}(\alpha/2)$ there exists $\eta^-$ such that for $\xi=\eta^-.\xi^+$ we have $\chi(\xi, v_0(\xi^+))=\alpha$. This will imply that $h_{\rm top}(\cL^+_{\mathbf{A}}(\alpha/2))
	\le h_{\rm top}(\cL(\alpha))$.

Let us first consider  the special case when the vector $v_0(\xi^+)$ is simultaneously fixed by all the maps $f_i$. In this case, this vector is also fixed by all the maps $f_i^{-1}$, and we can choose $\xi=\xi^+.\xi^+=(\ldots\xi_1\xi_0.\xi_0\xi_1\ldots)$. Indeed,
\[
	(f_\xi^{-n})'(v_0(\xi^+)) 
	= \prod_{i=0}^{n-1} (f_{\xi_i}^{-1})'(v_0(\xi^+)) 
	= \prod_{i=0}^{n-1} (f_{\xi_i}'(v_0(\xi^+)))^{-1} 
	= (f_\xi^n)'(v_0(\xi^+))^{-1}
\]
and hence $\chi^+(\xi^+, v_0(\xi^+))=\alpha$ implies $\chi(\xi, v_0(\xi^+))=\alpha$.

Assume now that there exists some $f_i$ such that $f_i(v_0(\xi^+))\neq v_0(\xi^+)$. Since we consider $\mathrm{SL}(2,\bR)$  cocycles, the Lyapunov spectra for  the cocycle generated by $\mathbf A$ and the one for  the cocycle generated by $\mathbf A^{-1}=\{A_i^{-1}\colon A_i\in\mathbf A\}$ coincide. Hence there is $\eta^+\in\cL^+_{\mathbf A^{-1}}(\alpha/2)$.
We now apply Proposition~\ref{prolem:regular} to $\mathbf A^{-1}$.  Either $v_0(\xi^+)=v_0(\xi^+,\mathbf A)$ does not coincide with the vector $v_0(\eta^+,\mathbf A^{-1})$ defined with respect to $\mathbf A^{-1}$ and $\eta^+$ and hence $\chi^+(\eta^+,v_0(\xi^+,F_{\mathbf A^{-1}}))=-\alpha$. In this case, we have $\chi(\xi,v_0(\xi^+))=\alpha$ for the concatenated two-sided sequences $\xi=\eta^+.\xi^+=(\ldots \eta_1\eta_0.\xi_0\xi_1\ldots)$. Or, $v_0(\xi^+,\mathbf A)= v_0(\eta^+,\mathbf A^{-1})$.  
Then for some map $f_i$, $ v\eqdef f_i(v_0(\xi^+))\ne v_0(\xi^+)$, in which case we have $\chi(\xi,v)=\alpha$ with $\xi=(\ldots\eta_1\eta_0 i.\xi_0\xi_1\ldots)$, ending the proof of the claim.	
\end{proof}

The symmetry of the entropy spectrum of the skew-product  now follows from the next claim.

\begin{claim}
\label{cl:parte2} 
For every $\alpha>0$ it holds
 $h_{\rm top}(\cL(\alpha))=h_{\rm top}(\cL(-\alpha))$. 
 \end{claim}
 
 \begin{proof}
 The proof of this claim is analogous to the previous one.
 Let $(\xi^+,v)$ such $\chi^+(\xi^+,v)=-\alpha$. By Corollary~\ref{cor:limexireg} we have $\lambda_1(\mathbf A,\xi^+)=2\alpha$. Now by Proposition~\ref{prolem:regular}, we in fact have  $\chi^+(\xi^+,w)=-\alpha$ for all $w\ne v_0(\xi^+)$, the latter being well defined. Again by Proposition~\ref{prolem:regular} now applied to $\mathbf A^{-1}$  there is $\eta^+\in\cL^+_{\mathbf A^{-1}}(\alpha/2)$ so that: either $v_0(\xi^+)=v_0(\xi^+,\mathbf A)$ does not coincide with the vector $v_0(\eta^+,\mathbf A^{-1})$ in which case we would have $\chi(\xi,v_0(\eta^+,\mathbf A^{-1}))=-\alpha$, where $\xi=(\ldots \eta_1\eta_0.\xi_0\xi_1\ldots)$. Or $v_0(\xi^+)=v_0(\xi^+,\mathbf A)=v_0(\eta^+,\mathbf A^{-1})$. In this case, if $f_j(v_0(\eta^+,\mathbf A^{-1}))= v_0(\xi^+)$ for every $j$ then we argue as above and let $\xi=(\ldots \xi_1\xi_0.\xi_0\xi_1\ldots)$. Otherwise,  there is  some map $f_j$ such that $w=f_j(v_0(\eta^+,\mathbf A^{-1}))\ne v_0(\xi^+)$, in which case we have $\chi(\xi,w)=\alpha$ for $\xi=(\ldots\eta_1\eta_0 j.\xi_0\xi_1\ldots)$. 
This proves $h_{\rm top}(\cL^+_{\mathbf A}(\alpha/2))\le h_{\rm top}(\cL(-\alpha))$.
\end{proof}
The proof of the theorem is now complete.
\end{proof}

\appendix
\section{The set $\fE_{N,\rm shyp}$ of elliptic cocycles}\label{App:A}

In this section, we define $\fE_{N,\rm shyp}$ and prove that this set is open and dense (in $\fE_N$). 

First observe that, according to Remark \ref{rem:example}, given $F\colon\Sigma_2\times\bS^1\to\Sigma_2\times\bS^1$ with fiber maps $f_0,f_1$ which satisfies Axioms CEC$\pm$ and Acc$\pm$ there exists $\varepsilon=\varepsilon(f_0,f_1)>0$ such that every skew-product $G\colon\Sigma_2\times\bS^1\to\Sigma_2\times\bS^1$ with fiber maps $g_0,g_1$ which are $\varepsilon$-close to $f_0,f_1$, respectively, also satisfies those axioms.   Also observe that for every $N>2$, every skew-product $H\colon\Sigma_N\times\bS^1\to\Sigma_N\times\bS^1$ with fiber maps $h_0,\ldots,h_{N-1}$ such that $h_0=f_0$ and $h_1=f_1$ also satisfies the axioms.

Second, take $F$ such that $f_0$ is a Morse-Smale diffeomorphism with exactly two fixed points (a global attractor and a global repeller) and $f_1$ is an irrational rotation. By Remark \ref{r.Gorod} the skew-product $F$ satisfies the axioms and hence we can define an $\varepsilon(f_0,f_1)$ as above.

Let us denote by $\fH_1\subset \mathrm{SL}(2,\bR)$ the subset of hyperbolic matrices and by $\fI_1\subset \mathrm{SL}(2,\bR)$ the one of ``irrational rotations"
\[
	\fI_1
	\eqdef \left\{\left(\begin{matrix}
		\cos2\pi\theta&\sin2\pi\theta\\-\sin2\pi\theta&\cos2\pi\theta\end{matrix}
				\right)\colon\theta\in[0,1),\theta\not\in\bQ\right\}.
\]	
Note that if $A\in\fH_1$ and $B\in\fI_1$, then with $\mathbf A=\{ A,B\}$ the skew-product $F_{\mathbf A}$  satisfies the axioms and we can define an $\varepsilon(f_A,f_B)$ as above (recall \eqref{eq:defsteskecoc}).
Now let
\[\begin{split}
	\fE_{N,\rm shyp}
	&\eqdef 
	\bigcup_{A\in\fH_1,B\in\fI_1}
	\Big\{ \mathbf A\in \fE_N\colon
		\text{there exist }A',B'\in\langle \mathbf A\rangle,C\in\mathrm{SL}(2,\bR)
		\text{ so that }\\
		&\phantom{\bigcup_{A\in\fH_1,B\in\fI_1}}
			 f_{C^{-1}A'C},f_{C^{-1}B'C}\text{ are }\varepsilon(f_A,f_B)\text{-close to }f_A,f_B\text{, respectively}
		\Big\} .
\end{split}\]
  
Note that there is a natural identification of $\mathrm{SL}(2,\bR)^N$ with a subset of $\bR^{4N}$. 

\begin{proposition}\label{p.elliptic}
The set $\fE_{N,\rm shyp}\subset\fE_N$ is open and dense (in $\fE_N$). Moreover,  for every  one-step $2\times 2$ matrix cocycle $\mathbf A\in\fE_{N,\rm shyp}$ its induced step skew-product $F_{\mathbf A}$ satisfies 
Axioms CEC$\pm$ and Acc$\pm$ and is proximal.
\end{proposition}

Certainly the existence of a dense (and automatically open) subset of $\fE_N$ of cocycles with a hyperbolic element is well know, but let us sketch  a proof for completeness. 

\begin{proof}
By definition, the set $\fE_{N,\rm shyp}$ is open. It remains to show its density (in $\fE_N$). For that we consider the subset of $\fE_{N,\rm shyp}$,
\[
\begin{split}
	\fE_{N,\rm shyp}'
	&\eqdef 
	\bigcup_{A\in\fH_1,B\in\fI_1}
\Big\{ \mathbf A\in \fE_N\colon
		\text{there exist }A',B'\in\langle \mathbf A\rangle,C\in\mathrm{SL}(2,\bR)
		\text{ so that }\\
	 &\phantom{\bigcup_{A\in\fH_1,B\in\fI_1}}		 C^{-1}A'C=A ,C^{-1}B'C=B
		\Big\} .
		\end{split}\]

\begin{lemma}
	$\fE_{N,\rm shyp}'$ is dense in $	\fE_{N}$.
\end{lemma}

\begin{proof}
Let $\mathbf A\in\fE_N$, that is, assume that $\langle\mathbf{A}\rangle$ contains  elliptic matrices. Note that if $A\in \mathrm{SL}(2,\bR)$ is hyperbolic then for every $C\in \mathrm{SL}(2,\bR)$ we have that $C^{-1}A C$ is also hyperbolic. 

\begin{claim}
There is an arbitrarily small perturbation of 
 $\mathbf A$ containing a hyperbolic element.
 \end{claim}

\begin{proof}
Recall that $A\in \mathrm{SL}(2,\bR)^N$ is called \emph{parabolic} if $\lvert\trace A\rvert=2$. Notice that every parabolic $A$ can be arbitrarily approximated by  hyperbolic ones.
There are three cases:
(1) $\mathbf A$ contains a hyperbolic element;
(2)$\mathbf A$ contains a parabolic element; and 
(3)  $\mathbf A$ contains only elliptic elements.
Observe that in the first two cases we are done. We now consider the third case.

Pick $A_1,A_2\in\mathbf A$ and the vector $v=(1,0)$. Possibly after a small perturbation, we can assume that $A_1$ has irrational rotation number and hence, after possibly a new perturbation,  we can assume that
$A(v)=v$, where $A=A_1^k \circ A_2$ for some (large) $k$. If $v$ is a hyperbolic fixed point for $f_A$ we are done. Otherwise we
consider the perturbations of $A_2$ given by
$$
A_{2,t}= A_2 \circ \left(\begin{matrix}1+t&0\\0&1/(1+t)\end{matrix}\right)
 $$
 and observe that  $A_t(v)=v$, where $A_t= A_1^k \circ A_{2,t}$  and that $v$ is necessarily hyperbolic for $f_{A_t}$.
Recalling that $\fE_N$ is open, we can assume that this perturbation was sufficiently small such that $\mathbf A_t=\{A_1,A_{2,t},A_3,\ldots,A_N\}$ is elliptic and that $\langle \mathbf A_t\rangle$ contains a hyperbolic element. 
\end{proof}

Observe now that the above achieved hyperbolic elements will not be destroyed by sufficiently small further perturbations. It remains to obtain one further perturbation to get an element which is matrix-conjugate to a irrational rotation.

Consider an  elliptic matrix  $A'=A_{i_k} \circ \cdots  \circ A_{i_1}$. If its rotation number is  already irrational we are done. Otherwise we consider rotation matrices $R_r$, small $r\ge 0$, the elliptic matrix 
$
A_r' = A_{i_k} \circ R_r \circ \cdots  \circ A_{i_1}\circ R_r,
$
and the map $F(r) = \trace (A_r)= 2\cos (\varrho (A_r))$. 
By \cite[Lemma A.4]{AviBocYoc:10}, $F'(0)>0$.
This immediately implies that  there are  arbitrarily small perturbations of $\mathbf{A}$ with 
irrational rotation number. This proves the density of $\fE_{N,\rm shyp}'$ (in $\fE_N$) and finishes the proof of the lemma.
 \end{proof}

Note that the second part of the proposition just rephrases
Remark~\ref{rem:example},  that asserts that for every $\mathbf{A}$ in 
$\mathcal{E}_{N, \mathrm{shyp}}$ the corresponding skew-product map $F_{\mathbf{A}}$
satisfies the axioms and proximality.
\end{proof}

\section{Entropy}\label{App:B}

Let $X$ be a compact metric space. Consider a continuous map $f\colon X\to X$, a set $Y\subset X$,  and a finite open cover $\mathscr{A} = \{A_1, A_2,\ldots, A_n\}$ of $X$. Given $U\subset X$ we write $U \prec \mathscr{A}$ if there is an index $j$ so that $U\subset A_j$, and $U\nprec\mathscr{A}$ otherwise.
Taking $U\subset X$ we define
$$
	n_{f,\mathscr{A}}(U) :=
		\begin{cases}
		0&\text{ if } U \nprec \mathscr{A},\\
		\infty &\text{ if } f^k(U)\prec \mathscr{A}\text{ for all } k\in\mathbb{N},\\
		 \ell&\text{ if }  f^k(U)\prec \mathscr{A}\text{ for all }  k\in \{0, \dots,  \ell-1\}
		 	\text{ and }f^\ell(U)\nprec\mathscr{A}.
		\end{cases}
$$
If $\mathcal U$ is a countable collection of open sets, given $d>0$ let
\[
	 m(\mathscr A,d,\mathcal U)
	:= \sum_{U\in\mathcal U}e^{-d \,n_{f,\mathscr{A}}(U)}.
\]
Given a set $Y\subset X$, let
$$
	m_{\mathscr{A}, d} (Y)
	:= \lim_{\rho \to 0}\inf \Big\{m(\mathscr A,d,\mathcal U)\colon
		Y \subset\bigcup_{U\in\mathcal U} U, e^{-n_{f,\mathcal{A}}(U)}<\rho
		\text{ for every } U\in\mathcal U
	\Big\}.
$$
Analogously to what happens for the Hausdorff measure, $d\mapsto m_{\mathcal{A},d}(Y)$
jumps from $\infty$ to $0$ at a unique critical point and we define
$$
  h_{\mathscr{A}}(f,Y)
	:= \inf\{d\colon m_{\mathscr{A}, d}(Y)=0\}
   = \sup\{d\colon m_{\mathscr{A}, d}(Y)=\infty\}.
$$
The \emph{topological entropy} of $f$ on the set $Y$ is defined by
$$
	h_{\rm top}(f,Y)
	:= \sup_{\mathscr{A}} h_{\mathscr{A}}(f,Y) ,
$$
When $Y=X$, we simply write $h_{\rm top}(X) = h_{\rm top}(f,X)$. 

By~\cite[Proposition 1]{Bow:73},  in the case of $Y$ compact this definition is equivalent to the canonical definition of topological entropy (see, for example, \cite[Chapter 7]{Wal:82}).

\bibliographystyle{alpha}
\bibliography{bib}
\end{document}